\documentclass[aap,noinfoline]{imsart}
\RequirePackage[OT1]{fontenc}
\RequirePackage{amsthm,amsmath}
\RequirePackage[numbers]{natbib}
\RequirePackage[colorlinks,citecolor=blue,urlcolor=blue]{hyperref}
\usepackage{amssymb}
\usepackage{amsmath}
\usepackage{amsthm}
\usepackage{amsfonts}
\usepackage{mathtools}
\usepackage{relsize}
\usepackage{xcolor}
\usepackage{enumitem}

\startlocaldefs
\numberwithin{equation}{section}
\theoremstyle{plain}
\newtheorem{thm}{Theorem}[section]
\newtheorem{corollary}[thm]{Corollary}
\newtheorem{lemma}[thm]{Lemma}
\newtheorem*{heuristic}{Heuristic}
\newtheorem{definition}{Definition}[section]
\newtheorem{example}[thm]{Example}
\newtheorem*{definition*}{Definition}
\newtheorem{proposition}{Proposition}
\newtheorem{remark}[thm]{Remark}
\newtheorem{formula}{Formula}[section]

\newcommand\equalsquestion{\stackrel{\mathclap{\normalfont\mbox{?}}}{=}}

\endlocaldefs
\DeclareMathOperator*{\esssup}{ess\,sup}

\begin{document}

\begin{frontmatter}
\title{Pathwise Convergence of the Hard Spheres Kac Process}
\runtitle{Pathwise Convergence of the Kac Process}

\begin{aug}
\author{\fnms{Daniel} \snm{Heydecker}\thanksref{t1}\ead[label=e1]{dh489@cam.ac.uk}}.

\thankstext{t1}{This work was supported by the UK Engineering and Physical Sciences Research Council (EPSRC) grant EP/L016516/1 for the University of Cambridge Centre for Doctoral Training, the Cambridge Centre for Analysis}

\runauthor{D. Heydecker}

\affiliation{University of Cambridge}

\address{Centre for Mathematical Sciences\\
Wilberforce Road, Cambridge\\ CB3 0WA\\
\printead{e1}\\}

\end{aug}

\begin{abstract}
We derive two estimates for the deviation of the $N$-particle, hard-spheres Kac process from the corresponding Boltzmann equation, measured in expected Wasserstein distance. Particular care is paid to the long-time properties of our estimates, exploiting the stability properties of the limiting Boltzmann equation at the level of realisations of the interacting particle system. As a consequence, we obtain an estimate for the propagation of chaos, uniformly in time and with polynomial rates, as soon as the initial data has a $k^\mathrm{th}$ moment, $k>2$.  Our approach is  similar to Kac's proposal of relating the long-time behaviour of the particle system to that of the limit equation. Along the way, we prove a new estimate for the continuity of the Boltzmann flow measured in Wasserstein distance.
\end{abstract}


\begin{keyword}
\kwd{Kac Process}
\kwd{Law of Large Numbers}
\kwd{Wasserstein Distance}
\end{Boltzmann Equation}

\end{frontmatter}

 \section{Introduction \& Main Results} Kac \cite{FKT} introduced a Markov model for the behaviour of a dilute gas, corresponding to the spatially homogeneous Boltzmann equation.  We consider an ensemble of $N$ indistinguishable particles, with velocities  $v_1(t), ..., v_N(t) \in \mathbb{R}^d$ at time $t\ge 0$, which are are encoded in the empirical velocity distribution \begin{equation} \mu^N_t=N^{-1}\sum_{i=1}^N \delta_{v_i(t)}. \end{equation} Throughout, unless specified otherwise, we consider only the following example, known as the \emph{hard spheres} kernel, of Kac processes, which is one of two main examples of physical interest. The dynamics are as follows: \begin{enumerate}\item For every (unordered) pair of particles with velocities $v, v_\star \in \text{supp}(\mu^N_t)$, the particles collide at a rate $2|v-v_\star|/N$.

\item When two particles collide, take an independent random variable $\Sigma$, distributed uniformly on $S^{d-1}$. The particles then separate in direction $\Sigma.$ 
\item The velocities change to $v'(v, v_\star, \Sigma)$ and $v'_\star(v,v_\star, \Sigma)$, given by conservation of energy and momentum as \begin{equation}\label{eq: PCV} v'(v, v_\star, \Sigma)=\frac{v+v_\star+\Sigma|v-v_\star|}{2}; \hspace{0.5cm} v_\star'(v, v_\star, \Sigma)=\frac{v+v_\star-\Sigma|v-v_\star|}{2} \end{equation} The measure changes to \begin{equation} \label{eq: change of measure at collision} \mu \mapsto \mu^{N, v, v_\star, \Sigma} = \mu+\frac{1}{N}(\delta_{v'}+\delta_{v'_\star}-\delta_{v}-\delta_{v_\star}). \end{equation} \end{enumerate}  More formally, we consider the space $\mathcal{S}$ of Borel measures on $\mathbb{R}^d$, satisfying \begin{equation} \langle 1, \mu \rangle =1; \hspace{0.5cm} \langle v, \mu \rangle =0; \hspace{0.5cm} \langle |v|^2, \mu \rangle = 1\end{equation} where we have adopted the notational conventions that angle brackets $\langle, \rangle$ denote integration against a measure, and $v$ denotes the identity function on $\mathbb{R}^d$. $\mathcal{S}$ is called the \emph{Boltzmann Sphere}, and consists of those measures with normalised mass, momentum and energy. We write $\mathcal{S}^k$ for the subspace of $\mathcal{S}$ where the $k^\text{th}$ moment $\langle |v|^k, \mu\rangle$ is finite, and define the following family of weights: \begin{equation}\Lambda_k(\mu):=\langle (1+|v|^2)^\frac{k}{2}, \mu\rangle. \end{equation}  This leads to a natural family of subspaces: \begin{equation} \label{eq: definition of SKA} \mathcal{S}^k_a := \{\mu \in \mathcal{S}: \Lambda_k(\mu)\leq a\}. \end{equation} For shorthand, we will often write $\Lambda_k(\mu, \nu):=\max(\Lambda_k(\mu), \Lambda_k(\nu))$. \medskip \\ 
  Let $\mathcal{S}_N$ be the subset of $\mathcal{S}$ consisting of normalised empirical measures on $N$ points; we will typically write $\mu^N$ for a generic element of $\mathcal{S}_N$. Formally, the Kac process is the Markov process on $\mathcal{S}_N$ with kernel \begin{equation} \label{eq: definition of script Q} \mathcal{Q}_N(\mu^N)(A)= N \int_{\mathbb{R}^d \times \mathbb{R}^d \times S^{d-1}} 1(\mu^{N, v, v_\star, \sigma} \in A) |v-v_\star| \mu^N(dv)\mu^N(dv_\star) d\sigma.\end{equation} Note that, since the map $\mu^N \mapsto \mu^{N,v,v_\star, \sigma}$ preserves particle number, momentum, and kinetic energy, $\mathcal{Q}_N(\mu^N)$ is supported on $\mathcal{S}_N$ whenever $\mu^N \in \mathcal{S}_N$. We write $(\mu^N_t)_{t \geq 0}$ for a Kac process on $N$ particles. Observe that the rates are bounded by $2N$, and so for any initial datum $\mu^N_0$, the law of a Kac process started from $\mu^N_0$ exists, and is unique, and the process is almost surely non-explosive.

\paragraph{Measure Solutions to the Boltzmann Equation}Following many previous works, \cite{L&M, M+M, ACE}, we study measure-valued solutions to the Boltzmann equation. We define the Boltzmann collision operator $Q(\mu, \nu)$  for measures $\mu, \nu \in \mathcal{S}$ as \begin{multline} \label{eq: defn of Q} Q(\mu, \nu)=\int_{\mathbb{R}^d\times \mathbb{R}^d\times S^{d-1}} \left\{\delta_{v'}+\delta_{v_\star'}-\delta_v-\delta_{v_\star}\right\}|v-v_\star|d\sigma \mu(dv)\nu(dv_\star). \end{multline} 
For brevity, we will denote $Q(\mu, \mu)$ by $Q(\mu)$.  We say that a family $(\mu_t)_{t\geq 0}$ of measures in $\mathcal{S}$ satisfies the \emph{Boltzmann equation} if, for any bounded measurable $f$ of compact support, \begin{equation} \tag{BE}\label{BE} \forall t \geq 0 \hspace{1cm} \langle f, \mu_t \rangle =\langle f, \mu_0 \rangle +\int_0^t \langle f, Q(\mu_s)\rangle ds. \end{equation} The Boltzmann equation is known to have a unique fixed point $\gamma \in \mathcal{S}$, which is given by the Maxwellian, or Gaussian, density: \begin{equation}
\gamma(dv)=\frac{e^{-\frac{d}{2}|v|^2}}{(2\pi d^{-1})^{d/2}}dv.
 \end{equation}

 \paragraph{Measuring Convergence to the Boltzmann Equation}To discuss the convergence of Kac's process to the Boltzmann equation, we will work with the following \emph{Wasserstein metric} on $\mathcal{S}$. Consider the Sobolev space of test functions  \begin{equation} X=W^{1, \infty}(\mathbb{R}^d)=\{\text{Bounded, Lipschitz functions  } f:\mathbb{R}^d \rightarrow \mathbb{R} \}; \end{equation}
 \begin{equation} \|f\|_X:=\max\left(\sup_{v}|f|(v), \hspace{0.1cm} \sup_{v\neq w} \frac{|f(v)-f(w)|}{|v-w|}\right). \end{equation} We write $B_X$ for the unit ball of $X$; that is, those functions which are $1$-bounded and $1$-Lipschitz. Given a function $f$ on $\mathbb{R}^d$, we write $\hat{f}$ for the function \begin{equation} \hat{f}(v)=\frac{f(v)}{1+|v|^2}. \end{equation} We write $\mathcal{A}$ for the space of weighted-Lipschitz functions:\begin{equation} \label{eq: defn of script A}\mathcal{A}:=\left\{f: \mathbb{R}^d \rightarrow \mathbb{R}: \hat{f}\in X, \|\hat{f}\|_X\leq 1\right\}.\end{equation} We will also write \begin{equation} \label{eq: defn of script A 0}\mathcal{A}_0=\left\{f: \mathbb{R}^d \rightarrow \mathbb{R}: \hat{f}\in L^\infty(\mathbb{R}^d), \|\hat{f}\|_\infty\leq 1\right\}.\end{equation} The weighted Wasserstein metric $W$ is given by the duality: \begin{equation}\label{eq: definition of W} W(\mu, \nu):= \sup_{f\in\mathcal{A}}|\langle f, \mu-\nu\rangle|. \end{equation}  
We make the following remark on alternative possible choices of metric. Our metric $W$ is closely related to the $p$- Wasserstein metrics $W_p$ on the subspaces $\mathcal{S}^p$, given by \begin{equation}\label{eq: definition of Wp} W_p(\mu, \nu)=\inf\left\{\int_{\mathbb{R}^d} |v-w|^p\pi(dv,dw): \hspace{0.2cm} \pi\text{ is a coupling of } \mu \text{ and } \nu\right\}.\end{equation} In the special case $p=1$, the metric $W_1$ is known as the Monge-Kantorovich-Wasserstein (MKW) metric, and can alternatively be given by \begin{equation} W_1(\mu, \nu)=W\left(\frac{\mu}{1+|v|^2}, \frac{\nu}{1+|v|^2}\right).\end{equation} It is straightforward to check that, on the space $\mathcal{S}$, the metrics $W, W_1, W_2$ all induce the same topology, and that for some absolute constant $C$, we have the bound $W_1 \le CW$ on $\mathcal{S}$. Moreover, on the subspaces $\mathcal{S}^k_a$ defined in (\ref{eq: definition of SKA}), with $k>2$, we can find explicit bounds $W \le C W_1^\alpha$, with $\alpha\in(0,1)$. \medskip \\We now state the motivating result of \cite{ACE} on the convergence of the Kac process to the Boltzmann equation:
\begin{proposition}\label{thrm: bad convergence theorem} \cite[Theorem 10.1]{ACE} Let $k>2$. We say that a family $(\mu_t)_{t\ge 0}$ is locally $\mathcal{S}^k$-bounded if $\sup_{s\leq t} \hspace{0.1cm} \Lambda_k(\mu_s) <\infty $ for any $t \ge 0$. \\ \\ For any $\mu_0 \in \mathcal{S}^k$, there is a unique locally $\mathcal{S}^k$-bounded solution to the Boltzmann equation (\ref{BE}), starting from $\mu_0$; we write this solution as $(\phi_t(\mu_0))_{t\geq 0}$.   \medskip\\ Moreover, for any $\epsilon>0$,   $t_\text{fin}<\infty$, $\lambda<\infty$, there exist constants $C(\epsilon, \lambda, k, t_\text{fin})<\infty$ and $\alpha(d,k)>0$ such that, whenever $(\mu^N_t)_{t\geq 0}$ is a Kac process on $N\geq 1$ particles, with $\Lambda_k(\mu^N_0)  \leq \lambda, \Lambda_k(\mu_0) \leq \lambda$, we have \begin{equation} \mathbb{P}\left(\sup_{t\leq t_\text{fin}} W(\mu^N_t, \phi_t(\mu_0))>C(W(\mu^N_0, \mu_0)+N^{-\alpha})\right)< \epsilon.\end{equation} For $d\geq 3$ and $k>8$, we can take $\alpha = \frac{1}{d}$.  \end{proposition}

While the study of the convergence of the Kac process to the Boltzmann equation is a well-known and extensively studied topic, this is most usually studied through the propagation of chaos, discussed below, by contrast to the \emph{pathwise} style of estimate here which we seek to emulate. We note that the existence of solutions is known \cite{L&M} for the case $k=2$, but that nothing is known for the convergence of the Kac process in this case.\\ 

From existence and uniqueness, we can consider the Boltzmann equation as describing a non-linear semigroup of flow operators on $(\phi_t)_{t\geq 0}$ on $ \cup_{k>2} \mathcal{S}^k$. 
To prove Proposition \ref{thrm: bad convergence theorem}, Norris \cite{ACE} introduces a family of random linear operators $E_{st}$, and develops a representation formula in terms of these operators, which will be reviewed in Sections \ref{sec: continuity of BE}, \ref{sec: LMR}. Cruicial to the proof are estimates for the operator norms of $E_{st}$, which are obtained by Gr\"onwall-style estimates. As a result, the constant $C$ depends badly on the terminal time $t_\text{fin}$, with \textit{a priori} exponential growth. Our work was inspired by the observation that strong \emph{stability estimates} for the non-linear semigroup $(\phi_t)$, proven by Mischler and Mouhot \cite{M+M}, allow us to avoid using Gr\"onwall-style estimates, and hence obtain estimates with better long-time properties.   

\paragraph{Chaoticity} We will also discuss the notion of chaoticity, which is the usual framework used to analyse the convergence of the Kac process to the Boltzmann equation. In this context, it is natural to preserve the labels on the particles, and to consider the \emph{labelled Kac process} $\mathcal{V}^N_t=(v_1(t),...,v_N(t))$, taking values in the labelled Boltzmann Sphere \begin{equation} \mathbb{S}^{N}=\left\{(v_1, ..., v_N) \in (\mathbb{R}^d)^N: \hspace{0.2cm}\sum_{i=1}^N v_i=0, \hspace{0.2cm}\sum_{i=1}^N |v_i|^2 = N\right\}. \end{equation} We may recover  recover $\mathcal{S}_N$ by taking empirical measures: \begin{equation} \theta_N: \mathbb{S}^N \rightarrow \mathcal{S}_N; \hspace{1cm} (v_1, ..., v_N) \mapsto \frac{1}{N} \sum_{i=1}^N \delta_{v_i}.\end{equation}Moreover, if $\mathcal{V}^N_t$ is a labelled Kac process, then $\mu^N_t=\theta_N(\mathcal{V}^N_t)$ is an unlabelled Kac process. We write $\mathcal{LV}^N_t$ for the law of $(v_1(t),..,v_N(t))$ on $\mathbb{S}^N$. We will measure chaoticity using the following  (unweighted) Wasserstein metrics on probability measures on $(\mathbb{R}^d)^l$ for all $l\ge 1$, defined in a similar way to (\ref{eq: definition of W}): \begin{equation}\label{eq: definition of script W} \mathcal{W}_{1,l}\left(\mathcal{L}, \mathcal{L}'\right)=\sup\left\{\int_{(\mathbb{R}^d)^l} f(V)\hspace{0.1cm}(\mathcal{L}(dV)-\mathcal{L}'(dV)) \right\}\end{equation} where the supremum is over all functions $f$ of the form $f=f_1\otimes f_2 \otimes...\otimes f_l$, with each $f_i$ a bounded and Lipschitz test function, $f_i\in B_X$, and the subscript $l$ recalls the relevant dimension. We now recall the following definition from \cite{FKT}: 
\begin{definition*}[Finite Dimensional Chaos] For each $N$, let $\mathcal{L}^N$ be a law on $\mathbb{S}^N$, which is symmetric under permutations of the indexes. We say that $(\mathcal{L}^N)_{N\ge 2}$ is $\mu$-chaotic, if, for all $l \ge 1$, we have \begin{equation} \label{eq: POC}\mathcal{W}_{1,l}\left(\Pi_l[\mathcal{L}^N], \mu^{\otimes l}\right) \rightarrow 0\end{equation} where $\Pi_l$ denotes the marginal distribution on the first $l$ factors. \end{definition*} A stronger notion, put forward by Mischler and Mouhot \cite{M+M}, is that of \emph{infinite-dimensional chaos}, which allows the number of marginals $l$ to vary with $N$:\begin{equation} \label{eq: IDPOC} \max_{1\le l\le N}\left[\frac{1}{l} \mathcal{W}_{1,l}\left(\Pi_l[\mathcal{L}^N], \mu^{\otimes l}\right)\right] \rightarrow 0.\end{equation}
Kac proposed the following \emph{propagation of chaos} property. Let $(\mathcal{V}^N_t)_{t\ge 0}$ be a labelled Kac process, such that the initial distribution $\mathcal{L}\mathcal{V}^N_t$ is $\mu_0$-chaotic. Then, for all times $t\ge 0$, the law $\mathcal{LV}^N_t$ will be $\phi_t(\mu_0)$-chaotic, where $\phi_t(\mu_0)$ is the solution to the Boltzmann equation starting at $\mu_0$. This is the original sense in which Kac proposed to study the convergence of his model to the Boltzmann equation, and has been extensively studied; key previous results in this direction will be discussed in our literature review.
\subsection{Main Results} We now state the main results of the paper, concerning the long-time nature of the convergence to the Boltzmann flow. Our first theorem controls the deviation from the Boltzmann flow at a single, deterministic time $t\geq 0$, which we refer to as a \emph{pointwise} estimate. Moreover, this estimate is \emph{uniform in time}. 
 \begin{thm} \label{thrm: PW convergence} Let $0<\epsilon<\frac{1}{d}$ and let $a\ge 1$. For sufficiently large $k$, depending on $\epsilon, d$, let $(\mu^N_t)_{t\geq 0}$ be a Kac process in dimension $d\geq 3$, and let $\mu_0 \in \mathcal{S}^k$, satisfying the moment bounds \begin{equation} \Lambda_k(\mu^N_0) \leq a; \hspace{1cm} \Lambda_k(\mu_0)\le a. \end{equation}  Then for some $C=C(\epsilon, d, k)< \infty$ and $\zeta=\zeta(d)>0$, we have the uniform bound \begin{equation} \sup_{t\geq 0} \hspace{0.1cm} \left\|W\left(\mu^N_t, \phi_t\left(\mu_0\right)\right)\right\|_{L^2(\mathbb{P})} \leq C a \hspace{0.1cm} \left( N^{\epsilon-1/d} +W\left(\mu^N_0, \mu_0\right)^\zeta\right).\end{equation} This generalises, by conditioning, to the case where the initial data $\mu^N_0$ is random, provided that $\mathbb{E}\Lambda_k(\mu^N_0)\le a$. \end{thm} This result is, to the best of our knowledge, new, although an equivalent result is known for Maxwell molecules \cite{CF 2018}. We will see, in Theorem \ref{corr: PW convergence as POC}, that estimates of this form imply the propagation of chaos for hard spheres, in the sense of (\ref{eq: POC}-\ref{eq: IDPOC}), with better rates than found in \cite{M+M} for the hard spheres process. \medskip \\
Our second main theorem controls, in $L^p(\mathbb{P})$, the maximum deviation from the Boltzmann flow up to a time $t_\text{fin}$, in analogy with Proposition \ref{thrm: bad convergence theorem}. We refer to this as a \emph{pathwise, local uniform in time} estimate.
 \begin{thm} \label{thrm: Main Local Uniform Estimate}  Let $0<\epsilon<\frac{1}{2d}$, $a\ge 1$ and $p\geq 2$. For sufficiently large $k\geq 0$, depending on $\epsilon, d$, let $(\mu^N_t)_{t\geq 0}$ be a Kac process on $N\geq 2$ particles and let $\mu_0 \in \mathcal{S}^k$, with initial moments \begin{equation} \Lambda_{kp}(\mu^N_0)\le a^p ;\hspace{1cm}\Lambda_k(\mu_0)\le a. \end{equation} For some $\alpha=\alpha(\epsilon, d, p)>0$ and $C=C(\epsilon, d, p, k)< \infty$ and $\zeta=\zeta(d)>0$, we can estimate, for all $t_\text{fin}\ge 0$, \begin{equation} \left\|\hspace{0.1cm}\sup_{t\leq t_\text{fin}} \hspace{0.1cm} W\left(\mu^N_t, \phi_t(\mu_0)\right) \hspace{0.1cm}\right\|_{L^p(\mathbb{P})} \leq Ca\left( (1+t_\text{fin})^{1/p}\hspace{0.1cm} N^{-\alpha} + W(\mu^N_0, \mu_0)^\zeta)\right).\end{equation} $\alpha$ is given explicitly by \begin{equation} \alpha = \frac{p'}{2d}-\epsilon \end{equation}where $1<p'\le 2$ is the H\"{o}lder conjugate to $p$. \end{thm} At the end of this section, we will discuss related results, and how they may be compared to this estimate.  \medskip\\
 An unfortunate feature of these approximation theorems is the dependence on the unknown, and potentially large, moment index $k$; a trivial reformulation which avoids this is to ask instead for an exponential moment bound $\langle e^{z|v|},\mu^N_0\rangle \le b$, for some $z>0$. We will also prove the following variant of the theorems above which allows us to use any moment estimate higher than second.  \begin{thm}\label{thm: low moment regime}[Convergence with few moment estimates] Let $k>2$ and $a\ge 1$. Let $(\mu^N_t)$ be an $N$-particle Kac process, and $\mu_0$ in $\mathcal{S}$ with initial moment estimates \begin{equation} \Lambda_k(\mu^N_0) \le a;\hspace{1cm} \Lambda_k(\mu_0) \le a. \end{equation} There exists $\epsilon=\epsilon(d,k)>0$ and a constant $C=C(d,k)$ such that \begin{equation} \label{eq: pw convergence with few moments}\sup_{t\ge 0} \left\|W\left(\mu^N_t, \phi_t(\mu_0)\right)\right\|_{L^1(\mathbb{P})} \le Ca(N^{-\epsilon}+W(\mu^N_0,\mu_0)^\epsilon).\end{equation} For a local uniform estimate, if $p\ge 2$, then there exists a constant $C=C(d,k,p)$ and $\epsilon=\epsilon(d,k,p)>0$ such that, for all $t_\mathrm{fin}<\infty$, \begin{equation}  \left\|\sup_{t\le t_\mathrm{fin}}W\left(\mu^N_t, \phi_t(\mu_0)\right)\right\|_{L^1(\mathbb{P})} \le Ca((1+t_\mathrm{fin})^{1/p}N^{-\epsilon}+W(\mu^N_0,\mu_0)^\epsilon).\end{equation}  \end{thm} In the course of proving this result, we will see that the higher moment conditions are only required to obtain the optimal rates on a very short time interval $[0, u_N]$ and, in particular, we can obtain very good time-dependence without higher moment estimates. \medskip \\ 
We also study the long-time behaviour of the Kac Process. We cannot extend Theorem \ref{thrm: Main Local Uniform Estimate} to control the maximum deviations over all times $t\geq 0$, due to the following recurrence features of the Kac process. 
\begin{thm}\label{thrm: No Uniform Estimate} There exists a universal constant $C>0$ such that, for every $N$, for every $k> 2$ and $a>1$, there exists a Kac process $(\mu^N_t)_{t\geq 0}$ with initial moment $\Lambda_{k}(\mu^N_0)\le a$ but, almost surely, \begin{equation} \label{eq: conclusion of 1.5} \limsup_{t\rightarrow \infty} \hspace{0.1cm} W\left(\mu^N_t, \phi_t(\mu^N_0)\right) \geq 1-\frac{C}{\sqrt{N}}.\end{equation} Hence we cannot omit the factor of $(1+t_\text{fin})^{1/p}$ in Theorem \ref{thrm: Main Local Uniform Estimate}. \end{thm}  In keeping with the terminology above, we say that there is no \emph{pathwise, uniform in time} estimate. In the course of proving Theorem \ref{thrm: No Uniform Estimate}, we will show that the long-time deviation (\ref{eq: conclusion of 1.5}) is typical for the Kac process. We will show that the Kac process returns, infinitely often, to `highly ordered' subsets of $\mathcal{S}_N$, which are far from the Boltzmann flow. However, we make the following remark on the times necessary for such deviations to occur. \begin{corollary}\label{corr: variation 3} Define \begin{equation} T_{N, \epsilon}=\inf\left\{t\geq 0: W(\mu^N_t, \phi_t(\mu_0))>\epsilon \right\}.\end{equation} Let $(\mu^N_t)$ be a family of  Kac processes with an initial exponential moment bound: $\langle e^{z|v|}, \mu^N_0\rangle \leq b$, for some $z>0$ and $b>0$. Let $\mu_0\in\mathcal{S}$ satisfy $\langle e^{z|v|}, \mu_0\rangle \le b$, and suppose that $W(\mu^N_0, \mu_0)\rightarrow 0$ in probability. \medskip\\ Let $t_{N, \epsilon, \delta}$ be the quantile constants of $T_{N, \epsilon}$ under $\mathbb{P}$; that is, \begin{equation} \mathbb{P}(T_{N, \epsilon} \leq t_{N, \epsilon, \delta})\geq \delta. \end{equation} Then, for fixed $\epsilon, \delta >0$, $t_{N, \epsilon, \delta}\rightarrow \infty$, faster than any power of $N$. \end{corollary} This follows as an immediate consequence of Theorem \ref{thrm: Main Local Uniform Estimate}. Taken together with Theorem \ref{thrm: No Uniform Estimate}, we see that macroscopic deviations occur, but typically at times growing faster than any power of $N$. 

In the course of proving Theorems \ref{thrm: PW convergence}, \ref{thrm: Main Local Uniform Estimate}, we will establish the following continuity estimate for the Boltzmann flow $\phi_t$ measured in the Wasserstein distance $W$, which may be of independent interest. \begin{thm} \label{thrm: W-W continuity of phit} There exist constants $k, C, w$ depending only on $d$ such that, whenever $a\ge 1$ and $\mu, \nu \in \mathcal{S}^k_a$, we have the estimate  \begin{equation} W\left(\phi_t(\mu),  \phi_t(\nu)\right) \le Ce^{wt}a W(\mu, \nu). \end{equation} Moreover, for all $k>2$, there exist constants $C=C(k,d)$ and $\zeta=\zeta(k,d)>0$ such that, whenever $\mu, \nu \in \mathcal{S}^k_a$, we have the estimate \begin{equation} \label{eq: good continuity estimate} \sup_{t\ge 0} \hspace{0.1cm}W\left(\phi_t(\mu),\phi_t(\nu)\right) \le C a W(\mu, \nu)^\zeta.\end{equation} \end{thm} In the second part of the theorem, and in Theorems \ref{thrm: PW convergence}, \ref{thrm: Main Local Uniform Estimate} above, the exponent $\zeta$ can be taken to be $\lambda_0/({\lambda_0+2w})$ by making $k$ large enough, where $w$ is as in the first part of the theorem, and $\lambda_0=\lambda_0(d)>0$ is the spectral gap of the linearised Boltzmann operator. While it may be possible to obtain better continuity results, with $\zeta$ close to 1, we will not explore this here.\medskip \\   
Due to a result of Sznitman \cite{Sznitman Chaos}, the property of chaoticity is equivalent to convergence of the empirical measures in expected Wasserstein distance $W$. Therefore, as mentioned before, the theorems displayed above are closely related to the propagation of chaos for the hard-spheres Kac process, proven in \cite{M+M}. We now give a chaoticity result which may be derived from the previous theorems. \begin{thm}[Theorems \ref{thrm: PW convergence}, \ref{thm: low moment regime} as a Chaos Estimate] \label{corr: PW convergence as POC} We can view Theorems \ref{thrm: PW convergence}, \ref{thm: low moment regime} as \emph{Propagation of Chaos} and \emph{Conditional Propagation of Chaos}, as follows. \medskip \\ We denote $\mathcal{P}_t^N(\mathcal{V}^N, \cdot)$ the transition probabilities of the $N$-particle labelled Kac process, started at $\mathcal{V}^N\in\mathbb{S}^N$. We form the symmetrised version, which we denote $\mathcal{P}^N_t(\mu^N, \cdot)$ by \begin{equation} \mathcal{P}^N_t(\mu^N, A)=\frac{1}{\#\theta_N^{-1}(\mu^N)}\hspace{0.1cm}\sum_{\mathcal{V}^N \in \theta_N^{-1}(\mu^N)}\mathcal{P}^N_t(\mathcal{V}^N, A). \end{equation} Let $k>2$ and $a\ge 1$, and suppose $\mu^N_0 \in \mathcal{S}_N$ satisfies a moment bound $\Lambda_k(\mu^N_0)\le a$. Then we can estimate \begin{equation} \label{eq: CPOC} \sup_{t\ge 0}\hspace{0.1cm}\max_{1\le l\le N} \hspace{0.1cm} \frac{\mathcal{W}_{1,l}\left(\Pi_l[\mathcal{P}^N_t(\mu^N_0, \cdot)], \phi_t(\mu^N_0)^{\otimes l}\right)}{l} \le C\hspace{0.1cm}a \hspace{0.1cm} N^{-\beta}  \end{equation} for some constants $C=C(d,k)<\infty$; $\beta=\beta(d,k)>0$. This has the following consequences: 

\begin{enumerate}[label=\roman{*}).]

\item \emph{(Chaotic case)} Let $k, a$ be as above, and suppose $\mu_0\in \mathcal{S}$ satisfies $\Lambda_k(\mu_0)\le a.$ \medskip \\ Construct initial data $\mathcal{V}^N_0=(v_1(0),...v_N(0))$ as follows. Let $u_1, ....u_N$ be an independent, and identically distributed sample from $\mu_0$. Define \begin{equation} \overline{u}_N=\frac{1}{N}\sum_{i=1}^N u_i; \hspace{1cm}s_N=\frac{1}{N}\sum_{i=1}^N|u_i-\overline{u}_N|^2\end{equation} and set \begin{equation} v_i(0)=s_N^{-1/2}(u_i-\overline{u}_N),\hspace{0.3cm} i=1,2,...,N;\hspace{0.6cm} \mathcal{V}^N_0=(v_1(0),...,v_N(0)). \end{equation} Let $\mathcal{V}^N_t$ be a labelled Kac process starting from $\mathcal{V}^N_0$. Then there exist constants $C=C(d,k)<\infty$; $\beta=\beta(d,k)>0$ such that \begin{equation} \sup_{t\ge 0}\hspace{0.1cm}\max_{1\le l\le N} \hspace{0.1cm} \frac{\mathcal{W}_{1,l}\left(\Pi_l[\mathcal{LV}^N_t], \phi_t(\mu_0)^{\otimes l}\right)}{l} \le C\hspace{0.1cm}N^{-\beta}.  \end{equation}
 \item \emph{(General Case)}  Let $a, k$ be as above, and suppose that $(\mathcal{V}^N_t)_{t\ge 0}$ are labelled Kac processes such that the empirical measures $\mu^N_0$ satisfy \begin{equation} \mathbb{E}\Lambda_k(\mu^N_0) \le a.\end{equation} Then we have the estimate \begin{equation} \sup_{t\ge 0}\hspace{0.1cm}\max_{1\le l\le N} \hspace{0.1cm} \frac{\mathcal{W}_{1,l}\left(\Pi_l[\mathcal{LV}^N_t], \mathcal{L}^l_t\right)}{l} \le C\hspace{0.1cm}a\hspace{0.1cm}N^{-\beta} \end{equation} for $C$ and $\beta$ as in the main statement, and where $\mathcal{L}^l_t$ is the probability measure given by \begin{equation} \label{eq: defn of ll} \mathcal{L}^l_t=\mathbb{E}\left[\phi_t(\mu^N_0)^{\otimes l}\right]. \end{equation} \end{enumerate}  \end{thm}  

 \begin{remark} \begin{enumerate}[label=\roman{*}).] 
 \item Roughly, (\ref{eq: CPOC}) says that, conditional on the observation of the empirical data $\mu^N_0$ at time $0$, the law $\mathcal{L}\mathcal{V}^N_t$ is quantitatively $\phi_t(\mu^N_0)$-chaotic. This may be viewed as propagation of chaos, with the heuristic that `conditional on $\mu^N_0, \mathcal{V}^N_0$ is $\mu^N_0$-chaotic'. We term this \emph{conditional propagation of chaos}. In this spirit, we may view the main estimate (\ref{eq: CPOC}) and point (ii.) as a \emph{quenched} and \emph{annealed} pair.
 \item The polynomial result obtained here improves on the previously known result \cite[Theorem 6.2]{M+M} for the hard spheres chaos. This improvement is due to the continuity estimate (\ref{eq: good continuity estimate}), which improves on the corresponding estimate in \cite[Equations 6.39, 6.42]{M+M}; we could derive the chaoticity estimate (\ref{eq: CPOC}) by using the estimate (\ref{eq: good continuity estimate}) in the arguments of \cite[Section 6]{M+M}, at the cost of potentially requiring a stronger initial moment control. We will recall the relevant arguments for completeness, and this will be discussed in the literature review.
 \item This construction of chaotic initial data in point (i.) is due to \cite[Proposition 9.2]{ACE},  which may be thought of as `as close to perfect independence as possible'.  
\item We will show that the main point can be deduced from Theorem \ref{thrm: PW convergence} or \ref{thm: low moment regime}. However, we will see in  Section \ref{sec: proof of POC} that deriving either of these from this result appears to be no less technical than the main proof presented in Section \ref{sec: proof of pw}. \end{enumerate}   \end{remark}  In our arguments, we will frequently encounter numerical constants which are ultimately absorbed into the constants $C$ whose dependence is specified in the relevant theorem. To ease notation, we will denote inequality, up to such a constant, by $\lesssim$.

\subsection{Plan of the paper} Our programme will be as follows:  \begin{enumerate} [label=\text{\roman*.}] 
\item In the remainder of this section, we will present a review of known results in the study of the Kac process and similar models. We will then discuss several aspects of our results, and how they may be interpreted. 

\item For later convenience, we discuss some classical moment estimates for the Kac process and the Boltzmann equation. These allow us to stochastically control the weights $\Lambda_k$ in appropriate $L^p$ spaces.

\item We cite the analytical \emph{regularity and stability estimates} from Mischler and Mouhot, \cite{M+M}. The stability estimates, in particular, are crucial to obtaining the good time-dependence in Theorems \ref{thrm: PW convergence}, \ref{thrm: Main Local Uniform Estimate}.
\item As a first application of the stability estimates, we analyse the continuity of the Boltzmann flow $\phi_t$ on subsets $\mathcal{S}^k_a$, with respect to the metric $W$, and uniformly in time. This is the content of Theorem \ref{thrm: W-W continuity of phit}, and allows us to reduce Theorems \ref{thrm: PW convergence}, \ref{thrm: Main Local Uniform Estimate} to the special case $\mu_0=\mu^N_0$.
\item We use ideas of infinite-dimensional differential calculus, developed by \cite{M+M}, to prove an \emph{interpolation decomposition} of the difference $\mu^N_t - \phi_t(\mu^N_0)$. This is the key identity used for the proofs of Theorems \ref{thrm: PW convergence}, \ref{thrm: Main Local Uniform Estimate}, as all of the terms appearing in our formula can be controlled by the stability estimates.
\item We then turn to the proof of Theorem \ref{thrm: PW convergence}. The main technical aspect is the control of a family of martingales $(M^{N,f}_t)_{f\in \mathcal{A}}$, uniformly in $f$.  This is obtained using a quantitative compactness argument similar to that in \cite{ACE}.
\item For a local uniform analysis, we first adopt the ideas of Theorem \ref{thrm: PW convergence} to a local uniform setting, with suitable adaptations, to state a local uniform martingale estimate, and deduce a preliminary, weak version of Theorem \ref{thrm: Main Local Uniform Estimate} with worse dependence in $t_\text{fin}$. We then use the stability estimates to `bootstrap' to the improved estimate Theorem \ref{thrm: Main Local Uniform Estimate}, and finally return to prove the local martingale estimate.
  \item We next prove Theorem \ref{thm: low moment regime}. The strategy here is to use a localised form of the main argument from \cite{ACE} to control behaviour on a very short time interval $[0, u_N]$, and use the previous results, together with the \emph{moment production} property recalled in Section \ref{sec:moment estimates}, to control behaviour at times larger than $u_N$. 
\item We prove Theorem \ref{thrm: No Uniform Estimate}, based on relaxation to equilibrium. 
\item Finally, we prove the chaoticity result Theorem \ref{corr: PW convergence as POC}. This proof follows a similar pattern to the proof in \cite{M+M}, using our esimates. \end{enumerate} 

\subsection{Literature Review} We will now briefly discuss related works, to which our results may be compared.

\paragraph{1. Probabilistic Techniques for the Kac Process and Boltzmann Equation} The probabilistic, \emph{pathwise} approach to the Kac process was pioneered by Tanaka \cite{Tanaka 78,Tanaka 02}, who constructed a Markov process describing the velocity of a `typical' particle in the Kac process with Maxwell molecules, and whose law at time $t$ is the solution to the associated Boltzmann equation. This was generalised by Fournier and M\'el\'eard \cite{FM} to include the cases without cutoff, and for non-Maxwellian molecules. A similar idea was used by Rousset \cite{Rousset} to prove convergence to equilibrium as $t\rightarrow \infty$. \medskip \\ Our main convergence results may be compared to the motivating work of Norris \cite{ACE}, of which the main result is recalled in Proposition \ref{thrm: bad convergence theorem} above. Theorem \ref{thrm: Main Local Uniform Estimate} improves on Proposition \ref{thrm: bad convergence theorem} in two notable ways. Firstly, we have much better asymptotic behaviour in the time-horizon $t_\text{fin}$, which was the original motivation for our work. Secondly, we control the deviation in the stronger sense of $L^p$, rather than in probability; this arises as a result of using moment estimates within the framework of a `growth control', rather than excluding events of small probability where the moments are large. We also remark that the analysis of the martingale term in Sections \ref{sec: proof of pw}, \ref{sec: proof of LU}  is simplified from the equivalent analysis in \cite[Theorem 1.1]{ACE} by our `interpolation decomposition', Formula \ref{form:newdecomposition}, which removes anticipating behaviour.  
\paragraph{2. Propagation of Chaos for the Kac Process} The problem of propagation of chaos for the Kac Process and Boltzmann equation has been extensively studied. The earliest results in this direction are due to McKean \cite{McKean 67}, Gr\"unbaum \cite{Gruenbaum} and Sznitman \cite{Sznitman BE}, and prove the qualitative statement (\ref{eq: POC}) for the cases of the hard spheres kernel considered here, or for the related case of Maxwell molecules. Recent work has produced quantitative estimates: Mischler and Mouhout \cite{M+M} showed propagation of infinite-dimensional chaos (\ref{eq: IDPOC}) for both hard spheres and Maxwell molecules. The estimates are uniform in time, with a quantitative estimate going as $(\log N)^{-r}$ for the hard spheres case. As remarked above, our estimates (Theorem \ref{thrm: PW convergence}, \ref{thm: low moment regime}, \ref{corr: PW convergence as POC}) improve this rate; this improvement is due to the improvement of Theorem \ref{thrm: W-W continuity of phit} over the corresponding estimate in \cite{M+M}, and this will be discussed further below. More recently, \cite{CF 2018} proved a chaoticity estimate for Maxwell molecules in $d=3$, measured in the $L^2(\mathbb{P})$ norm of Wasserstein$_2$ distance (\ref{eq: definition of Wp}), and with an almost optimal rate $N^{\epsilon-1/3}$, which is almost completely analagous to Theorem \ref{thrm: PW convergence}. 

\paragraph{3. Propagation of Chaos for Related Models} We also mention the study of other models in kinetic theory where chaoticity has been studied. Malrieu \cite{Malrieu} studied a McKean-Vlasov model related to granular media equations, and deduced chaoticity for a related system. The main estimate here is a uniform in time estimate, similar in nature to Theorem \ref{thrm: PW convergence}. Similarly, Bolley, Guillin and Malrieu  \cite{BGM} have also proven propagation of chaos for a particle system associated to a Vlasov-Focker-Plank equation, through a pointwise convergence result. Most recently, Durmus et al. \cite{Durmus} have proved a uniform in time chaoticity estimate based on a coupling approach, for the case with a confinement potential. Both of these models are amenable to the general framework of \cite{M+M}, and propagation of chaos for these models has been proven using the same techniques in a companion paper \cite{M+MCompanion}. \medskip \\We may also compare Theorem \ref{thrm: Main Local Uniform Estimate} to a result of Bolley, Guillin and Villani \cite[Theorem 2.9]{BGV}, which proves exponential concentration of the maximum $\sup_{t\le t_\text{fin}} W(\mu^N_t, \phi_t(\mu))$ about $0$, for McKean-Vlasov dynamics. This improves upon the rates $\mathcal{O}(N^{-\infty})$ which would be obtained using Theorem \ref{thrm: Main Local Uniform Estimate}, but does not produce an explicit $L^p(\mathbb{P})$ bound. More recently, Holding \cite{Holding} proved a result similar to Theorem \ref{thrm: Main Local Uniform Estimate} for McKean-Vlasov systems interacting through a H\"older continuous force, in order to deduce propagation of chaos. However, neithere of these results track the dependence in the terminal time $t_\text{fin}$, and so may have much weaker time dependence than our result. To the best of our knowledge, no local uniform estimate for the McKean-Vlasov system exists which seeks to optimise time dependence in the spirit of Theorem \ref{thrm: Main Local Uniform Estimate}; the applicability of our methods to this system will be considered in the discussion section below.   \medskip \\ The notion of chaoticity has also been studied in more abstract settings. Sznitman \cite{Sznitman Chaos} has studied equivalent conditions for a family of measures to be chaotic, and Gottlieb \cite{Gottlieb} has produced a necessary and sufficient condition for families of Markov chains to propagate chaoticity.

\paragraph{4. Relaxtion to Equilibrium of the Kac Process} Kac \cite{FKT} proposed to relate the asymptotic behaviour of the Boltzmann flow $\phi_t(\mu_0)$ to the asymptotic relaxation to equilibrium of the particle system, and conjectured the existence of a spectral gap for the master equation. This has been extensively studied, and Kac's conjecture on the spectral gap positively answered \cite{Carlen 00,Janvresse,Carlen 03,Malsen}. However, this is not an entirely satisfactory answer for Kac's question on convergence to equilibrium; for chaotic initial data, this still requires times order $\mathcal{O}(N)$ to show relaxation to equilibrium. Carlen et al. also considered in a later paper \cite{Carlen 08} the more intricate notion of convergence \emph{in relative entropy}, which somewhat avoids this problem. Mischler and Mouhot \cite{M+M} answered Kac's question, proving relaxation to equilibrium in Wasserstein distance, uniformly in $N$, for the cases of hard spheres and Maxwell molecules. \medskip \\ We remark that our philosophy is similar to Kac's proposal. Rather than investigating the long-time behaviour of the \emph{law} $\mathcal{LV}^N_t$ of the Kac process, our results use the asymptotics of the Boltzmann equation to partially understand the asymptotics of \emph{realisations} of the Kac process. Moreover, Theorem \ref{thrm: No Uniform Estimate} shows that this cannot be extended to completely understand the full, long-time asymptotics in this sense.

 \subsection{Discussion of Our Results}\label{sec: discussion of our results}
In this subsection, we will discuss the interpretation of our results, especially in view of the framework of chaoticity set out above. 

\paragraph{1. Theorems \ref{thrm: PW convergence}, \ref{thrm: Main Local Uniform Estimate} as a pathwise interpretation of the Boltzmann Equation} The main philosophy of our approach follows \cite{ACE}, in considering the Kac process as a Markov chain, and adapting techniques \cite{D&N,PDF} from the general scaling limits of Markov processes.  \medskip \\ It is instructive to compare this to the case of a particle system evolving under Vlasov dynamics. In this case, we write $\mu^{N,\text{Vl}}_t$ for the $N$-particle empirical measure, evolving under (nonrandom) Hamiltonian dynamics; Dobrushin \cite{Dobrushin} showed that $\mu^{N, \text{Vl}}_t$ is a weak measure solution to the associated mean field PDE, the Vlasov equation. For the case of Kac dynamics, we may interpret Theorems \ref{thrm: PW convergence}, \ref{thrm: Main Local Uniform Estimate} as saying that \begin{equation}  \forall t\ge 0\hspace{1cm} \mu^N_t=\phi_t(\mu^N_0)+\mathcal{N}^N_t\end{equation} where $\mathcal{N}^N_t$ is a stochastic noise term, which is small in an appropriate sense. This is a general phenomenon in the `fluid limit' scaling of Markov processes \cite{D&N,PDF,ACE}. In this sense, we may interpret the Boltzmann equation in a \emph{pathwise} sense; we stress that this interpretation of the Boltzmann equation does \emph{not} require any chaoticity assumptions on the initial data. 

\paragraph{2. Theorem \ref{thrm: PW convergence} as Propagation of Chaos}  It is natural, and instructive, to compare our chaoticity result Theorem \ref{corr: PW convergence as POC} and our techniques to those of \cite{M+M}, on whose work we build. \medskip \\  In Theorem \ref{corr: PW convergence as POC}, we have improved the rate of chaoticity, from $(\log N)^{-r}$ to a polynomial estimate $N^{-\alpha}$. In proving this result, we will compare our estimates to the estimates of the three error terms $\mathcal{T}_1$, $\mathcal{T}_2$, $\mathcal{T}_3$ in the abstract result \cite[Theorem 3.1]{M+M}:\begin{enumerate}[label=\roman{*}).] \item The first term $\mathcal{T}_1$ is a purely combinatorial term which may be controlled by general, elementary arguments. \item The second error term $\mathcal{T}_2$ may be controlled by $\mathbb{E}W(\phi_t(\mu^N_0), \mu^N_t)$, which is a special case of Theorems \ref{thrm: PW convergence}, \ref{thm: low moment regime} with $\mu_0=\mu^N_0$. \item The third error $\mathcal{T}_3$ depends on the continuity of the Boltzmann flow $\phi_t$ in Wasserstein distance, which is controlled by the H\"older estimates Theorem \ref{thrm: W-W continuity of phit}. \end{enumerate} As mentioned above, the improvement over \cite[Theorem 6.2]{M+M} is due to the improved control on $\mathcal{T}_3$, using the estimate (\ref{eq: good continuity estimate}). The controls on $\mathcal{T}_1, \mathcal{T}_2$ are similar to those in \cite{M+M}, and the claimed result (\ref{eq: CPOC}) follows by using our estimates (\ref{eq: pointwise bound on martingale term}, \ref{eq: good continuity estimate}) in the arguments of \cite[Section 6]{M+M}. In order to give a self-contained proof, we will recall the relevant arguments in Section \ref{sec: proof of POC}. \medskip\\  We also remark that we use each of the assumptions (\textbf{A1}-\textbf{5}) from \cite{M+M} in our analysis: \begin{enumerate}[label=\roman{*}).] \item Assumption (\textbf{A1}) corresponds to the moment bounds, which follow from the discussion of moment bounds in Proposition \ref{thrm:momentinequalities}.
\item Assumption (\textbf{A2}i) and (\textbf{A5}) concern the continuity of the Boltzmann flow $\phi_t$, which is addressed in Theorem \ref{thrm: W-W continuity of phit}. Assumption (\textbf{A2}ii) concerns the continuity of the collision operator $Q$, which is discussed in Section \ref{sec: Regularity and Stability Estimates}.
\item Assumption (\textbf{A3}) is the convergence of the generators. A special case of this is the content of Lemma \ref{lemma:DAP}, which is used to prove our `interpolation decomposition' Formula \ref{form:newdecomposition}. 
\item Assumption (\textbf{A4}) is the differential stability of the Boltzmann flow $\phi_t$, recalled in Proposition \ref{thrm: stability for BE}, which is crucial to obtaining estimates with good long-time properties.\end{enumerate} We will also see that, in order to recover Theorem \ref{thrm: PW convergence} theorem from either of the chaoticity results (Theorem \ref{corr: PW convergence as POC} or \cite[Theorem 6.2]{M+M}), we would need to move a supremum over test functions $f$ \emph{inside an expectation}, which corresponds to one of the most technical steps in our proof (Lemmas \ref{thrm: pointwise martingale control}, \ref{thrm: local uniform martingale control}). Moreover, this technique cannot generalise to produce a pathwise, local uniform convergence result analogous to Theorem \ref{thrm: Main Local Uniform Estimate} or Proposition \ref{thrm: bad convergence theorem}.

 \paragraph{3. Theorems \ref{thrm: PW convergence}, \ref{thrm: Main Local Uniform Estimate} without chaoticity} We also remark that neither of the approximation results Theorems \ref{thrm: PW convergence}, \ref{thrm: Main Local Uniform Estimate} require special preparation of the initial data, beyond a moment estimate; in particular, both are valid even if the initial data $\mathcal{V}^N_0$ are not chaotic. We will now give an explicit example of such a distribution where this chaoticity property fails. \begin{example}[Non-chaotic initial data]\label{ex: nonchaotic initial data} Assume that $N$ is a multiple of $2^d$. Choose $\Sigma \in S^{d-1}$ uniformly at random, and let $P_1, P_2,...,P_{2^d}$ be the $2^d$ points obtained from $\Sigma$ by all reflections in coordinate axes. Let $\mathcal{V}^N_0$ be given by giving $\frac{N}{2^d}$ particles velocity $P_i$, for each $i=1,...,2^d$, such that the resulting law $\mathcal{LV}^N_0$ is symmetric. Then each marginal distribution is the uniform distribution $\text{Uniform}(S^{d-1}) \in \mathcal{S}$, but there exists a constant $\delta>0$, uniform in $N$, such that \begin{equation} W\left(\mu^N_0, \text{Uniform}(S^{d-1})\right) \ge \delta >0\end{equation} almost surely, where $\mu^N_0$ is the empirical measure of $\mathcal{V}^N_0$. In particular, by Sznitman's characterisation, $\mathcal{V}^N_0$ is not $\text{Uniform}(S^{d-1})$-chaotic.  \end{example}
In cases such as this, we may still understand the Boltzmann equation as `nearly' holding pathwise, in the sense of point 1. Alternatively, we may view the result Theorem \ref{thrm: PW convergence}, and its consequence in Theorem \ref{corr: PW convergence as POC}, as a chaoticity estimate for $\mathcal{V}^N_t$ about $\phi_t(\mu^N_0)$, \emph{conditional on the initial measure $\mu^N_0$}.

\paragraph{4. Theorem \ref{thrm: No Uniform Estimate} in view of the $H$-Theorem}
As commented after the statement of Theorem \ref{thrm: No Uniform Estimate}, the key idea of the proof of Theorem \ref{thrm: No Uniform Estimate} is that the Kac process $\mu^N_t$ will, infinitely often, return to `highly ordered' subsets of the state space $\mathcal{S}_N$. However, this appears to contradict a na\"ive statement of Boltzmann's celebrated $H$-Theorem \cite{Boltzmann H Thrm}, that \emph{``entropy increases"}. Indeed, this is highly reminiscent of Zermelo's objection, based on Poincar\'e recurrence of deterministic dynamical systems \cite{Zermelo H Thrm}. \medskip \\ However, our results are compatible with the $H$-Theorem, which is rigorously established in \cite{M+M}. This apparent paradox arises because the $H$-functional, representing the negative of entropy, is a \emph{statistical}, and not \emph{pathwise}, concept; that is, $H_t$ depends on the data $\mathcal{V}^N_t$ through the law $\mathcal{LV}^N_t$, rather than being a random variable depending directly on a particular observation $\mathcal{V}^N_t(\omega)$. In particular, for our case, the time $T_N$ of reaching the `ordered state' is a large, random time, and observing a particular realisation $T_N(\omega)=t$ tells us very little about the general behaviour $\mathcal{LV}^N_t$, and so about the entropy at time $t$.
\paragraph{5. Sharpness of our Results} We will now discuss how sharp the main results (Theorems \ref{thrm: PW convergence}, \ref{thrm: Main Local Uniform Estimate}) are, with regards to dependencies in $N$, and the terminal time $t_\text{fin}$ in the case of Theorem \ref{thrm: Main Local Uniform Estimate}. \\ 
\subparagraph{5a. $N$-dependence} It is instructive to first consider the `optimal' case of independent particles, for which the empirical measure converges in Wasserstein distance at rate $N^{-1/d}$. More precisely, for $d\ge 3$, let $\mu \in \mathcal{S}^k_a$ for $k\ge \frac{3d}{d-1}$, and let $\mu^N$ be an empirical measure for $N$ independent draws from $\mathcal{S}$. Then, for some $C=C(a, k, d)$, we have\begin{equation} \left\|W(\mu^N, \mu)\right\|_{L^2(\mathbb{P})} \le C N^{-1/d}. \end{equation} This is shown in \cite[Proposition 9.3]{ACE}. Moreover, this rate is optimal: if $\mu$ is absolutely continuous with respect to the underlying Lebesgue measure, then the optimal approximation in $W$ metric is of the order $N^{-1/d}$, for $d\ge 3$. Results of Talagrand (\cite{T1,T2}, and discussion in \cite{Wasserstein}) suggest that this may also be true for higher $L^p$ norms, at least for the simple  case of the uniform distribution on $(-1, 1]^d$. \medskip \\ In view of this, we see that the exponent for the pointwise bound is \emph{almost sharp}, in the sense that we obtain exponents $\epsilon-\frac{1}{d}$ which are arbitrarily close to the optimal exponent $-\frac{1}{d}$, but cannot obtain the optimal exponent itself. This appears to be a consequence of using a particular estimate (\ref{eq: stability for BE 1}) from \cite{M+M}, which is `almost Lipschitz' in a similar sense. For the local uniform estimate Theorem \ref{thrm: Main Local Uniform Estimate}, we obtain exponent $-\alpha$, where $\alpha$ is given by \begin{equation} \alpha=-\epsilon+\frac{p'}{2d}; \hspace{1cm} \frac{1}{p}+\frac{1}{p'}=1.\end{equation} In the special case $p=2$, this produces the almost sharp exponent as discussed above. However, for $p>2$, the exponents are bounded away from $-\frac{1}{d}$, and so do not appear to be sharp.  \\ 
\subparagraph{5b. Time Dependence} In light of Theorem \ref{thrm: No Uniform Estimate}, we see that we cannot exclude the factor $(1+t_\text{fin})^{1/p}$ in Theorem \ref{thrm: Main Local Uniform Estimate}. Hence, this time dependence is sharp \emph{among power laws}. However, we do not know what the \emph{true} sharpest time-dependence is. Similar techniques to those of Graversen and Peskir \cite{GP} may be able to provide a sharper bound; we do not explore this here.\medskip \\ We remark that Theorem \ref{thrm: Main Local Uniform Estimate} interpolates between almost optimal $N$ dependence at $p=2$, and almost optimal $t_\text{fin}$ dependence as $p\rightarrow \infty$. Moreover, by taking $p\rightarrow \infty$, we sacrifice optimal dependence in $N$, but the exponent $\alpha(d,p)$ is bounded away from $0$, and so we have good convergence, on any polynomial time scale. This is the content of Corollary \ref{corr: variation 3}.
\paragraph{6. Further Applicability of our Methods in Kinetic Theory} Finally, we will mention other models in kinetic theory which may be amenable to our techniques.
\begin{enumerate}[label=\alph{*}).] \item \emph{Sharp $N$ dependence for hard spheres.} We believe that our techniques could be modified to prove an estimate for Theorem \ref{thrm: PW convergence}, and Theorem \ref{thrm: Main Local Uniform Estimate} in the case $p=2$, in order to obtain the optimal rate $N^{-1/d}$ discussed above; however, this would likely come at the cost of poor dependence in time. Since a similar result (Proposition \ref{thrm: bad convergence theorem}) is already known, and since this is not the spirit of this work in seeking to optimise time dependence, we will not consider this further. \\ \item \emph{The Kac process on Maxwell Molecules.} In addition to the hard spheres case analysed here, the main collision kernel of physical interest is the case of \emph{Maxwell molecules} with or without cutoff. Many of the estimates used in our argument for the hard spheres kernel have an analagous version for Maxwell molecules, including the stability estimates proven in \cite{M+M}. For this case, a result similar to Theorem \ref{thrm: PW convergence} is already known \cite[Theorem 2]{CF 2018}.\\ 
\item \emph{McKean-Vlasov Dynamics, and Inelastic Collisions.} Other kinetic system which may be analysed in the framework of \cite{M+M} include cases of \emph{McKean-Vlasov} dynamics, and \emph{Inelastic Collisions, coupled to a heat bath}, which have been studied in the functional framework of \cite{M+M} by Mischler, Mouhot and Wennburg in a companion paper \cite{M+MCompanion}. In these cases, the analagous estimates for stability and differentiability, computed in \cite{M+MCompanion}, have potentially poor dependence in time. As a result, our methods would still apply, but with correspondingly poor time dependence. \medskip \\ For the case of McKean-Vlasov dynamics without confinement potential, this is a fundamental limitation; Malrieu \cite{Malrieu} showed that the propagation of chaos is \emph{not} uniform in time. Instead, he proposed to study a \emph{projected} particle system, which satisfies uniform propagation of chaos, and whose limiting flow has exponential convergence to equilibrium \cite[Theorem 6.2]{Malrieu}. This suggests that it may be possible to use our bootstrap method, used in the proof of Theorem \ref{thrm: Main Local Uniform Estimate}, to obtain a pathwise estimates with good long-time properties, analagous to Theorem \ref{thrm: Main Local Uniform Estimate}. \medskip \\ We remark that, in the case of McKean-Vlasov dynamics, the presence of Brownian noise may complicate the derivation of the interpolation decomposition (Formula \ref{form:newdecomposition}), which is the key identity required for our argument.   \end{enumerate}

\subsection*{Acknowledgements} I am grateful to my supervisor, James Norris, for the suggestion of this project and for several useful remarks which allowed me to strengthen the results, and to Cl\'ement Mouhot, for a useful conversation concerning the interpretation of our results. I would also like to express my gratitude to the two anonymous reviewers, whose suggestions and comments over the course of two iterations led to several substantial strengthening of the results.   \section{Moment estimates}\label{sec:moment estimates}

In order to deal with the appearance of the moment-based weights $\Lambda_k$ in future calculations, we discuss the moment structure of Kac's Process and the Boltzmann Equation. That is, we seek bounds on  $\Lambda_k(\mu_t)$ where $\mu_t$ is, correspondingly, either a Kac process, or a solution to the Boltzmann equation.\medskip \\ The results presented here are mostly classical, and the arguments are well-known for the Boltzmann equation. Central to the proof is an inequality due to Povzner \cite{Povzner}, from which Elmroth \cite{Elmroth} deduced global moment bounds for the (function-valued) Boltzmann equation in terms of the moments of the initial data. This conclusion was strengthened to moment \emph{production} by Desvillettes \cite{Desvilettes} provided control of an initial moment $\Lambda_s(\mu_0)$ for any $s>2$. Wennberg \cite{Wennberg,Wennberg Mischler} demonstrated an optimal version of this result, only requiring finite initial energy $\langle |v|^2, \mu_0\rangle$.  Bobylev \cite{Bobylev} proved propagation of exponential moments, which may also be applied here as a simplification. These results have been proven for measure-valued solutions of the Boltzmann equation by Lu and Mouhot \cite{L&M}, and the techniques have been applied to the Kac process by Mischler and Mouhot \cite{M+M} and Norris \cite{ACE}. We collect below the precise results which we will use.
 \begin{proposition} [Moment Inequalities for the Kac Process and Boltzmann Equation] \label{thrm:momentinequalities} We have the following moment bounds for polynomial velocity moments: \\		
\begin{enumerate}[label={(\roman*.)},ref={2.\roman*.}] \item \label{lemma:momentboundpt1}  Let $(\mu^N_t)_{t\geq 0}$ be a Kac process on $N\geq 1$ particles, and let $q>2$, $p\ge 2$ with $q\ge p$. Then there exists a constant $C(p,q)<\infty$ such that, for all $t\ge 0$, \begin{equation} \label{eq: pointwise moment bound}  \mathbb{E}\left[ \Lambda_q(\mu^N_t) \right] \leq C(1+t^{p-q})\Lambda_p(\mu^N_0)\end{equation} and, for another constant $C=C(q)$, \begin{equation} \label{eq: local uniform moment bound} \mathbb{E}\left(\sup_{0\leq t \leq t_\text{fin}} \Lambda_q(\mu^N_t)\right) \leq (1+C(q)t_\text{fin})\Lambda_q(\mu^N_0).\end{equation} 
\item \label{lemma:momentboundpt2} Let $p, q$ be as above, and let, and $\mu_0\in \cup_{k>2}\mathcal{S}^k$. Then there exists a constant $C=C(p,q)$ such that the solution $\phi_t(\mu_0)$ to (\ref{BE}) satisfies \begin{equation}\label{eq: BE moment bound}  \Lambda_q(\phi_t(\mu_0)) \le C(1+t^{p-q})\Lambda_p(\mu_0).\end{equation} \item There exist constants $C_1, C_2 < \infty$ such that, whenever $\mu_0 \in \cup_{k>2} \mathcal{S}^k$, we have the bound for all $t\ge 0$ \begin{equation} \int_0^t \Lambda_3(\phi_s(\mu_0))ds \le C_1 t+C_2\langle\hspace{0.05cm} (1+|v|^2) \hspace{0.05cm}\log(1+|v|^2),\mu_0\rangle.\end{equation}   As a consequence, if $c\ge0$, then there exists $w<\infty, k<\infty$ such that, for all $t\ge 0$, \begin{equation} \exp\left(c\int_0^t \Lambda_3(\phi_s(\mu_0))ds\right)\le e^{wt}\Lambda_k(\mu_0). \end{equation}
\end{enumerate} \end{proposition}  The first item is exactly \cite[Proposition 3.1]{ACE}. For the second item, if $\phi_t(\mu_0)$ is locally $\mathcal{S}^k$ bounded for all $k$, then we can apply the same reasoning as the cited proposition to the Boltzmann equation. To remove this condition, we consider the Boltzmann equation started from $\mu_\delta=\phi_\delta(\mu_0)$: thanks to the qualitative moment creation property \cite{Desvilettes, Wennberg Mischler}, the Boltzmann flow started at $\mu_\delta$ is locally $\mathcal{S}^k$ bounded for all $k$, and so the claimed result holds with $\mu_\delta$ in place of $\mu_0$. The claimed result may then be obtained by carefully taking the limit $\delta\downarrow 0$. \medskip \\  The first conclusion of item iii.  is proven in \cite[Equation 6.20]{M+M}, and the final point follows, using the interpolation, for all $\mu \in \mathcal{S}$,\begin{equation} \langle (1+|v|^2) \hspace{0.05cm}\log(1+|v|^2),\mu\rangle \le 8(1+\log \Lambda_5(\mu)). \end{equation}  In our estimates for the various terms of the interpolation decomposition, we will frequently encounter the weightings $\Lambda_k(\mu^N_t)$ appearing in the integrand. We refer to points (i-ii.) of Proposition \ref{thrm:momentinequalities}, along with the following lemma, as \emph{growth control} of the weightings, which allows us to control these factors in suitable $L^p$ norms.  \begin{lemma} \label{lemma:momentincreaseatcollision} Let $\left(\mu^N_t\right)_{t\geq 0}$ be a Kac process on $N\geq 1$ particles, and fix an exponent $k\geq 2$. Then for any time $t\geq 0$, and any measure $\mu^N$ which can be obtained from $\mu^N_t$ by a collision, \begin{equation}  \Lambda_k(\mu^N) \leq 2^{\frac{k}{2}+1} \Lambda_k(\mu^N_t)\end{equation} \end{lemma}  \begin{proof} This is immediate, by noting that if $v, v_\star$ are pre-collision velocities leading to post-collision $v', v_\star'$, we have the bound \begin{equation} \begin{split} (1+|v'|^2)^k &\le ((1+|v|^2)+(1+|v_\star|^2))^\frac{k}{2} \\ & \le 2^{k/2}((1+|v|^2)^\frac{k}{2}+(1+|v_\star|^2)^\frac{k}{2}). \end{split}\end{equation} Using the same bound for $v_\star'$ leads to the claimed result. \end{proof}

A final property of the weighting estimates which will prove useful is the following correlation inequality: \begin{lemma} \label{lemma: correlation of moments} Let $k_1, k_2 \geq 2$, and let $\mu \in \mathcal{S}^{k_1+k_2}$. Then we have \begin{equation} \Lambda_{k_1}(\mu)\Lambda_{k_2}(\mu)\leq \Lambda_{k_1+k_2}(\mu).\end{equation}  \end{lemma} \begin{proof} Since the maps $x\mapsto (1+|x|^2)^{k_i/2}$, for $i=1,2$, are both monotonically increasing on $[0, \infty)$, for any $v, v_\star$ we have the bound \begin{equation} \left\{(1+|v|^2)^{k_1/2}-(1+|v_\star|^2)^{k_1/2}\right\}\left\{(1+|v|^2)^{k_2/2}-(1+|v_\star|^2)^{k_2/2}\right\}\geq 0.\end{equation} Integrating both variables with respect to $\mu$ produces the result.  \end{proof}
\section{Regularity and Stability Estimates} \label{sec: Regularity and Stability Estimates}

In this section, we give precise statements of analytical results concerning the flow maps $(\phi_t)_{t\geq 0}$, and the drift operator $Q$, which will be used in our convergence theorems. We need a combination of \emph{regularity} for the drift map $Q$, which appears in the proof of Lemma \ref{thrm: local uniform martingale control}, and \emph{differentiability and stability} results for the flow maps $(\phi_t)_{t\geq 0}$.

\subsection{Stability Estimates}
The key component to our analysis of the Kac process is the \emph{stability} of the limiting Boltzmann equation - that is, that the limit flow suppresses errors, rather than allowing exponential amplification. We begin by defining appropriate linear structures.
\begin{definition}\label{def: weighted normed spaces} Consider the space $Y$ of signed measures, given by \begin{equation} Y=\left\{ \xi: \hspace{0.3cm}\|\xi\|_\mathrm{TV} <\infty;\hspace{0.3cm} \langle 1, \xi \rangle =0\right\}.\end{equation}  We equip $Y$ with the total variation norm $\|\cdot\|_\mathrm{TV}$. For real $q\geq 0$, we define the subspace $Y_q$ of measures with finite $q^\text{th}$ moments: \begin{equation} Y_q =\left\{ \xi \in Y: \langle 1+|v|^q, |\xi|\rangle < \infty  \right\}. \end{equation} We define the norm with $q$-weighting on $Y_q$ by \begin{equation} \|\xi\|_{\mathrm{TV}+q}=\langle 1+|v|^q, |\xi|\rangle.\end{equation} The notation $\|\cdot\|_{\mathrm{TV}+q}$ is chosen to emphasise that this is a total variation norm, with additional polynomial weighting of order $q$, while avoiding potential ambiguity with the $L^q$ norms of random variables. \end{definition}  \begin{remark}\label{rmk: compactness} The total variation norms $\|\cdot\|_{\mathrm{TV}+q}$ appearing in the following analysis are much stronger than the Wasserstein distance appearing in Theorems \ref{thrm: PW convergence}, \ref{thrm: Main Local Uniform Estimate}, \ref{thm: low moment regime}. We can understand this as follows. Recalling the definitions of $\mathcal{A}, \mathcal{A}_0$ in (\ref{eq: defn of script A}, \ref{eq: defn of script A 0}), we note that the $\mathrm{TV}+2$ distance is given by a duality \begin{equation} \|\mu-\nu\|_{\mathrm{TV}+2} =\sup_{f\in \mathcal{A}_0} \hspace{0.1cm} |\langle f, \mu-\nu\rangle| \end{equation} and, if we write $\mathcal{A}|_r, \mathcal{A}_0|_r$ for the restriction of functions to $[-r,r]^d$, then the inclusion \begin{equation} \mathcal{A}|_r\subset \mathcal{A}_0|_r\end{equation} is compact in the norm of $\mathcal{A}_0|_r$, by the classical theorem of Arzel\'a-Ascoli. This is at the heart of a \emph{quantitative compactness} argument in Lemmas \ref{thrm: pointwise martingale control}, \ref{thrm: local uniform martingale control}, which allows us to to take the supremum over $f\in\mathcal{A}$ inside the expectation.  \end{remark} 
We can now state the precise results as they appear in \cite[Lemma 6.6]{M+M}:
 \begin{proposition}\label{thrm: stability for BE} Let $\eta \in (0,1)$. Then there are absolute constants $C\in (0, \infty)$ and $\lambda_0>0$ such that, for $k$ large enough (depending only on $\eta$), and all $\mu, \nu \in \mathcal{S}^k$, there is a unique solution $(\xi_t)_{t\geq 0} \subset Y_2$ to the linearised differential equation \begin{equation} \label{eq: definition of the difference term} \xi_0=\nu-\mu; \hspace{0.5cm} \partial_t \xi_t = 2Q(\phi_t(\mu), \xi_t).\end{equation}  This solution satisfies the bounds \begin{equation} \label{eq: stability for BE 1}  \|\phi_t(\nu)-\phi_t(\mu)\|_{\mathrm{TV}+2} \leq C e^{-\lambda_0 t/2} \Lambda_{k}(\mu, \nu)^\frac{1}{2}\|\mu-\nu\|_\mathrm{TV}^\eta; \end{equation} \begin{equation} \label{eq: stability for BE 1.5}  \|\xi_t\|_{\mathrm{TV}+2} \leq C e^{-\lambda_0 t/2} \Lambda_{k}(\mu, \nu)^\frac{1}{2}\|\mu-\nu\|_\mathrm{TV}^\eta;\end{equation} \begin{equation}\label{eq: stability for BE 2}
\|\phi_t(\nu)-\phi_t(\mu) - \xi_t \|_{\mathrm{TV}+2} \leq C e^{-\lambda_0 t/2} \Lambda_{k}(\mu, \nu)^\frac{1}{2}\|\mu-\nu\|_\mathrm{TV}^{1+\eta}.\end{equation}  This allows us to define a linear map $\mathcal{D}\phi_t(\mu)$ by \begin{equation} \mathcal{D}\phi_t(\mu)[\nu-\mu]:=\xi_t.\end{equation} This linear map will play the r\^ole of a functional derivative for the Boltzmann flow $\phi_t$ in the calculus developed by \cite{M+M}. \end{proposition} 
To obtain estimates with the weighted metric $W$, we will use a version of Proposition \ref{thrm: stability for BE} with the difference $\phi_t(\mu)-\phi_t(\nu)$ measured in stronger norms $\|\cdot \|_{\mathrm{TV}+q}$. The following estimate may be obtained by a simple interpolation between Propositions \ref{thrm:momentinequalities}, \ref{thrm: stability for BE}. \begin{corollary} \label{cor: new stability for BE} Let $q\geq 2$, $\eta \in (0,1)$ and $\lambda<\lambda_0$. Then for all $k$ large enough, depending on $\eta, \lambda$ and $q$, there exists a constant $C$ such that \begin{equation} \forall \mu, \nu \in \mathcal{S}^k, \hspace{1cm}\|\phi_t(\mu)-\phi_t(\nu)\|_{\mathrm{TV}+q} \leq C e^{-\lambda t/2} \Lambda_{k}(\mu, \nu)^\frac{1}{2} \|\mu-\nu\|_\mathrm{TV}^{\eta}. \end{equation} \end{corollary} We emphasise that the rapid decay is the key property that allows us to obtain good long-time behaviour for our estimates. The pointwise estimate Theorem \ref{thrm: PW convergence} and the initial estimate for pathwise local uniform convergence Lemma \ref{lemma: initial LU bound} would hold for estimates  \begin{equation} \label{eq: weaker stability 1}  \|\phi_t(\nu)-\phi_t(\mu)\|_{\mathrm{TV}+5} \leq F(t) \Lambda_{k}(\mu, \nu)^\frac{1}{2}\|\mu-\nu\|_\mathrm{TV}^\eta; \end{equation}  \begin{equation}\label{eq: weaker stability 2}
\|\phi_t(\nu)-\phi_t(\mu) - \xi_t \|_{\mathrm{TV}+2} \leq G(t) \Lambda_{k}(\mu, \nu)^\frac{1}{2}\|\mu-\nu\|_\mathrm{TV}^{1+\eta}\end{equation} for functions $F,G$ such that \begin{equation} \label{eq: weaker stability 3} \left(\int_0^\infty F^2 dt\right)^{1/2}<\infty;\hspace{0.5cm} \int_0^\infty G dt<\infty. \end{equation} The full strength of exponential decay is used to `bootstrap' to the pathwise local uniform estimate Theorem \ref{thrm: Main Local Uniform Estimate}, which provides better behaviour in the time horizon $t_\text{fin}$, with only a logarithmic loss in the number of particles $N$. Provided that $F\rightarrow 0$ as $t\rightarrow \infty$, we could use the same `bootstrap', but with a potentially much larger loss in $N$. 
\subsection{Regularity Estimates}

 For the proof of the local uniform estimate Lemma \ref{thrm: local uniform martingale control}, it will be important to control the continuity of $Q$ \emph{after application of the flow maps} $\phi_t$; for brevity, we will write the composition as $Q_t=Q\circ \phi_t$. We can exploit the use of the stronger $\|.\|_{\mathrm{TV}+2}-$ norm in the stability estimates Proposition \ref{thrm: stability for BE}, to prove a strong notion of continuity for $Q_t$, including the dependence on $t$. \medskip \\ It is well known that, for $q\geq 1$, and $\mu, \nu \in \mathcal{S}^{q+1}$, we have the bilinear estimate \begin{equation} \|Q(\mu)-Q(\nu)\|_{\mathrm{TV}+q}\lesssim \Lambda_{q+1}(\mu, \nu)^\frac{1}{2}\|\mu-\nu\|_{\mathrm{TV}+(q+1)} \end{equation} and, by interpolating, this leads to \begin{equation} \label{eq: holder continuity of Q} \|Q(\mu)-Q(\nu)\|_{\mathrm{TV}+q}\lesssim \Lambda_{3(q+1)}(\mu, \nu)^\frac{1}{2}\|\mu-\nu\|_{\mathrm{TV}}^\frac{1}{2}. \end{equation}  Combining this the stability estimate in Corollary \ref{cor: new stability for BE}, we deduce the following. For  $q\geq 1$, $\eta \in (0,1)$ and $\lambda<\lambda_0$, then there exists $k$ such that, for $\mu, \nu \in \mathcal{S}^k$, we have the estimate \begin{equation} \label{eq: Lipschitz continuity of Q}\|Q_t(\mu)-Q_t(\nu)\|_{\mathrm{TV}+q} \lesssim e^{-\lambda t} \Lambda_{k}(\mu, \nu)^\frac{1}{2}\|\mu-\nu\|_\mathrm{TV}^\eta.\end{equation}
\section{Proof of Theorem \ref{thrm: W-W continuity of phit}}\label{sec: continuity of BE}  As a first application of the stability estimates, we will now prove Theorem \ref{thrm: W-W continuity of phit}, which establishes a continuity result for the Boltzmann flow $(\phi_t)$ with respect to our weighted Wasserstein metric $W$. For Theorems \ref{thrm: PW convergence}, \ref{thrm: Main Local Uniform Estimate}, we wish to approximate a given starting point $\mu_0$ by an empirical measure $\mu^N_0 \in \mathcal{S}_N$ on $N$ points; in this context, the total variation distance is too strong, as there is no discrete approximation to any continuous measure $\mu_0$. We therefore seek a continuity estimate for the Boltzmann flow $\phi_t$, measured in the Wasserstein distance $W$ defined in (\ref{eq: definition of W}), and which is uniform in time. \medskip \\  The proof combines a representation formula, and associated estimates, from \cite{ACE}, which establishes the first claim; the second claim will then follow using a long-time estimate recalled in Proposition \ref{thrm: stability for BE}. We will first review the definition, and claimed representation formula for the Boltzmann flow. \begin{definition}[Linearised Kac Process]\label{def: LKP} Write  $V=\mathbb{R}^d$ and $V^*$ for the signed space $V^*=V\times\{\pm 1\}=V^+\sqcup V^-$. We write $\pi: V^*\rightarrow V$ as the projection onto the first factor, and $\pi_\pm: V^\pm\rightarrow V$ for the obvious bijections.  \\ Let $(\rho_t)_{t\ge 0}$ be family of measures on $V=\mathbb{R}^d$ such that \begin{equation} \langle 1, \rho_t \rangle =1;\hspace{1cm} \langle |v|^2, \rho_t\rangle =1;\end{equation}  \begin{equation} \label{eq: integrability for environment}  \int_0^t \Lambda_3(\rho_s)ds <\infty \hspace{1cm} \text{for all }t<\infty.  \end{equation} The \emph{Linearised Kac Process} \emph{in environment $(\rho_t)_{t\ge 0}$} is the branching process on $V^*$ where each particle of type $(v,1)$, at rate $2|v-v_\star|\rho(dv_\star) d\sigma$, dies, and is replaced by three particles, of types \begin{equation} (v'(v,v_\star,\sigma),1);\hspace{0.5cm}(v_\star'(v,v_\star, \sigma),1);\hspace{0.5cm}(v_\star,-1) \end{equation} where $v', v_\star'$ are the post-collisional velocities given by (\ref{eq: PCV}). The dynamics are identical for particles of type $(v,-1)$, with the signs exchanged. \medskip \\ We write $\Xi^*_t$ for the associated process of unnormalised empirical measures on $V^*$, and define a signed measure $\Xi_t$ on $V$ by including the sign at each particle: \begin{equation} \Xi_t=\Xi^+_t-\Xi^-_t ; \hspace{1cm} \Xi^\pm_t=\Xi^\star_t\circ \pi_\pm^{-1}.\end{equation} We can also consider the same branching process, started from a time $s\ge 0$ instead. We write $E$ for the expectation over the branching process, which is not the full expectation in the case where $\rho$ is itself random. When we wish to emphasise the initial velocity $v$ and starting time $s$, we will write $E_{(s,v)}$ when the process is started from $\Lambda^*_0=\delta_{(v,1)}$ at time $s$, and $E_v$ in the case $s=0$. \end{definition} Provided that the initial data $\Xi_0$ is finitely supported, one can show that the branching process is almost surely non-explosive, and that \begin{equation}\label{eq: no explosion}E_{v_0} \langle 1+|v|^2, |\Xi_t|\rangle \le (1+|v|^2)\exp\left[8\int_0^t \Lambda_3(\rho_s) ds\right]. \end{equation}   \begin{remark} We can connect this branching process with a different proof of existence and uniqueness for the difference $\xi_t$ in Theorem \ref{thrm: stability for BE}. For existence, consider the linearised Kac process $(\Xi_t)_{t\ge 0}$ in environment $\rho_t=\phi_t(\mu)$, where particles are initialised at $t=0$ according to a Poisson random measure of intensity  \begin{equation} \theta(dv)= \begin{cases} \xi_0^+(dv)=\nu(dv) & \text{on }V^+  \\ \xi_0^-(dv)=\mu(dv) & \text{on }V^-. \end{cases}  \end{equation} Let $\xi_t = \mathbb{E}(\Xi_t)$, which may be formalised in the sense of a Bochner integral in the weighted space $(Y_2, \|\cdot\|_{\mathrm{TV}+2})$ defined in (\ref{def: weighted normed spaces}). Then the same proof of the representation formula \cite[Proposition 4.2]{ACE} shows that $\partial_t \xi_t = 2Q(\phi_t(\mu), \xi_t)$, and that this solution is unique. \end{remark} Recall from the introduction that $\mathcal{A}$ is the set of all functions $f$ on $\mathbb{R}^d$, such that $\widehat{f}(v)=(1+|v|^2)^{-1}f(v)$ satisfies \begin{equation} |\widehat{f}(v)|\le 1; \hspace{1cm} \frac{|f(v)-f(w)|}{|v-w|}\le 1 \hspace{1cm}\text{for all }v\neq w. \end{equation} From the bound (\ref{eq: no explosion}), we can now define, for functions of quadratic growth, \begin{equation} \label{eq: defn of f0t} f_{st}(v_0)=E_{(s,v_0)}\left[\langle f, \Xi_t\rangle \right].\end{equation} When we wish to emphasise the environment, we will write $f_{st}[\rho](v_0)$. We now recall the following estimates from \cite{ACE}:

\begin{proposition}[Continuity Estimates for $f_{st}$]\label{prop: continuity for branching process} Fix $t\ge 0$, and let $z_t$ be given by \begin{equation} z_t=3\exp\left[8\int_0^t  \Lambda_3(\rho_u)du \right].\end{equation}  Then, for $f\in\mathcal{A}$ and $s\le t$, we have $f_{st} \in z_t\hspace{0.1cm} \mathcal{A}$.  This is, in our notation, a reformulation of \cite[Propositions 4.3]{ACE}. \end{proposition}  The other result which we will use is the representation formula \cite[Proposition 4.2]{ACE}, which expresses the difference of two Boltzmann flows $\phi_t(\mu)-\phi_t(\nu)$ in terms of the functions $f_{0t}$. This may be obtained from the proof of \cite[Proposition 4.2]{ACE} without essential modification, as in the proof of \cite[Theorem 10.1]{ACE}. \begin{proposition}[Representation Formula]\label{prop: bad representation formula} Let $\mu, \nu \in \mathcal{S}^k$ for some $k>2$, and let $(\rho_t)_{t\ge 0}$ be given by \begin{equation} \rho_t=\frac{1}{2}(\phi_t(\mu)+\phi_t(\nu)) \end{equation} where $\phi_t(\mu)$ is the unique, locally $\mathcal{S}^k$-bounded solution to the Boltzmann equation, starting at $\mu$, and similarly for $\nu$. Then, for all $f\in \mathcal{A}$, we have \begin{equation} \label{eq: bad representation formula} \langle f, \phi_t(\mu)-\phi_t(\nu)\rangle =\left\langle f_{0t}[\rho], \mu-\nu\right\rangle.\end{equation}  \end{proposition} Note that the moment production property in Proposition \ref{thrm:momentinequalities} guarantees that (\ref{eq: integrability for environment}) holds for this environment. This will allow us to find an estimate for the Boltzmann flow $\phi_t$ which behaves well in short time. We now give the proof of Theorem \ref{thrm: W-W continuity of phit} \begin{proof}[Proof of Theorem \ref{thrm: W-W continuity of phit}] From the representation formula (\ref{eq: bad representation formula}) and continuity estimate Proposition \ref{prop: continuity for branching process}, for any $f\in \mathcal{A}$, \begin{equation} \label{eq: short time bound on BE} \langle f, \phi_t(\mu)-\phi_t(\nu)\rangle =\langle f_{0t}[\rho],\mu-\nu\rangle \le z_t\hspace{0.1cm} W(\mu, \nu) \end{equation} where $\rho_t=(\phi_t(\mu)+\phi_t(\nu))/2$. It therefore suffices to bound \begin{equation} z_t:=3\exp\left(4\int_0^t  \left[\Lambda_3(\phi_s(\mu))+\Lambda_3(\phi_s(\nu))\right] ds \right).\end{equation} Using the logarithmic moment production for the Boltzmann equation recalled in Proposition \ref{thrm:momentinequalities}, there exist constants $k,w$ such that  \begin{equation} \begin{split} z_t & \lesssim e^{wt} \Lambda_{k/2}(\mu)\Lambda_{k/2}(\nu) \\ &\hspace{0.5cm} \lesssim e^{wt} \Lambda_{k/2}(\mu, \nu)^2 \lesssim e^{wt}\Lambda_k(\mu, \nu).\end{split} \end{equation} This proves the first claim. For the second claim, we first deal with the case where $k\ge 3$ is large enough that the above holds, and such that the stability estimate Proposition \ref{thrm: stability for BE} holds with H\"older exponent $\eta=\frac{1}{2}$. Fix $\mu, \nu \in \mathcal{S}^k_a$, and assume without loss of generality that $0<W(\mu, \nu)<1$. From the stability estimate (\ref{eq: stability for BE 1}) we have \begin{equation} \|\phi_t(\mu)-\phi_t(\nu)\|_{\mathrm{TV}+2} \lesssim a^\frac{1}{2} e^{-\lambda_0t/2} \hspace{0.05cm} \end{equation} for some constants $\lambda_0>0$. It is immediate from the definitions that \begin{equation} W(\mu, \nu)\le \|\mu-\nu\|_{\mathrm{TV}+2} \end{equation} and so combining with the previous result, we have\begin{equation} W\left(\phi_t(\mu), \phi_t(\nu)\right)\lesssim a \min\left(e^{-\lambda_0 t/2}, W(\mu, \nu)e^{wt}\right).\end{equation} The right hand side is maximised when $e^{-\lambda_0 t/2}=W(\mu, \nu)e^{wt}$, which occurs when \begin{equation} t=-\frac{2}{\lambda_0+2w} \hspace{0.1cm} \log \hspace{0.05cm}W(\mu, \nu).\end{equation} Therefore, the maximum value of the right-hand side is \begin{equation} \label{eq: holder continuity} \begin{split} \sup_{t\ge 0} W\left(\phi_t(\mu), \phi_t(\nu)\right) & \lesssim a\exp\left(\frac{\lambda_0}{\lambda_0+2w} \log W(\mu, \nu)\right) \\ & =aW(\mu, \nu)^\zeta\end{split} \end{equation} with \begin{equation} \zeta(d)=\frac{\lambda_0}{\lambda_0+2w}\end{equation} which is the claimed H\"older continuity, for $k$ sufficiently large. \medskip \\ Finally, we deal with the second point for arbitrary $k>2$. This argument uses a localisation principle to control the moments on a very short initial interval $[0,u]$, and may be read as a warm-up to the more involved arguments in the proof of Theorem \ref{thm: low moment regime}. \medskip \\ Let $k_0$ be large enough such that the estimate (\ref{eq: holder continuity}) holds, and let $\zeta_0$ be the resulting exponent. Let $\beta=\frac{k-2}{2}$, let $\mu, \nu$ be as in the statement of the result, and let $u\in (0,1]$ be chosen later. Define \begin{equation} T=\inf\left\{t\ge 0: \Lambda_3(\rho_t)>\frac{\beta t^{\beta-1}+1}{2}\right\}\end{equation} where $\rho_t$ is as above. We now deal with the two cases $T>u, T\le u$ separately. \medskip \\ If $T>u$, then we have the estimate \begin{equation} \begin{split} z_u&:=3\exp\left(4\int_0^u \Lambda_3(\rho_s)ds\right) \\ & \le 3\exp\left(4\int_0^1 \frac{\beta s^{\beta-1}+1}{2} ds\right)\lesssim 1. \end{split} \end{equation} Using the representation formula in Proposition \ref{prop: bad representation formula} as in (\ref{eq: short time bound on BE}), we therefore obtain \begin{equation} \label{eq: v short time BE estimate} \sup_{t\le u} W(\phi_t(\mu),\phi_t(\nu)) \lesssim W(\mu, \nu).\end{equation} Using (\ref{eq: holder continuity}) on $\phi_u(\mu), \phi_u(\nu)$, and using the moment production property recalled in Proposition \ref{thrm:momentinequalities}, we have the estimate \begin{equation} \label{eq: restarted BF estimate}\sup_{t\ge u} W(\phi_t(\mu),\phi_t(\nu)) \lesssim u^{2-k_0}W(\mu, \nu)^{\zeta_0}. \end{equation} We next deal with the case $T\le u$. In this case, comparing the moment production property to the definition of $T$ shows that \begin{equation}   T^{\beta-1}\lesssim \Lambda_3(\phi_T(\mu))+\Lambda_3(\phi_T(\nu))\lesssim a T^{k-3} ;\hspace{1cm} T\le u\end{equation} which rearranges to produce the bound $1\lesssim au^{k/2-1}$. In particular, in this case, we have \begin{equation} \label{eq: bad moment case} \sup_{t\ge 0} W(\phi_t(\mu),\phi_t(\nu))\le 4 \lesssim a u^{k/2-1}. \end{equation} Combining estimates (\ref{eq: v short time BE estimate}, \ref{eq: restarted BF estimate}, \ref{eq: bad moment case}), we see that in all cases, \begin{equation} \sup_{t\ge 0} W(\phi_t(\mu), \phi_t(\nu)) \lesssim u^{2-k_0}W(\mu,\nu)^{\zeta_0}+au^{k/2-1}.\end{equation} Now, if we choose $u=\min(1,W(\mu, \nu)^\delta) $ for sufficiently small $\delta>0$, we obtain \begin{equation} \sup_{t\ge 0} W(\phi_t(\mu), \phi_t(\nu)) \lesssim aW(\mu,\nu)^\zeta \end{equation} for a new exponent $\zeta=\zeta(d,k)>0$. \end{proof}

\section{The Interpolation Decomposition for Kac's Process}\label{sec: interpolation decomposition} 
We introduce a pair of random measures associated to the Markov process $(\mu^N_t)_{t\geq 0}$. The \emph{jump measure} $m^N$ is the un-normalised empirical measure on $(0,\infty) \times \mathcal{S}_N$, of all pairs $(t, \mu^N)$, such that the system collides at time $t$, with new measure $\mu^N$. Its \emph{compensator} $\overline{m}^N$ is the random measure on $(0, \infty)\times \mathcal{S}_N$ given by \begin{equation}\label{eq: definition of mbar} \overline{m}^N(dt,d\mu^N)=\mathcal{Q}_N(\mu^N_{t-}, d\mu^N)dt \end{equation} where $\mathcal{Q}_N(\cdot, \cdot)$ is the transition kernel of the Kac process, given by (\ref{eq: definition of script Q}). The goal of this section is to prove the following `interpolation decomposition' for the difference between Kac's process and the Boltzmann flow, which is the key identity required for the proofs of Theorems \ref{thrm: PW convergence}, \ref{thrm: Main Local Uniform Estimate}. This is based on an idea of Norris \cite{PDF}, which was inspired by \cite[Section 3.3]{M+M}.\begin{formula}\label{form:newdecomposition}  Let $\mu^N_t$ be a Kac process on $N\geq 2$ particles, and suppose $f \in \mathcal{A}_0$ is a test function. To ease notation, we write \begin{equation} \label{eq: definition of Delta} \Delta(s,t,\mu^N)=\phi_{t-s}(\mu^N)-\phi_{t-s}(\mu^N_{s-}); \hspace{0.5cm} 0\le s \le t, \hspace{0.2cm}\mu^N \in \mathcal{S}_N; \end{equation} \begin{equation} \label{eq: definition of psi} \psi(u,\mu, \nu)= \phi_{u}(\nu)-\phi_{u}(\mu)-\mathcal{D}\phi_{u}(\mu)[\nu-\mu]; \hspace{0.3cm} u\ge 0,\hspace{0.3cm} \mu, \nu \in \bigcap_{k>2} \mathcal{S}^k\end{equation} where $\mathcal{D}\phi_t$ is the derivative of the Boltzmann flow $\phi_t$, defined in Proposition \ref{thrm: stability for BE}; this makes sense, provided that all moments of $\mu, \nu$ are finite. Then we can decompose \begin{equation}  \langle f, \mu^N_t -\phi_t(\mu^N_0) \rangle =M^{N,f}_t + \int_0^t \langle f, \rho^N(t-s, \mu^N_s) \rangle ds   \end{equation} where \begin{equation} \label{eq: definition of MNFT} M^{N,f}_t=\int_{(0,t]\times \mathcal{S}_N} \langle f, \Delta(s,t,\mu^N) \rangle (m^N-\overline{m}^N)(ds, d\mu^N_s)\end{equation} and where $\rho^N$ is given in terms of the transition kernel $\mathcal{Q}_N$ (\ref{eq: definition of script Q}) by \begin{equation} \label{eq: definition of rho} \langle f, \rho^N(u, \mu^N)\rangle  = \int_{\mathcal{S}_N}\langle f, \psi(u,\mu^N, \nu) \rangle \mathcal{Q}_N(\mu^N, d\nu).\end{equation} \end{formula} \begin{remark} \begin{enumerate}[label=\roman{*}).] \item This is the key identity needed for Theorems \ref{thrm: PW convergence}, \ref{thrm: Main Local Uniform Estimate}; the remainder of the proofs are to establish suitable controls over each of the two terms. 
\item This representation formula offers two major advantages over the equivalent representation formula in \cite{ACE}, which will be recalled in Proposition \ref{prop: very bad rep formula}.  \begin{itemize} \item Firstly, all the quantities appearing in our formula are adapted to the natural filtration of $(\mu^N_t)_{t\ge 0}$, and so we can use martingale estimates directly; by contrast, \cite[Proposition 4.2]{ACE} contains anticipating terms. This allows us to prove convergence in $L^p$ spaces, rather than simply in probability.
\item Secondly, all terms appearing in our formula may be controlled by the stability estimates (\ref{eq: stability for BE 1}, \ref{eq: stability for BE 2}). This allows us to exploit the stability of the limit equation, at the level of \emph{individual realisations of} the empirical particle system $\mu^N_0$.  \end{itemize}  \end{enumerate} \end{remark}  The main technicality in the proof of this is to derive a Chapman-Kolmogorov-style equation, which allows us to manipulate the functional derivatives $\mathcal{D}\phi_t$. This is the content of the following lemma.
\begin{lemma}[Exchange Lemma]\label{lemma:DAP} Let $\mu^N \in \mathcal{S}_N$ and $f\in \mathcal{A}$. Then for all times $t\geq 0$, we have the equalities \begin{equation}\label{eq:exchangeDandI} \begin{split} &\frac{d}{dt}\langle f, \phi_t(\mu^N)\rangle = \langle f,\mathcal{D}\phi_t(\mu^N)\left[Q(\mu^N)\right]\rangle \\& \hspace{1cm}= \int_{\mathbb{R}^d \times \mathbb{R}^d \times S^{d-1}} \hspace{0.05cm} \langle f,\mathcal{D}\phi_t(\mu^N)[\mu^{N, v, v_\star, \sigma}-\mu^N]\rangle \hspace{0.05cm} |v-v_\star|\hspace{0.05cm}N\hspace{0.05cm}d\sigma \mu^N(dv)\mu^N(dv_\star)\end{split} \end{equation} where $\mu^{N,v,v_\star,\sigma}$ is the post-collision measure given by (\ref{eq: change of measure at collision}), $\mathcal{Q}_N$ is the generator of the Kac process (\ref{eq: definition of script Q}) and where $\mathcal{D}\phi_t$ is the functional derivative given by Proposition \ref{thrm: stability for BE}.  \end{lemma}The first equality is familiar from semigroup theory, but is complicated by the non-linearity of the flow maps; we resolve this by using ideas of the infinite dimensional differential calculus developed in \cite{M+M}. The second equality can be thought of as a continuity property for the linear map $\mathcal{D}\phi_t(\mu^N)[\cdot]$, and is justified in Lemma \ref{lemma:DAP} by the explicit construction of the derivative in Proposition \ref{thrm: stability for BE}.\medskip \\ Assuming this for the moment, we now prove the interpolation decomposition Formula \ref{form:newdecomposition}. \begin{proof}[Proof of Formula \ref{form:newdecomposition}] To begin with, we restrict to bounded, measurable $f$. Fix $t\geq 0$, and consider the process $\Gamma^{N,f,t}_s=\langle f, \phi_{t-s}(\mu^N_s)\rangle$, for $0\leq s \leq t$. Then $\Gamma^{N,f,t}$ is c\`{a}dl\`{a}g, and is differentiable on intervals where $\mu^N_s$ is constant. On such intervals, Lemma \ref{lemma:DAP} tells us that 
\begin{multline}\begin{split} \frac{d}{ds} & \langle f, \phi_{t-s}(\mu^N_s)\rangle =-\left.\frac{d}{du}\right|_{u=t-s}\langle f, \phi_u(\mu^N_s)\rangle\\[1ex] & = -\int_{\mathbb{R}^d\times \mathbb{R}^d\times S^{d-1}}  \langle f, \mathcal{D}\phi_{t-s}(\mu^N_s)[\mu^{N,v,v_\star, \sigma}-\mu^N_s]  \rangle  |v-v_\star|\hspace{0.05cm}N\hspace{0.05cm} \mu^N_s(dv)\mu^N_s(dv_\star)d\sigma \\[1ex] & = - \int_{\mathcal{S}_N}  \langle f, \mathcal{D}\phi_{t-s}(\mu^N_s)[\mu^N-\mu^N_s]  \rangle \mathcal{Q}_N(\mu^N_s, d\mu^N)\end{split}\end{multline} 
where the final equality is to rewrite integral in terms of the transition kernel $\mathcal{Q}_N$ of the Kac process, defined in (\ref{eq: definition of script Q}). Writing $\mathcal{I}_t$ for the (finite) set of jumps $\mathcal{I}_t=\{s\le t: \mu^N_s \neq \mu^N_{s-}\}$, the contribution to $\Gamma^{N,f,t}_t-\Gamma^{N,f,t}_0$ from drift between jumps is \begin{equation} \begin{split}&\int_{(0,t]\setminus \mathcal{I}_t} \hspace{0.1cm}\frac{d}{ds}\langle \phi_{t-s}(\mu^N_s)\rangle \hspace{0.1cm} ds \\&=-\int_{((0,t]\setminus \mathcal{I}_t)\times\mathcal{S}_N} \langle f, \mathcal{D}\phi_{t-s}(\mu^N_s)[\mu^N-\mu^N_s]\rangle\mathcal{Q}_N(\mu^N_s, d\mu^N) ds.\end{split} \end{equation} Using the definitions (\ref{eq: definition of Delta}, \ref{eq: definition of psi}) of $\psi$ and $\Delta$, the integrand can be expressed as \begin{equation} \langle f, \mathcal{D}\phi_{t-s}(\mu^N_s)[\mu^N-\mu^N_s]\rangle = \langle f,\Delta(s,t,\mu^N)-\psi(t-s, \mu^N_s, \mu^N)\rangle \end{equation} for any $s \not \in \mathcal{I}_t.$ Since the set $\mathcal{I}_t$ has $0$ Lebesgue measure, the set $\mathcal{I}_t\times\mathcal{S}_N$ has $0$ measure with respect to $\mathcal{Q}_N(\mu^N_s,d\mu^N)ds$, and so the inclusion of this set does not change the integral. Using the definitions (\ref{eq: definition of mbar}, \ref{eq: definition of rho}) of $\overline{m}^N$ and $\rho^N$, we can rewrite the integral as \begin{equation} \label{eq: contribution from drift} \begin{split} &\int_{(0,t]\times \mathcal{S}_N} \langle f, \psi(t-s, \mu^N_s,\mu^N)-\Delta(s, t, \mu^N)\rangle \mathcal{Q}_N(\mu^N_s, d\mu^N) ds \\ =&\int_0^t \langle f, \rho^N(t-s, \mu^N_s)\rangle ds - \int_{(0,t]\times \mathcal{S}_N} \langle f, \Delta(s,t,\mu^N)\rangle \overline{m}^N(ds, d\mu^N). \end{split}\end{equation}  
On the other hand, at the times when $\mu^N_s$ jumps, we have \begin{equation} \Gamma^{N,f,t}_s-\Gamma^{N,f,t}_{s-}=\langle f, \phi_{t-s}(\mu^N_s)-\phi_{t-s}(\mu^N_{s-})\rangle =\langle f, \Delta(s,t,\mu^N_s)\rangle. \end{equation}
Therefore, the contribution to $\Gamma^{N,f,t}_t-\Gamma^{N,f,t}_0$ from jumps is \begin{equation} \label{eq: contribution from jump} \begin{split} \sum_{s\in \mathcal{I}_t} \Gamma^{N,f,t}_s-\Gamma^{N,f,t}_{s-}= \int_{(0,t]\times \mathcal{S}_N}\langle f, \Delta(s,t,\mu^N)\rangle\hspace{0.1cm} m^N(ds, d\mu^N) \\ = M^{N,f}_t+\int_{(0,t]\times \mathcal{S}_N}\langle f, \Delta(s,t,\mu^N)\rangle\hspace{0.1cm} \overline{m}^N(ds, d\mu^N)  \end{split} \end{equation} Combining the contributions (\ref{eq: contribution from drift}, \ref{eq: contribution from jump}), we see that \begin{equation} \begin{split} \langle f, \mu^N_t-\phi_t(\mu^N_0)\rangle &= \Gamma^{N,f,t}_t-\Gamma^{N,f,t}_0 \\[1ex]& = \int_{(0,t]\setminus \mathcal{I}_t} \hspace{0.1cm}\frac{d}{ds} \langle f, \phi_{t-s}(\mu^N_s)\rangle ds + \sum_{s\in \mathcal{I}_t} \Gamma^{N,f,t}_s-\Gamma^{N,f,t}_{s-} \\[1ex] & = M^{N,f}_t+\int_0^t\langle f ,\rho^N(t-s, \mu^N_s)ds  \end{split} \end{equation}  as desired. \end{proof}
\subsection{Proof of Lemma \ref{lemma:DAP}} In this subsection, we will prove the Chapman-Kolmogorov property Lemma \ref{lemma:DAP}, which is crucial to the interpolation decomposition. We prove the two claimed equalities separately. 
\begin{lemma} Let $N\ge 2$ and let $\mu^N \in \mathcal{S}_N$. Then, for all $t>0$ and $f\in \mathcal{A}$, we have the differentiability  \begin{equation} \label{eq: first equality of CK} \frac{d}{dt}\langle f, \phi_t(\mu^N)\rangle = \langle f, \mathcal{D}\phi_t(\mu^N)[Q(\mu^N)]\rangle.  \end{equation} At $t=0$, this is a one-sided, right differentiability.  \end{lemma} The following proof uses ideas of \cite{M+M}, notably the concept of the infinite-dimensional differential calculus and building on ideas of \cite[Lemma 2.11]{M+M}. \begin{proof} Throughout, fix $\mu^N\in \mathcal{S}_N$ and $f\in \mathcal{A}$. Recall, for clarity, the notation $Q_t(\mu)=Q(\phi_t(\mu))$. Using the boundedness of appropriate moments of $\mu^N\in\mathcal{S}_N$, together with the continuity estimate (\ref{eq: holder continuity of Q}), it is straightforward to see that the map $t\mapsto Q_t(\mu^N)$ is H\"older continuous in time, with respect to the weighted norm $\|\cdot\|_{\mathrm{TV}+2}$: for some constant $C_1=C_1(N)$, we have the estimate \begin{equation}\label{eq: time Holder continuity of Qt} \|Q_t(\mu^N)-Q_s(\mu^N)\|_{\mathrm{TV}+2} \le C_1 |t-s|^\frac{1}{2}. \end{equation} From the definition (\ref{BE}) of the Boltzmann dynamics, together with dominated convergence, we have that \begin{equation} \langle f, \phi_t(\mu^N_0)\rangle =\langle f, \mu^N\rangle +\int_0^t \langle f, Q_s(\mu^N)\rangle ds. \end{equation} Therefore, the map $t\mapsto \langle f, \phi_t(\mu^N)\rangle$ is continuously differentiable in time, with derivative \begin{equation} \frac{d}{dt} \langle f, \phi_t(\mu^N)\rangle=\langle f, Q_t(\mu^N)\rangle\end{equation} where, at $t=0$, this is a one-sided, right derivative. It therefore suffices to show that (\ref{eq: first equality of CK}) holds as a \emph{right} derivative.  \medskip \\ Fix $t\ge 0$, and observe that, for $s>0$ small enough, $\nu^N_s=\mu^N+sQ(\mu^N)$ defines a measure $\nu^N_s\in \mathcal{S}$. From the semigroup property, it follows that  $\phi_t(\phi_s(\mu^N))=\phi_{t+s}(\mu^N)$, and we can therefore expand \begin{equation} \begin{split} &\big\langle f, \phi_{t+s}(\mu^N)-\phi_t(\mu^N)-s\mathcal{D}\phi_t(\mu^N)[Q(\mu^N)]\big\rangle\\ &=\underbrace{\langle f, \phi_t(\phi_s(\mu^N))-\phi_t(\nu^N_s)\rangle}_{:=\mathcal{T}_1(s)}  + \underbrace{\langle f,\phi_t(\nu^N_s)-\phi_t(\mu^N)-s\mathcal{D}\phi_t(\mu)[Q(\mu^N)]\rangle}_{:=\mathcal{T}_2(s)}.  \end{split}\end{equation} We will now show that each of the two terms $\mathcal{T}_1, \mathcal{T}_2$ are $o(s)$, which implies the result.
\paragraph{Estimate on $\mathcal{T}_1(s)$} Let $\eta\in (\frac{2}{3},1)$, and choose $k$ large enough that the stability estimates (\ref{eq: stability for BE 1}, \ref{eq: stability for BE 2}) hold with exponent $\eta$. As $s\downarrow 0$, the probability measures $\nu^N_s =\mu^N+sQ(\mu^N)$ and $\phi_s(\mu^N)$ are bounded in $\mathcal{S}^k$. Therefore, from (\ref{eq: stability for BE 1}), there exists a constant $C_2=C_2(N)<\infty$ such that, for all $s>0$ small enough, \begin{equation} \label{eq: est of T1, 1} \|\phi_t(\phi_s(\mu))-\phi_t(\nu_s)\|_{\mathrm{TV}+2} \le C_2\|\phi_s(\mu)-\nu_s\|^\eta_{\mathrm{TV}+2}.\end{equation} The left-hand side is a bound for $\mathcal{T}_1(s)$. Using the estimate (\ref{eq: time Holder continuity of Qt}) above, we estimate the right-hand side, following \cite[Lemma 2.11]{M+M}: \begin{equation} \label{eq: est of T1, 2}\begin{split} \|\phi_s(\mu^N)-\nu^N_s\|_{\mathrm{TV}+2} &= \left\|\int_0^s (Q_u(\mu^N)-Q_0(\mu^N)) du\right\|_{\mathrm{TV}+2} \\ & \le \int_0^s \|Q_u(\mu^N)-Q_0(\mu^N)\|_{\mathrm{TV}+2} \hspace{0.1cm}du \\ & \le C_1(N)\int_0^s u^\frac{1}{2}du = \frac{2}{3}C_1(N) s^\frac{3}{2}.\end{split} \end{equation} Combining the estimates (\ref{eq: est of T1, 1}, \ref{eq: est of T1, 2}), we see that \begin{equation} \mathcal{T}_1(s) \le C_2\left(\frac{2}{3}C_1\right)^\eta s^\frac{3\eta}{2}.\end{equation} Since we chose $\eta>\frac{2}{3}$, this shows that $\mathcal{T}_1$ is $o(s)$ as $s\downarrow 0$. 

\paragraph{Estimate on $\mathcal{T}_2$} Let $\eta$ and $k$ be as above, and recall that in (\ref{eq: stability for BE 2}), $\xi_t$ is the definition of $\mathcal{D}\phi_t(\mu)[\nu-\mu]$. We now apply this estimate to $\mu^N$ and $\nu^N_s$, noting that $\nu^N_s=\mu^N+sQ(\mu^N)$ and $\phi_s(\mu^N)$ are bounded in $\mathcal{S}^k$ as $s\downarrow 0$, and that $\nu^N_s-\mu^N=sQ(\mu^N)$. The bound (\ref{eq: stability for BE 2}) now shows that, for some constants $C_3, C_4<\infty$, \begin{equation} \begin{split} \|\phi_t(\nu^N_s)-\phi_t(\mu^N)-s\mathcal{D}\phi_t(\mu^N)[Q(\mu^N)]\|_{\mathrm{TV}+2} &\le C_3\|\nu^N_s-\mu^N\|_\mathrm{TV}^{1+\eta} \\& = C_3\|sQ(\mu^N)\|_\mathrm{TV}^{1+\eta}\\ & \le C_4 s^{1+\eta}. \end{split}\end{equation} The left-hand side is a bound for $\mathcal{T}_2$, which implies that $\mathcal{T}_2$ is $o(s)$, as desired. Together with the previous estimate on $\mathcal{T}_1$, this concludes the proof.\end{proof} 

   
We now turn to the proof of the second equality in (\ref{eq:exchangeDandI}), that is, \begin{equation}\label{eq:DAP}\begin{split}&\langle f,\mathcal{D}\phi_t(\mu^N)\left[Q(\mu^N)\right]\rangle \\& = \int_{\mathbb{R}^d \times \mathbb{R}^d \times S^{d-1}} \hspace{0.05cm} \langle f,\mathcal{D}\phi_t(\mu^N)[\mu^{N, v, v_\star, \sigma}-\mu^N]\rangle \hspace{0.05cm}N\hspace{0.05cm} |v-v_\star|d\sigma \mu^N(dv)\mu^N(dv_\star). \end{split} \end{equation} Using the definition (\ref{eq: change of measure at collision}), we see that the integral on the right-hand side is equivalent to that defining $Q(\mu^N)$ in (\ref{eq: defn of Q}). However, we \emph{cannot} simply exchange the integration with the linear map $\mathcal{D}\phi_t$, as the construction in Proposition \ref{thrm: stability for BE} does not guarantee that $\mathcal{D}\phi_t(\mu^N)$ is bounded as a linear map. We will instead prove (\ref{eq:DAP}) from the \emph{explicit} way in which $\mathcal{D}\phi_t(\mu^N)$ is constructed in Proposition \ref{thrm: stability for BE}, and show that this construction implies `enough' continuity. \medskip\\ This is closely related to, and may be derived from, condition (\textbf{A3}), convergence of the generators, in \cite{M+M}. We present here a more direct proof, to avoid introducing additional spaces and notation. The crucial observation of our argument is that `enough' small perturbations of a discrete measure $\mu^N\in\mathcal{S}_N$ will remain in $\mathcal{S}$; this is made precise in equation (\ref{eq: proof of dap delta is small}). The same idea is present in the corresponding argument \cite[Section 5.5]{M+M}, but not made explicit. \medskip \\ Before turning to the proof of (\ref{eq:DAP}), we will prove the following auxiliary lemma. In order to justify the exchange of various integrals, we wish to improve the moments of the derivative $\xi_t=\mathcal{D}\phi_t(\mu)[\nu-\mu]$ in Proposition \ref{thrm: stability for BE}. The following argument combines ideas of \cite[Proposition 4.2]{ACE} and \cite[Lemma 6.3]{M+M}. \begin{lemma}\label{lemma: moments of xi} Suppose $\mu, \nu \in \cap_{k\ge 2} \mathcal{S}^k$, and let $(\xi_t)_{t\ge 0}$ be the solution to the differential equation (\ref{eq: definition of the difference term}). Then, for all $k\ge 2$, there exists a constant $c=c(k)$ such that, for all $t\ge 0$, \begin{equation} \|\xi_t\|_{TV+k}\le 2\Lambda_k(\mu,\nu) \exp\left(ct \Lambda_{k+1}(\mu)\right). \end{equation} Moreover, if $k'>2$ is large enough, then we have the continuity estimate, for all $0\le s \le t$, and for some absolute constants $C_1, C_2$, \begin{equation} \|\xi_t-\xi_s\|_{\mathrm{TV}+k}\le C_1\Lambda_{k+k'}(\mu, \nu)^\frac{1}{2}\exp\left(\frac{1}{2}C_2\Lambda_{2(k+1)}(\mu)t\right)(t-s)^\frac{1}{2}.\end{equation}    \end{lemma} \begin{proof} Firstly, we observe that, by hypothesis, the map $t\mapsto \xi_t$ is continuous in the norm $\|\cdot\|_{\mathrm{TV}+2}$, and is therefore locally bounded. We have the estimate on total variation \begin{equation}\label{eq: bound TV of Q} \|Q(\phi_t(\mu), \xi_t)\|_\mathrm{TV} \le 4\int_{\mathbb{R}^d\times\mathbb{R}^d}|v-v_\star|\phi_t(\mu)(dv)|\xi_t|(dv_\star) \le 8 \|\xi_t\|_{\mathrm{TV}+2}\end{equation} where we have used the bound $|v-v_\star|\le (1+|v|^2)(1+|v_\star|^2)$. Similarly, we estimate \begin{equation} \label{eq: continuity of partial t xi t}\begin{split} &\|Q(\phi_t(\mu), \xi_t)-Q(\phi_s(\mu), \xi_s)\|_\mathrm{TV} \\[0.5ex] & \hspace{1cm}\le \|Q(\phi_t(\mu)-\phi_s(\mu), \xi_t)\|_\mathrm{TV} +\|Q(\phi_s(\mu), \xi_t-\xi_s)\|_\mathrm{TV} \\[0.5ex] & \hspace{1cm} \le 4(\|\xi_t\|_{\mathrm{TV}+2}\hspace{0.05cm}\|\phi_t(\mu)-\phi_s(\mu)\|_{\mathrm{TV}+2}+2\hspace{0.05cm}\|\xi_t-\xi_s\|_{\mathrm{TV}+2}).  \end{split} \end{equation} Since $t\mapsto \phi_t(\mu)$ is continuous in $\|\cdot\|_{\mathrm{TV}+2}$, it follows that the map \begin{equation} t\mapsto \partial_t\xi_t=2Q(\phi_t(\mu), \xi_t) \end{equation} is continuous and locally bounded in $\|\cdot\|_{\mathrm{TV}}$. Therefore, for all $t\ge 0$, the measure $\pi_t=\int_0^t |\partial_s\xi_s| ds$ is a finite measure, and $\partial_s \xi_s$ is absolutely continuous with respect to $\pi_t$ for all $0\le s\le t$. Therefore, by a result of Norris \cite[Lemma 11.1]{ACE} on the time variation of signed measures, there exists a measurable map $f: [0,\infty)\times \mathbb{R}^d \rightarrow \{-1,0,1\}$ such that \begin{equation}\label{eq: definition of f}  \xi_t=f_t|\xi_t|;\hspace{1cm} |\xi_t|=|\xi_0|+\int_0^t f_s\partial_s\xi_s ds.\end{equation}Writing $\check{f}_s(v)=(1+|v|^k)f_s$, we have the bound \begin{equation} \begin{split} &\hspace{1cm}\label{eq: gain of integrability of xi} \langle 1+|v|^k, |\xi_t|-|\xi_0|\rangle\\  &=\int_0^t ds \int_{\mathbb{R}^d\times\mathbb{R}^d\times S^{d-1}} (\check{f}(v')+\check{f}(v_\star')-\check{f}(v_\star)-\check{f}(v))|v-v_\star|\hspace{0.05cm}\phi_s(\mu)(dv)\hspace{0.05cm}\xi_s(dv_\star)d\sigma \\[1ex]& \le \int_0^t ds \int_{\mathbb{R}^d\times \mathbb{R}^d \times S^{d-1}} (2+|v'|^k+|v_\star'|^k+|v_\star|^k-|v|^k)|v-v_\star|\hspace{0.05cm}\phi_s(\mu)(dv_\star)\hspace{0.05cm}|\xi_s|(dv)d\sigma.\end{split}  \end{equation} Now, there exists a constant $C_1=C_1(k)$ such that, for all $v, v_\star, \sigma$, we have the bound \begin{equation} |v'|^k+|v_\star'|^k+|v_\star|^k-|v|^k \le C_1(k)(|v|^{k-2}|v_\star|^2+|v_\star|^k)\end{equation} Therefore, for a different constant $C_2=C_2(k)$,  \begin{equation} 2|v-v_\star|(2+|v'|^k+|v_\star'|^k+|v_\star|^k-|v|^k) \le C_2(k)(1+|v|^k)(1+|v_\star|^{k+1}).\end{equation} Using the moment bounds in Proposition \ref{thrm:momentinequalities}, we obtain for some $c=c(k)$, \begin{equation} \begin{split} &  \langle 1+|v|^k, |\xi_t|\rangle  \le \langle 1+|v|^k, |\xi_0|\rangle \\& \hspace{2cm}+ C_2\int_0^t \int_{\mathbb{R}^d\times\mathbb{R}^d}(1+|v|^k)(1+|v_\star|^{k+1})|\xi_s|(dv)\phi_s(\mu)(dv_\star) \\[1ex] &\hspace{1cm} \le \langle1+|v|^k, |\xi_0|\rangle +c\Lambda_{k+1}(\mu)\int_0^t \langle 1+|v|^k, |\xi_s|\rangle ds.\end{split} \end{equation} Gr\"onwall's lemma now gives the claimed moment bound. For the continuity statement, if $k'$ is chosen large enough that (\ref{eq: stability for BE 1.5}) holds for some $\eta<1$, then (\ref{eq: bound TV of Q}) gives the bound \begin{equation}\label{eq: continuity in TV} \|Q(\phi_t(\mu), \xi_t)\|_\mathrm{TV} \le C_3\Lambda_{k'}(\mu, \nu) \end{equation} and therefore, for all $0\le s\le t$, \begin{equation} \label{eq: continuity in TV'} \|\xi_t-\xi_s\|_\mathrm{TV} \le C_3\Lambda_{k'}(\mu, \nu)(t-s).\end{equation} The continuity statement follows by combining (\ref{eq: continuity in TV'}) with the moment bound for $2k$, with the interpolation \begin{equation} \|\xi_t-\xi_s\|_{\mathrm{TV}+k}\le \|\xi_t-\xi_s\|^{1/2}_{\mathrm{TV}}\hspace{0.1cm}\|\xi_t+\xi_s\|_{\mathrm{TV}+2k}^{1/2} \end{equation} and using the correlation property (Lemma \ref{lemma: correlation of moments}) to absorb both moment terms. \end{proof} We can now prove the second claimed equality in Lemma \ref{lemma:DAP}. \begin{lemma} Let $\mu^N \in \mathcal{S}_N$, for $N\ge 2$. Then we have the equality \begin{equation}\label{eq:DAP2}\begin{split}&\hspace{1cm}\mathcal{D}\phi_t(\mu^N)\left[Q(\mu^N)\right] \\& = \int_{\mathbb{R}^d \times \mathbb{R}^d \times S^{d-1}} \hspace{0.05cm} \mathcal{D}\phi_t(\mu^N)[\mu^{N, v, v_\star, \sigma}-\mu^N] \hspace{0.05cm}N\hspace{0.05cm} |v-v_\star|d\sigma \mu^N(dv)\mu^N(dv_\star). \end{split} \end{equation} where the right hand side is a Bochner integral in the space $(Y_2,\|\cdot\|_{\mathrm{TV}+2})$. In particular, the equality (\ref{eq:DAP}) holds.  \end{lemma} \begin{proof} We exploit the fact that, for $\delta>0$ small enough, we have \begin{equation} \label{eq: proof of dap delta is small} \mu^N + \delta Q(\mu^N)\in \mathcal{S}; \hspace{0.5cm} \forall v, v_\star, \sigma, \hspace{0.2cm}\mu^N +\delta[\mu^{N,v,v_\star,\sigma}-\mu^N] \in \mathcal{S}. \end{equation} We will assume that $\delta>0$ is chosen so that this holds. For $v, v_\star \in \text{Supp}(\mu^N)$ and $\sigma \in S^{d-1}$, we define $\xi^{N,v,v_\star,\sigma}_t$ by the differential equation \begin{equation} \xi^{N,v,v_\star,\sigma}_0 = \delta [\mu^{N, v, v_\star, \sigma}-\mu^N]; \hspace{1cm} \partial_t \xi^{N,v,v_\star,\sigma}_t = 2Q(\phi_t(\mu^N),\xi^{N,v,v_\star,\sigma}_t). \end{equation}From Proposition \ref{thrm: stability for BE}, the solution to this equation exists, and is unique. By the characterisation of the derivative $\mathcal{D}\phi_t(\mu^N)$, we also have \begin{equation} \xi^{N,v,v_\star,\sigma}_t = \delta \hspace{0.1cm} \mathcal{D}\phi_t(\mu^N)[\mu^{N,v,v_\star,\sigma}-\mu^N]\end{equation} From Lemma \ref{lemma: moments of xi}, we also have a bound that $\| \xi^{N,v,v_\star,\sigma}_s\|_{\mathrm{TV}+4} \leq C$ for all $s\le t$, and for some constant $C=C(\mu^N,N,t)$ independent of $v, v_\star$ and $\sigma$. In this notation, we wish to establish the equality \begin{equation}\label{eq: desired equality for DAP} \mathcal{D}\phi_t(\mu^N)[Q(\mu^N)]\equalsquestion\int_{\mathbb{R}^d\times\mathbb{R}^d\times S^{d-1}} \xi^{N,v,v_\star,\sigma} |v-v_\star|\mu^N(dv)\mu^N(dv_\star) d\sigma.\end{equation} From the bound above, the right-hand side is well-defined as a Bochner integral in $(Y_2, \|\cdot\|_{\mathrm{TV}+2})$. \medskip \\ Firstly, arguing as in (\ref{eq: bound TV of Q}), for all $t \ge 0$, there is a constant $C=C(\mu^N, N, t)$ such that, for all $v, v_\star, \sigma$ and $s\le t$, we have   \begin{equation}\begin{split} &\|Q(\phi_t(\mu^N),\xi^{N,v,v_\star,\sigma}_s)\|_{\mathrm{TV}+3} \le C.\end{split}\end{equation} We now define \begin{equation} \xi_t = \int_{\mathbb{R}^d\times \mathbb{R}^d \times S^{d-1}} \xi^{N,v,v_\star,\sigma}_t |v-v_\star| d\sigma \hspace{0.1cm} \mu^N(dv) \mu^N (dv_\star) \end{equation} where the right-hand side is a Bochner integral in $(Y_3, \|\cdot\|_{\mathrm{TV}+3})$. From the definition (\ref{eq: defn of Q}) of $Q$, we have \begin{equation} \xi_0 = \delta \hspace{0.1cm}N^{-1}\hspace{0.1cm} Q(\mu^N)\end{equation} Moreover, using Fubini, we can express \begin{equation} \label{eq: use fubini}\begin{split} &\xi_t-\xi_0 \\[1ex] &=\int_{\mathbb{R}^d\times \mathbb{R}^d \times S^{d-1}} \left\{ \int_0^t 2Q(\phi_s(\mu^N),\xi^{N,v,v_\star,\sigma}_s) ds\right\} |v-v_\star| d\sigma \hspace{0.1cm} \mu^N(dv) \mu^N (dv_\star) \\[1ex]& =\int_0^t \left\{ \int_{\mathbb{R}^d\times \mathbb{R}^d \times S^{d-1}} 2Q(\phi_s(\mu^N),\xi^{N,v,v_\star,\sigma}_s)  |v-v_\star| d\sigma \hspace{0.1cm} \mu^N(dv) \mu^N (dv_\star) \right\} ds.\end{split}\end{equation} The same argument as in (\ref{eq: bound TV of Q}) shows that, for fixed $\mu\in \mathcal{S}^3$, the map \begin{equation} Q(\mu, \cdot): (Y_3, \|\cdot\|_{\mathrm{TV}+3})\rightarrow (Y_2, \|\cdot\|_{\mathrm{TV}+2});\hspace{1cm}\xi\mapsto Q(\mu, \xi)  \end{equation} is a bounded linear map. It follows that, for all $s \ge 0$,  \begin{equation} Q(\phi_s(\mu^N), \xi_s)=\int_{\mathbb{R}^d\times \mathbb{R}^d \times S^{d-1}} Q(\phi_s(\mu^N),\xi^{N,v,v_\star,\sigma}_s)  |v-v_\star| d\sigma \hspace{0.1cm} \mu^N(dv) \mu^N (dv_\star) \end{equation} as an equality of Bochner integrals in $(Y_2, \|\cdot\|_{\mathrm{TV}+2})$. Therefore, (\ref{eq: use fubini}) shows that, for all $t\ge 0$,\begin{equation}\label{eq: integral eqn for xi t} \xi_t=\xi_0+\int_0^t 2Q(\phi_s(\mu^N), \xi_s)ds. \end{equation} From Lemma \ref{lemma: moments of xi}, there exists a constant $C=C(\mu^N,N,t)$ such that, for all $v, v_\star, \sigma$ and $0\le s \le t$, \begin{equation} \|\xi^{N,v,v_\star, \sigma}_t-\xi^{N,v,v_\star, \sigma}_s\|_{\mathrm{TV}+2} \le C(t-s)^\frac{1}{2}\end{equation} and therefore, for a different constant $C'$,  \begin{equation} \|\xi_t-\xi_s\|_{\mathrm{TV}+2} \le C'(t-s)^\frac{1}{2}.\end{equation} By the same reasoning as (\ref{eq: continuity of partial t xi t}), we see that the map $t\mapsto 2Q(\phi_t(\mu^N),\xi_t)$ is continuous with respect to the norm $\|\cdot\|_{\mathrm{TV}+2}$, and so we may differentiate (\ref{eq: integral eqn for xi t}) to obtain $\partial_t \xi_t = 2Q(\phi_t(\mu^N),\xi_t)$. From Proposition \ref{thrm: stability for BE}, this uniquely characterises the derivative $\mathcal{D}\phi_t(\mu^N)[\delta \hspace{0.05cm}N^{-1}\hspace{0.05cm} Q(\mu^N)]$. Hence we have the claimed equality\begin{equation} \begin{split} &\mathcal{D}\phi_t(\mu^N)[Q(\mu^N)] =\delta^{-1} N \xi_t \\& =\delta^{-1} \int_{\mathbb{R}^d\times \mathbb{R}^d \times S^{d-1}} \xi^{N,v,v_\star,\sigma}_t |v-v_\star|\hspace{0.05cm}N\hspace{0.05cm} d\sigma \hspace{0.1cm} \mu^N(dv) \mu^N (dv_\star) \\& =\int_{\mathbb{R}^d\times \mathbb{R}^d \times S^{d-1}} \mathcal{D}\phi_t(\mu^N)[\mu^{N,v,v_\star,\sigma}-\mu^N] |v-v_\star|\hspace{0.05cm}N\hspace{0.05cm} d\sigma \hspace{0.1cm} \mu^N(dv) \mu^N (dv_\star). \end{split} \end{equation} \end{proof}

\section{Proof of Theorem \ref{thrm: PW convergence}} \label{sec: proof of pw}

The main difficulty in obtaining a pathwise statement is the martingale term $M^{N,f}_t$ in Formula \ref{form:newdecomposition}, which we defined above as \begin{equation} M^{N,f}_t = \int_{(0,t]\times \mathcal{S}_N} \bigg \langle f, \phi_{t-s}(\mu^N)-\phi_{t-s}(\mu^N_{s-})\bigg\rangle (m^N-\overline{m}^N)(ds, d\mu^N). \end{equation}  Recall the definition of $\mathcal{A}$ as those functions $f: \mathbb{R}^d \rightarrow \mathbb{R}$ satisfying \begin{equation} \forall \hspace{0.1cm} v, v'\in \mathbb{R}^d, \hspace{0.5cm}|\hat{f}(v)|\leq 1; \hspace{1cm}|\hat{f}(v)-\hat{f}(v')|\leq |v-v'|.\end{equation} We will be interested in controlling an expression of the form $\sup_{f\in\mathcal{A}} |M^{N,f}_t|$, either pointwise in time, or (pathwise) locally uniformly in time. However, unlike in the finite dimensional cases in \cite{D&N}, we cannot directly apply estimates from the elementary theory of martingales, as such estimates degrade in large dimensions. Instead, we will use the relative compactness  discussed in Remark \ref{rmk: compactness} to argue that this is an  \emph{effectively finite dimensional} problem. More precisely, we show that it can be approximated by a discretised, finite dimensional martingale approximation problem, with the following trade off: that making the truncation error small requires taking a large (finite) dimensional martingale. As in \cite{D&N,ACE}, the martingale term is `small', as a function of $N$, but will increase as a function of the dimension of the approximation. By optimising over the  discretisation, we will be able to balance the two terms to find a useful estimate on the family of processes. This is the same approach as used for an equivalent problem in \cite[Theorem 1.1]{ACE}. \medskip \\  Finding the best exponents of $N$ we have been able to obtain uses a `hierarchical decomposition'. This approach was inspired by an equivalent technique used in \cite[Proposition 7.1]{ACE}. 
 \begin{lemma} \label{thrm: pointwise martingale control} Let $\epsilon>0$, $a\ge 1$ and $0<\lambda<\lambda_0$. Let $k$ be large enough that Corollary \ref{cor: new stability for BE} holds with $q=4$, exponent $\lambda$ and H\"{o}lder exponent $1-\epsilon$.  \\ Let $(\mu^N_t)_{t\geq 0}$ be a Kac process in dimension $d\geq 3$, with initial moment $\Lambda_{k}(\mu^N_0) \leq a$. Let $M^{N,f}_t$ be the processes given by (\ref{eq: definition of MNFT}). Then we have, uniformly in $t\geq 0$, \begin{equation} \left\|\hspace{0.1cm} \sup_{f\in \mathcal{A}} \left|M^{N,f}_t\right| \hspace{0.1cm} \right\|_{L^2(\mathbb{P})} \lesssim a^{1/2} \hspace{0.1cm} N^{\epsilon-1/d}.\end{equation}  \end{lemma}  Once we have obtained the control of the martingale term, the remaining proof of Theorem \ref{thrm: PW convergence} is straightforward.  \begin{proof}[Proof of Theorem \ref{thrm: PW convergence}] Take $k=k(\epsilon)$ as in Lemma \ref{thrm: pointwise martingale control}, and such that Proposition \ref{thrm: stability for BE} holds with exponent $\max(1-\epsilon, \frac{1}{2})$. \medskip \\ We first note that it is sufficient to prove the case $\mu_0=\mu^N_0$. Given this case, we use the continuity established in Theorem \ref{thrm: W-W continuity of phit} to estimate the difference \begin{equation} W\left(\phi_t(\mu^N_0), \phi_t(\mu_0)\right) \lesssim a^{1/2} W(\mu^N_0, \mu_0)^\zeta\end{equation} for some $\zeta=\zeta(d,k)$, which implies the claimed result. \medskip \\ From now on, we assume that $\mu_0=\mu^N_0$. From the interpolation decomposition Formula \ref{form:newdecomposition}, we majorise \begin{equation} \label{eq: dominate pointwise bound} W\left(\mu^N_t, \phi_t\left(\mu^N_0\right)\right) \leq \sup_{f\in \mathcal{A}} \left|M^{N,f}_t\right| + \int_0^t \sup_{f\in \mathcal{A}} \hspace{0.1cm} \langle f, \rho^N(t-s, \mu^N_s)\rangle \hspace{0.1cm} ds \end{equation} where, as in (\ref{eq: definition of psi}, \ref{eq: definition of rho}), the integrand is given by \begin{equation} \langle f, \rho^N(t-s, \mu^N_s)\rangle = \int_{\mathcal{S}_N} \langle f, \psi(t-s,\mu^N_s, \nu)\rangle \mathcal{Q}_N(\mu^N, d\nu); \end{equation} \begin{equation} \psi(u,\mu, \nu)=\phi_u(\nu)-\phi_u(\mu)-\mathcal{D}\phi_u(\mu)[\nu-\mu] \end{equation} and $\mathcal{Q}_N$ is the transition kernel (\ref{eq: definition of script Q}) of the Kac process.   \medskip \\ The first term of (\ref{eq: dominate pointwise bound}) is controlled in $L^2$ by Lemma \ref{thrm: pointwise martingale control}, and so it remains to bound the second term in $L^2$. Let $s\ge 0$, and let $\mu^N$ be a measure obtained from $\mu^N_s$ by a collision, as in (\ref{eq: change of measure at collision}). Then, using the estimate (\ref{eq: stability for BE 2}), we bound 
 \begin{equation} \begin{split} \|\psi(t-s, \mu^N_s, \mu^N)\|_{\mathrm{TV}+2} & = \|\phi_{t-s}(\mu^N)-\phi_{t-s}(\mu^N_s)-\mathcal{D}\phi_{t-s}(\mu^N_s)\|_{\mathrm{TV}+2} \\[2ex] &\lesssim e^{-\lambda_0(t-s)/2} \|\mu^N-\mu^N_2\|^{2-\epsilon}_\mathrm{TV}\hspace{0.1cm}\Lambda_k(\mu^N, \mu^N_s)^\frac{1}{2}. \end{split}\end{equation} By Lemma \ref{lemma:momentincreaseatcollision}, we know that $\Lambda_k(\mu^N)\lesssim \Lambda_k(\mu^N_s)$. Moreover, from the form (\ref{eq: change of measure at collision}) of possible $\mu^N$, we know that \begin{equation} \|\mu^N-\mu^N_s\|_\mathrm{TV}\le \frac{4}{N}\hspace{0.5cm}\text{for }\mathcal{Q}_N(\mu^N_s, \cdot)\text{-almost all }\mu^N. \end{equation} Therefore, almost surely, for all $s$ and $\mathcal{Q}_N(\mu^N_s, \cdot)$-almost all $\mu^N$, we have the bound \begin{equation} \|\psi(t-s,\mu^N_s, \mu^N)\|_{\mathrm{TV}+2} \lesssim e^{-\lambda_0(t-s)/2} N^{\epsilon-2}\hspace{0.1cm}\Lambda_k(\mu^N_s)^\frac{1}{2} \end{equation} where the implied constants are independent of $s, \mu^N_s$. Integrating with respect to $\mathcal{Q}_N(\mu^N_s, d\mu^N)$, we obtain an upper bound for $\langle f, \rho^N(t-s,\mu^N_s)\rangle$:\begin{equation} \label{eq: dominate rho} \begin{split}  \sup_{f\in \mathcal{A}}\hspace{0.05cm} \langle f, \rho^N(t-s, \mu^N_s)\rangle  &\leq \int_{\mathcal{S}_N} \left\|\psi(t-s, \mu^N_s, \mu^N)\right\|_{\mathrm{TV}+2} \hspace{0.1cm} \mathcal{Q}_N(\mu^N_s, d\mu^N) \\ & \lesssim e^{-\lambda_0(t-s)/2} \hspace{0.1cm} N^{\epsilon-1} \hspace{0.1cm} \Lambda_k(\mu^N_s)^\frac{1}{2}.\end{split}  \end{equation} We now take the $L^2$ norm of the second term in (\ref{eq: dominate pointwise bound}). Using Proposition \ref{lemma:momentboundpt1} to control the moments $\Lambda_k$ appearing in the integral, we obtain \begin{equation} \begin{split} \left\|\hspace{0.1cm} \int_0^t \sup_{f\in \mathcal{A}} \hspace{0.1cm} \langle f, \rho^N(t-s, \mu^N_s)\rangle \hspace{0.1cm} ds \hspace{0.1cm}\right\|_{L^2(\mathbb{P})} &\leq   \mathlarger{\mathlarger{\int}}_0^t \hspace{0.1cm} \left\| \sup_{f\in \mathcal{A}} \hspace{0.1cm} \langle f, \rho^N(t-s, \mu^N_s)\rangle \right\|_{L^2(\mathbb{P})}  ds \\ & \lesssim \int_0^t e^{-\lambda(t-s)/2}\hspace{0.1cm} N^{\epsilon-1} \hspace{0.1cm} \left\|\Lambda_{k}(\mu^N_s)^\frac{1}{2}\right\|_{L^2(\mathbb{P})} ds \\ & \lesssim N^{\epsilon-1} \hspace{0.1cm}  a^{1/2}.   \end{split} \end{equation} Noting that the exponent $\epsilon-1 < \epsilon-\frac{1}{d}$, we combine this with Lemma \ref{thrm: pointwise martingale control}, and keep the worse asymptotics. \end{proof} \begin{proof}[Proof of Lemma \ref{thrm: pointwise martingale control}] We begin by reviewing the following estimates for $1-$Lipschitz functions from \cite{ACE}. Following \cite{ACE}, we use angle brackets $\langle f \rangle_C $ to denote the average of a bounded function $f$ over a Borel set $C$ of finite, nonzero measure. \\ Let $f$ be $1-$ Lipschitz, and consider $B=[0,2^{-j}]^d$. Then, for some numerical constant $c_d$, we have \begin{equation} \label{eq:scalebound} \forall v \in B, \hspace{0.1cm} |f(v)-\langle f \rangle_B|\le  c_d 2^{-j};\hspace{1cm} |\langle f \rangle _B -\langle f \rangle _{2B} | \le c_d 2^{-j}. \end{equation} We note that both of these bounds are linear in the length scale $2^{-j}$ of the box.  We deal with the case $N\ge 2^{2d}$. \medskip\\ The proof is based on the following `hierarchical' partition of $\mathbb{R}^d$, given in the proof \cite[Proposition 7.1]{ACE}. \begin{itemize} \item For $j \in \mathbb{Z}$, we take $B_j=(-2^j, 2^j]$. \item Set $A_0 = B_0$ and, for $j\geq 1$, $A_j = B_j \setminus B_{j-1}$. \item For $j\geq 1$ and $l \ge 2$, there is a unique partition $\mathcal{P}_{j,l}$ of $A_j$ by $2^{ld}-2^{(l-1)d}$ translates of $B_{j-l}$. \item Similarly, write $\mathcal{P}_{0,l}$ for the unique partition of $A_0$ by $2^{dl}$ translates of $B_{-l}$. \item For $l\geq 3$ and $k\in \mathbb{Z}$, let $B\in \mathcal{P}_{j,l}$. We write $\pi(B)$ for the unique element of of $\mathcal{P}_{j,l-1}$ such that $B\subset \pi(B)$.\end{itemize} We deal first with the case $d\geq 3$. Fix discretisation parameters $L, J \ge 1$. Given a test function $f\in \mathcal{A}$, we can decompose \begin{equation} f=\sum_{j=0}^J \hspace{0.2cm} \sum_{l=2}^L \hspace{0.1cm} \sum_{B \in \mathcal{P}_{j,l}}  a_B(f)(1+|v|^2) 1_{B}+ \beta(f)\end{equation} where we define \begin{equation} a_B(f)=\begin{cases} \langle \hat{f} \rangle_ B  & \text{if } B \in \mathcal{P}_{j,2}, \text{ for some } j \ge 0 \\ \langle \hat{f} \rangle _B - \langle \hat{f} \rangle _{\pi(B)} & \text{if } B \in \mathcal{P}_{j,l}, \text{ for some } j \ge 0, l\ge 3\end{cases} \end{equation} and the equation serves to define the remainder term $\beta(f)$.  Write $h_B = 2^{2j}(1+|v|^2)1_B$, for $B\in \mathcal{P}_{j,l}$, and write $M^{N;B}_t = M^{N, h_B}_t$. We can now write \begin{equation}\begin{split} \label{eq: decomposition of MNFT} & M^{N,f}_t= \sum_{j=0}^J\sum_{l=2}^L \sum_{B \in \mathcal{P}_{j,l}} 2^{-2j}a_B(f) M^{N;B}_t  +R^{N,f}_t; \end{split} \end{equation} \begin{equation} \label{eq: definition of RNFT} R^{N,f}_t=   \int_{(0,t]\times \mathcal{S}_N}  \langle \beta(f), \Delta(s,t,\mu^N)\rangle (m^N-\overline{m}^N)(ds, d\mu^N) \end{equation} and where $\Delta$, $m^N$ and $\overline{m}^N$ are defined in Section \ref{sec: interpolation decomposition}.  This is the key decomposition in the proof. Roughly speaking: \begin{itemize} \item The martingales $M^{N;B}$ are controlled by a bound (\ref{eq: elltwo martingale control}) from the general theory of Markov chains, \emph{independently of f}. \item The coefficients $a_B$ depend on $f$, but are bounded, uniformly over $f\in \mathcal{A}$. \item On $B_J$, $\beta(f)$ will be small, uniformly in $f$, due to the Lipschitz bound on $f$ and the estimate (\ref{eq:scalebound}). This may be viewed as a \emph{relative compactness} argument, as discussed in Remark \ref{rmk: compactness}: given $\epsilon>0$, one could use this construction to produce a finite $\epsilon$-net for $\mathcal{A}|_{B_J}$ in the norm of $\mathcal{A}_0|_{B_J}$. \item $|\beta(f)|\leq 1$ is bounded on $\mathbb{R}^d\setminus B_J$, and the contribution from this region will be controlled by the moment bounds.\end{itemize}  To control the martingale term uniformly in $f$, observe that for $B \in \mathcal{P}_{j,l}$, the bound (\ref{eq:scalebound}) gives $2^{-2j}|a_B(f)|\lesssim 2^{-j-l}$, and $\#\mathcal{P}_{j,l}\le 2^{dl}$. Hence, independently of $f\in \mathcal{A}$, \begin{equation} \left(\sum_{j=0}^J \sum_{B \in \mathcal{P}_{j,l}} (a_B(f)2^{-2j})^2\right) \lesssim 2^{(d-2)l}.\end{equation} Now, by Cauchy-Schwarz, \begin{equation} \label{eq: use of CS} \sup_{f\in \mathcal{A}} \hspace{0.1cm} \left|\sum_{j=0}^J\sum_{l=2}^L \sum_{B \in \mathcal{P}_{j,l}} 2^{-2j}a_B(f) M^{N;B}_t\right| \lesssim \sum_{l=2}^L \left(\sum_{j=0}^J \sum_{B \in \mathcal{P}_{j,l}} \left\{M^{N;B}_t\right\}^2\right)^{1/2}2^{(d/2-1)l}. \end{equation} Let $(M^{N;B;t}_s)_{s\leq t}$ be the martingale \begin{equation} \label{eq: MNBTS} M^{N;B;t}_s = \int_{(0,s]\times\mathcal{S}_N} \langle h_B, \Delta(u,t,\mu^N)\rangle (m^N-\overline{m}^N)(du, d\mu^N).\end{equation} We can control the remaining martingale term pointwise in $L^2$ by applying the martingale bound (\ref{eq: elltwo martingale control}) at the terminal time $t$: \begin{equation} \begin{split} &  \left\|M^{N;B}_t\right\|_{L^2(\mathbb{P})}^2   = \mathbb{E} \int_{(0,t]\times \mathcal{S}_N} \langle (1+|v|^2)2^{2j}1_{B}, \Delta(s,t,\mu^N)\rangle ^2 \overline{m}^N(ds, d\mu^N) \\& \lesssim  \mathbb{E} \left[\int_{(0,t]\times \mathcal{S}_N} \langle (1+|v|^4)1_{B}, |\Delta(s,t,\mu^N)|\rangle ^2 \overline{m}^N(ds, d\mu^N)\right].\end{split} \end{equation} Summing over $B\in \mathcal{P}_{j,l}$ and $j=0, ..,J$, we Minkowski's inequality to move the sum inside the integral against $\Delta$, and note that $\sum_j \sum_{B\in\mathcal{P}_{j,l}} h_B \lesssim (1+|v|^4)$. This produces the bound  \begin{equation} 
 \begin{split}   &\sum_{j=0}^J \sum_{B\in \mathcal{P}_{j,l}} \left\|M^{N;B}_t\right\|_{L^2(\mathbb{P})}^2    \lesssim \mathbb{E} \left[\int_{(0,t]\times \mathcal{S}_N} \langle (1+|v|^4), |\Delta(s,t,\mu^N)|\rangle ^2 \overline{m}^N(ds, d\mu^N) \right] \\[1ex] & = \mathbb{E} \left[\int_{(0,t]\times \mathcal{S}_N} \|\phi_{t-s}(\mu^N)-\phi_{t-s}(\mu^N_{s-})\|^2_{\mathrm{TV}+4} \hspace{0.1cm} \overline{m}^N(ds, d\mu^N) \right] 
 \end{split}\end{equation} where the second line follows by the definition of $\Delta$ in (\ref{eq: definition of Delta}). Using the stability estimates in Corollary \ref{cor: new stability for BE} with $q=4$, we find \begin{equation} \begin{split} \sum_{j=0}^J\sum_{B\in\mathcal{P}_{j,l}} \left\|M^{N;B}_t\right\|_{L^2(\mathbb{P})}^2 \lesssim \mathbb{E} \left[\int_{(0,t]\times \mathcal{S}_N} e^{-\lambda(t-s)}\Lambda_k(\mu^N_s, \mu^N) N^{2(\epsilon-1)}\hspace{0.1cm}\overline{m}^N(ds, d\mu^N) \right]. \end{split} \end{equation} For $\overline{m}^N$-almost all $(s, \mu^N)$, we bound $\Lambda_k(\mu^N_s, \mu^N)\lesssim \Lambda_k(\mu^N_s)$ by Lemma \ref{lemma:momentincreaseatcollision}, and $\overline{m}^N(ds, \mathcal{S}_N)\le 2Nds$, to bound the right hand side by  \begin{equation} \begin{split} \sum_{j=0}^J\sum_{B\in\mathcal{P}_{j,l}} \left\|M^{N;B}_t\right\|_{L^2(\mathbb{P})}^2 &\lesssim \int_0^t e^{-\lambda(t-s)}N^{2\epsilon-1} \hspace{0.1cm} \mathbb{E}[\Lambda_k(\mu^N_s)]\hspace{0.1cm}ds \\ & \lesssim N^{2\epsilon-1}a^\frac{1}{2}\end{split} \end{equation} where the second line follows using the moment estimates for the Kac process, established in Proposition \ref{thrm:momentinequalities}. Therefore, (\ref{eq: use of CS}) gives \begin{equation} \begin{split} \label{eq: pointwise bound on martingale term} \left\|\hspace{0.1cm} \sup_{f\in \mathcal{A}} \hspace{0.1cm} \left|\sum_{j=0}^J\sum_{l=2}^L \sum_{B \in \mathcal{P}_{j,l}} a_B(f) M^{N;l}_t\right|\hspace{0.1cm} \right\|_{L^2(\mathbb{P})} & \lesssim N^{\epsilon-1/2} a^{1/2} \sum_{l=2}^L 2^{(d/2-1)l} \\[1ex] & \lesssim  N^{\epsilon-1/2} \hspace{0.2cm}2^{(d/2-1)L}\hspace{0.1cm}a^{1/2}. \end{split} 
 \end{equation}  The remaining points are a control on $\beta(f)$, uniformly in $f\in\mathcal{A}$, dealing with $B_J$ and $\mathbb{R}^d \setminus B_J$ separately.  Fix $f\in \mathcal{A}$ and let $B \in \mathcal{P}_{j,L}$ with $j\le J$.  The definition gives $\hat{\beta}(f)= \hat{f} - \langle \hat{f} \rangle _B$  on $B$, and so \begin{equation} \text{On } B, \hspace{0.5cm} |\beta(f)| = (1+|v|^2)|\hat{f}-\langle \hat{f} \rangle _B|  \hspace{0.1cm}\lesssim \hspace{0.1cm} (1+|v|^2)2^{j-L}.  \end{equation} Since $|v|\ge 2^{j-1}$ on $B$, and $B \in \mathcal{P}_{j,L}$ is arbitrary, we see that \begin{equation} \text{On } B_J, \hspace{0.5cm} |\beta(f)|\lesssim 2^{-L}(1+|v|^4).\end{equation} On the other hand, the uniform bound $\|\hat{f}\|_\infty \le 1$ implies that \begin{equation} \text{On }B_J^c, \hspace{0.5cm} |\beta(f)|\leq (1+|v|^2) \leq 2^{-2J}(1+|v|^4).\end{equation} Combining, we have the global bound for all $f\in \mathcal{A}$:\begin{equation} \forall v\in \mathbb{R}^d, \hspace{0.5cm} |\beta(f)| \lesssim (2^{-2J}+2^{-L})(1+|v|^4).\end{equation} Recalling the definition (\ref{eq: definition of Delta}) of $\Delta$, we use the stability estimate in Corollary \ref{cor: new stability for BE}, with $q=4$, and the moment increase bound Lemma \ref{lemma:momentincreaseatcollision}, as above to see that almost surely, for  $m^N+\overline{m}^N$-almost all $(s, \mu^N)$, we have the bound  \begin{equation} \label{eq: majorise integrand of error term}  \begin{split}  \sup_{f\in \mathcal{A}} \left| \langle \beta(f), |\Delta(s,t,\mu^N)|\rangle \right| & \lesssim (2^{-2J}+2^{-L})\hspace{0.1cm}\|\Delta(s,t,\mu^N)\|_{\mathrm{TV}+4} \\ &  \lesssim (2^{-2J}+2^{-L})\hspace{0.1cm} e^{-\lambda(t-s)/2}\hspace{0.1cm} N^{\epsilon-1} \hspace{0.1cm} \Lambda_{k}(\mu^N_{s-})^\frac{1}{2} \\ &=:H_s\end{split}\end{equation} where we introduced the shorthand $H_s$ for the final expression, for simplicity. We now use the trivial observation that \begin{equation} \label{eq: dominate integrand and integrator seperately} \sup_{f\in \mathcal{A}} \hspace{0.1cm}\left|R^{N,f}_t\right| \le \int_{(0,t]\times\mathcal{S}_N} \left\{\sup_{f\in \mathcal{A}}\hspace{0.1cm}\left\langle |\beta(f)|, |\Delta(s,t,\mu^N)|\right\rangle\right\}(m^N+\overline{m}^N)(ds, d\mu^N).\end{equation}We split the measure $m^N +\overline{m}^N= (m^N-\overline{m}^N)+2\overline{m}^N$ to obtain a uniform bound for the error terms $R^{N,f}_t$ defined in (\ref{eq: definition of RNFT}): \begin{equation} \label{eq: introduce t1 t2} \begin{split} \left\|\hspace{0.1cm}\sup_{f\in \mathcal{A}} \hspace{0.1cm}R^{N,f}_t\hspace{0.1cm}\right\|_{L^2(\mathbb{P})}  &  \lesssim \left\| \int_0^t H_s (m^N+\overline{m}^N)(ds,\mathcal{S}_N)\right\|_{L^2(\mathbb{P})} \\[1ex] &  \lesssim (2^{-2J}+2^{-L})N^{\epsilon-1}\left[\mathcal{T}_1+\mathcal{T}_2\right]\end{split} \end{equation} where we have written \begin{equation} \mathcal{T}_1=\left\| \int_0^t  e^{-\lambda(t-s)/2} \Lambda_{k}(\mu^N_{s-})^\frac{1}{2} \overline{m}^N(ds,\mathcal{S}_N)\right\|_{L^2(\mathbb{P})} \end{equation} \begin{equation} \mathcal{T}_2=  \left\| \int_0^t  e^{-\lambda(t-s)/2} \Lambda_{k}(\mu^N_{s-})^\frac{1}{2} (m^N-\overline{m}^N)(ds,\mathcal{S}_N)\right\|_{L^2(\mathbb{P})}. \end{equation}  $\mathcal{T}_1$ is controlled by dominating $\overline{m}^N(ds, \mathcal{S}_N)\leq 2N ds$ to obtain \begin{equation} \begin{split} \label{eq:dominatembar}  \mathcal{T}_1 \lesssim N \left\|\int_0^t e^{-\lambda(t-s)/2} \Lambda_{k}(\mu^N_{s})^\frac{1}{2} ds \right\|_{L^2(\mathbb{P})} & \lesssim N \int_0^t e^{-\lambda(t-s)/2} \|\Lambda_{k}(\mu^N_{s})^\frac{1}{2}\|_{L^2(\mathbb{P})} \hspace{0.1cm}  ds \\[1ex] &\lesssim N a^{1/2}. \end{split}\end{equation} We control $\mathcal{T}_2$ by It\^{o}'s isometry for $m^N-\overline{m}^N$, which is reviewed in (\ref{eq: QV of M}): \begin{equation} \label{eq:itoisometrycontrol}\begin{split}  \mathcal{T}_2^2 &= \mathbb{E} \left\{  \int_0^t  e^{-\lambda(t-s)} \Lambda_k(\mu^N_{s-}) \overline{m}^N(ds,\mathcal{S}_N)\right\}  \\& \lesssim N \int_0^t e^{-\lambda(t-s)} \mathbb{E}\left\{\Lambda_k(\mu^N_{s-}) \right\} ds \\ & \lesssim  N \hspace{0.1cm}a.\end{split} \end{equation} Combining (\ref{eq: introduce t1 t2}, \ref{eq:dominatembar}, \ref{eq:itoisometrycontrol}), we obtain \begin{equation} \label{eq:control of error term pw} \begin{split} \left\|\hspace{0.1cm}\sup_{f\in \mathcal{A}} \hspace{0.1cm}R^{N,f}_t\hspace{0.1cm}\right\|_{L^2(\mathbb{P})}  \lesssim (2^{-2J}+2^{-L})\hspace{0.1cm}N^{\epsilon-1}\hspace{0.1cm}a^{1/2}.\end{split} \end{equation} Finally, we combine (\ref{eq: decomposition of MNFT}, \ref{eq: pointwise bound on martingale term}, \ref{eq:control of error term pw}) to obtain \begin{equation} \left\|\hspace{0.1cm} \sup_{f \in \mathcal{A}} \left|M^{N,f}_t\right|\right\|_{L^2(\mathbb{P})} \hspace{0.2cm} \lesssim \hspace{0.2cm} N^{\epsilon} \hspace{0.1cm} a^{1/2} (N^{-1/2}\hspace{0.1cm}2^{(d/2-1)L}+2^{-L}+2^{-2J}). \end{equation} Taking $L=\lfloor \log_2(N)/d \rfloor$ and $J\uparrow \infty$ produces the claimed result. For $d=2$, we replace $2^{(d/2-1)L}$ by $L$ in (\ref{eq: pointwise bound on martingale term}), and optimise as before, absorbing the factors of $(\log N)$ to make the exponent of $N$ slightly larger. \end{proof} 
\section{Proof of Theorem \ref{thrm: Main Local Uniform Estimate}} \label{sec: proof of LU}
We now adapt the ideas of Theorem \ref{thrm: pointwise martingale control} to a local uniform setting, and working in $L^p$, to prove the local uniform approximation result Theorem \ref{thrm: Main Local Uniform Estimate}. As in the proof above, most of the work is in controlling the martingale term $(M^{N,f}_t)_{f\in \mathcal{A}}$ defined in (\ref{eq: definition of MNFT}), uniformly in $f$; for a pathwise local uniform estimate, we wish to control an expression of the form \begin{equation}\label{eq: local unf mg exp} \left\| \hspace{0.1cm}\sup_{f\in \mathcal{A}} \hspace{0.1cm}\sup_{t\leq t_\text{fin}} \hspace{0.1cm} \left|M^{N,f}_t\right|\right\|_{L^p(\mathbb{P})}. \end{equation} Since we will frequently encounter suprema of processes on compact time intervals, we introduce notation. For any stochastic process $M$, we write \begin{equation}\label{eq: use of star} M_{\star,t}=\sup_{s\leq t}|M_t| \end{equation}  Proving the sharpest asymptotics in the time horizon $t_\text{fin}$ requires working in $L^p$ instead of $L^2$, for large exponents $p$. This leads to a weaker exponent in $N$: we obtain only $N^{\epsilon-p'/2d}$ instead of $N^{\epsilon -1/d}$, where $p'\leq 2$ is the H\"{o}lder conjugate to $p$. However, by making $p$ large, we are able to obtain estimates which degrade slowly in the time horizon $t_\text{fin}$, with only a factor of $(1+t_\text{fin})^{1/p}$. The exponent for $t_\text{fin}$ can thus be made arbitrarily small, while the resulting exponent for $N$ is bounded away from $0$ as we make $p$ large. \\ \\ The key result required for the local uniform estimate is the following control of the expression (\ref{eq: local unf mg exp}), in analogy to Lemma \ref{thrm: pointwise martingale control}. \begin{lemma} \label{thrm: local uniform martingale control} Let $\epsilon>0$, $a\ge 1$ and $p\geq 2$, and let $1<p'\le 2$ be the H\"{o}lder conjugate to $p$. Let $k$ be large enough that Corollary \ref{cor: new stability for BE} holds for $q=5$, with H\"{o}lder exponent $1-\epsilon$, and with some $0<\lambda<\lambda_0.$ \\ Let $(\mu^N_t)_{t\geq 0}$ be a Kac process on $N\geq 2$ particles, with initial moment $\Lambda_{kp}(\mu^N_0)\leq a^p$. Let $M^{N,f}_t$ be the processes given by (\ref{eq: definition of MNFT}), and $M^{N,f}_{\star, t}$ their local suprema, as in (\ref{eq: use of star}). Then, for any time horizon $t_\text{fin}\in [0,\infty)$, we have the control \begin{equation} \left\|\hspace{0.1cm} \sup_{f\in \mathcal{A}} \hspace{0.1cm}  \hspace{0.1cm} M^{N,f}_{\star,t_\text{fin}}\hspace{0.1cm}\right\|_{L^p(\mathbb{P})} \lesssim a^{1/2} \hspace{0.1cm} N^{-\alpha}\hspace{0.1cm}(\log N)^{1/p'}\hspace{0.1cm}(1+t_\text{fin})^\frac{3p+1}{2p}\end{equation} where $\alpha = \frac{p'}{2d}-\epsilon$.\end{lemma} The proof of this Lemma follows the same ideas as the proof of the equivalent result, Lemma \ref{thrm: pointwise martingale control}, for the pointwise bound. However, in this case, we must modify the argument to work in $L^p$ rather than $L^2$, and also to control all terms uniformly on the compact time interval $[0, t_\text{fin}]$. This will be deferred until the end of this section.\medskip\\
Following the argument of the pointwise bound in Theorem \ref{thrm: PW convergence}, we can now produce an initial pathwise, local uniform estimate for the case $\mu_0=\mu^N_0$, with worse long-time behaviour. From this, we will `bootstrap' to the desired long-time behaviour in Theorem \ref{thrm: Main Local Uniform Estimate}.  \begin{lemma} \label{lemma: initial LU bound} Let $\epsilon>0$, $a\ge 1$ and $p\geq 2$, with H\"older conjugate $p'\le 2$. Choose $k$ large enough that Proposition \ref{thrm: stability for BE} holds with exponent $1-\epsilon$, and that Corollary \ref{cor: new stability for BE} holds with exponent $1-\epsilon$ and $q=5$. Let $(\mu^N_t)_{t\geq 0}$ be a Kac process on $N\geq 2$ particles, with initial moment $\Lambda_{kp}(\mu^N_0)\leq a^p$. Then, for any time horizon $t_\text{fin}\ge 0$, we have the control \begin{equation} \left\| \hspace{0.1cm} \sup_{t\leq t_\text{fin}} \hspace{0.1cm}W\left(\mu^N_t, \phi_t\left(\mu^N_0\right)\right) \hspace{0.1cm} \right\|_{L^p(\mathbb{P})} \lesssim a^{1/2} \hspace{0.1cm}N^{\epsilon-\frac{p'}{2d}}\hspace{0.1cm} (\log N)^{1/p'}(\hspace{0.1cm}1+t_\text{fin})^\frac{3p+1}{2p}.\end{equation}   \end{lemma}  \begin{proof}[Proof of Lemma \ref{lemma: initial LU bound}] As in Theorem \ref{thrm: PW convergence}, it remains to control the supremum of the integral term in Formula \ref{form:newdecomposition}\begin{equation} \sup_{t\leq t_\text{fin}} \int_0^t \sup_{f\in \mathcal{A}} \langle f, \rho^N(t-s, \mu^N_s)\rangle ds \end{equation} where $\rho^N$ is given by (\ref{eq: definition of rho}). Following the previous calculation (\ref{eq: dominate rho}), we majorise, for $s\leq t\leq t_\text{fin}$, \begin{equation} \label{eq: loc unf dominate rho}\sup_{f\in \mathcal{A}} \langle f, \rho^N(t-s, \mu^N_s)\rangle \lesssim  N^{\epsilon-1} \hspace{0.1cm} \sup_{u\leq t_\text{fin}} \left\{ \Lambda_{k}(\mu^N_u)^\frac{1}{2}\right\} \end{equation} from which it follows that \begin{equation} \label{eq: loc unf dominate rho 2} \sup_{t\leq t_\text{fin}} \int_0^t \sup_{f\in \mathcal{A}} \langle f, \rho^N(t-s, \mu^N_s)\rangle ds \lesssim N^{\epsilon-1} \hspace{0.1cm} t_\text{fin} \hspace{0.1cm} \sup_{u\leq t_\text{fin}} \left\{ \Lambda_{k}(\mu^N_u)^\frac{1}{2}\right\}.\end{equation} From the local uniform moment bound established in Proposition \ref{lemma:momentboundpt1}, and the initial moment bound on $\mu^N_0$, \begin{equation} \label{eq: locunf moment bound} \begin{split}  \left\|\hspace{0.1cm} \sup_{u\leq t_\text{fin}} \left\{ \Lambda_{k}(\mu^N_u)^\frac{1}{2}\right\} \right\|_{L^p(\mathbb{P})} &\leq \left\| \hspace{0.1cm} \sup_{u\leq t_\text{fin}} \left\{ \Lambda_{k}(\mu^N_u)^\frac{1}{2}\right\} \right\|_{L^{2p}(\mathbb{P})} \leq \mathbb{E}\left[\sup_{u\leq t_\text{fin}} \Lambda_{pk}(\mu^N_u)^\frac{1}{2}\right]^{1/2p} \\[1ex] & \lesssim a ^{1/2} \hspace{0.1cm} (1+t_\text{fin})^{1/2p}. \end{split} \end{equation} Combining the estimates (\ref{eq: loc unf dominate rho 2}, \ref{eq: locunf moment bound}), we see that \begin{equation} \left\|\sup_{t\leq t_\text{fin}} \int_0^t \sup_{f\in \mathcal{A}} \langle f, \rho^N(t-s, \mu^N_s)\rangle ds\right\|_{L^p(\mathbb{P})} \lesssim N^{\epsilon-1}\hspace{0.1cm}a^{1/2} \hspace{0.1cm}(1+t_\text{fin})^\frac{2p+1}{2p}.\end{equation} We combine this with Lemma \ref{thrm: local uniform martingale control} and keep the worse asymptotics.\end{proof}  

We will now show how to `bootstrap' to better dependence of the time horizon $t_\text{fin}$. Heuristically, the proof allows us to replace powers of $t_\text{fin}$ in the initial bound with the same power of $\log N$, and introduce an additional factor of $(1+t_\text{fin})^{1/p}$. As was remarked below Proposition \ref{thrm: stability for BE}, we could derive Theorem \ref{thrm: PW convergence} and Lemma \ref{lemma: initial LU bound} under the milder assumptions \begin{equation} \label{eq: weaker stability 4}  \|\phi_t(\nu)-\phi_t(\mu)\|_{\mathrm{TV}+5} \leq F(t) \Lambda_{k}(\mu, \nu)^\frac{1}{2}\|\mu-\nu\|_\mathrm{TV}^\eta; \end{equation}  \begin{equation}\label{eq: weaker stability 5}
\|\phi_t(\nu)-\phi_t(\mu) - \xi_t \|_{\mathrm{TV}+2} \leq G(t) \Lambda_{k}(\mu, \nu)^\frac{1}{2}\|\mu-\nu\|_\mathrm{TV}^{1+\eta}\end{equation} for functions $F,G$ such that \begin{equation} \label{eq: weaker stability 6} \left(\int_0^\infty F^2 dt\right)^{1/2}<\infty;\hspace{0.5cm} \int_0^\infty G dt<\infty. \end{equation} If we also assume that $F\rightarrow 0$ as $t\rightarrow \infty$, we can use an identical bootstrap argument, with $\log N$ replaced by a power of \begin{equation} \tau_N := \sup\{t: F(t) > N^{-\alpha}\}\end{equation} which produces a potentially larger loss. \emph{Hence, the the full strength of exponential decay in Proposition \ref{thrm: stability for BE} is used to control the asymptotic loss due to the bootstrap}. 
  \begin{proof}[Proof of Theorem \ref{thrm: Main Local Uniform Estimate}] As in the proof of Theorem \ref{thrm: PW convergence}, it is sufficient to prove the case $\mu^N_0=\mu_0$. Then, making $k$ larger if necessary, we may use Theorem \ref{thrm: W-W continuity of phit} to control $\sup_{t\ge 0}W(\phi_t(\mu^N_0), \phi_t(\mu_0))$, which proves the general result.\medskip \\  Let $0<\epsilon'<\epsilon$, and choose $k$ such that Lemma \ref{lemma: initial LU bound} holds for $\epsilon'$.  Let $\alpha'<\alpha$ be the exponent of $N$ obtained with $\epsilon'$ in place of $\epsilon$. From the stability estimate Proposition \ref{thrm: stability for BE}, we have \begin{equation} \forall \mu, \nu \in \mathcal{S}^k_a, \hspace{0.2cm}\|\phi_t(\mu)-\phi_t(\nu)\|_{\mathrm{TV}+2} \lesssim \Lambda_{k}(\mu, \nu)^\frac{1}{2} e^{-\lambda_0 t/2}.\end{equation}Define $\tau = \tau_N = -2\lambda_0^{-1} \log(N^{-\alpha'})$ and consider $t_\text{fin}> \tau +1 $. Fix a positive integer $n$, and partition the interval $[0, t_\text{fin}]$ as $I_1\cup I_1 \cup...\cup I_n$: \begin{equation} I_0=[0,\tau]; \hspace{0.5cm}I_r = \left[\tau+(r-1)\frac{t_\text{fin}-\tau}{n},\tau+r\frac{t_\text{fin}-\tau}{n}\right]=:[s_r+\tau,t_r].\end{equation} Write also $H_r=[s_r, t_r] \supset I_r$. Since the norm $\|\cdot\|_{\mathrm{TV}+2}$ dominates the Wasserstein distance $W$, we have the bound\begin{equation} \label{eq: bootstrap bound}\sup_{t\in I_r} \hspace{0.1cm} W(\mu^N_t, \phi_t(\mu^N_0)) \lesssim \sup_{t\in H_r} \hspace{0.1cm} W(\mu^N_t, \phi_{t-s_r}(\mu^N_{s_r})) + e^{-\lambda \tau} \Lambda_{k}(\mu^N_{s_r}, \phi_{s_r}(\mu^N_0))^\frac{1}{2}.\end{equation} We bound the two terms in (\ref{eq: bootstrap bound}) separately. Denote $(\mathcal{F}^N_t)_{t\geq 0}$ the natural filtration of $(\mu^N_t)_{t\geq 0}$. We control the first term by Lemma \ref{lemma: initial LU bound}, applied to the restarted process $(\mu^N_t)_{t\geq s_r}$: \begin{equation}\begin{split} \left\|\hspace{0.1cm}\sup_{t\in H_r} \hspace{0.1cm} W(\mu^N_t, \phi_{t-s_r}(\mu^N_{s_r})) \right \|_{L^p(\mathbb{P})} ^p = \mathbb{E}\left\{\mathbb{E}\left(\left[\left.\sup_{s_r \leq t \leq t_r} \hspace{0.1cm} W(\mu^N_t, \phi_{t-s_r}(\mu^N_{s_r}))\right]^p \right| \mathcal{F}^N_{s_r} \right) \right\} \\ \lesssim \mathbb{E} \left\{\Lambda_{pk}(\mu^N_{s_r}) ^{1/p}\right\}  \left(1+ \tau + \frac{t-\tau}{n}\right)^\frac{3p+1}{2} N^{-p\alpha'}\hspace{0.1cm}(\log N)^\frac{p}{p'}.\end{split}\end{equation} We control the moment in the usual way, using Proposition \ref{thrm:momentinequalities}\ref{lemma:momentboundpt1}, to obtain \begin{equation} \label{eq: Bootstrap bound 2}  \left\|\hspace{0.1cm}\sup_{t\in H_r} \hspace{0.1cm} W(\mu^N_t, \phi_{t-s_r}(\mu^N_{s_r})) \right \|_{L^p(\mathbb{P})} ^p    \lesssim a^p \left(1+ \tau + \frac{t-\tau}{n}\right)^\frac{3p+1}{2} N^{-p\alpha'}\hspace{0.1cm}(\log N)^\frac{p}{p'}.\end{equation}  We now turn to the second term in (\ref{eq: bootstrap bound}). Using the definition of $\tau$ and the moment estimates (\ref{eq: pointwise moment bound}, \ref{eq: BE moment bound}) in Proposition \ref{thrm:momentinequalities}, \begin{equation} \label{eq: Bootstrap bound 3} \|e^{-\lambda \tau/2} \Lambda_{k}(\mu^N_{s_r}, \phi_{s_r}(\mu^N_0))^\frac{1}{2} \|_{L^p(\mathbb{P})} \lesssim N^{-\alpha'} \hspace{0.1cm}  a^{1/2}. \end{equation} Combining the estimates (\ref{eq: Bootstrap bound 2}, \ref{eq: Bootstrap bound 3}), and absorbing  powers of $\tau$ into the powers of $(\log N)$, we obtain \begin{equation} \left\| \hspace{0.1cm} \sup_{t\in I_r} W(\mu^N_t, \phi_t(\mu^N_0)) \right\|_{L^p(\mathbb{P})} \lesssim a^{1/2} \hspace{0.1cm} \left(1+ \frac{t_\text{fin}-\tau}{n}\right)^\frac{3p+1}{2p}\left(N^{-\alpha'} (\log N)^{\frac{3p+1}{2p}+\frac{1}{p'}}\right).\end{equation} Observe that \begin{equation} \left\{ \sup_{\tau \leq t \leq t_\text{fin}} W\left(\mu^N_t, \phi_t(\mu^N_0)\right) \right\}^p  \leq \mathlarger{\mathlarger{\sum}}_{r=1}^n \left\{ \sup_{t \in I_r} W\left(\mu^N_t, \phi_t(\mu^N_0)\right) \right\}^p. \end{equation} Taking expectations and $p^\text{th}$ root, we find that\begin{equation} \begin{split} &\left\|\hspace{0.1cm} \sup_{\tau \leq t \leq t_\text{fin}} W\left(\mu^N_t, \phi_t(\mu^N_0)\right) \right\|_{L^p(\mathbb{P})} \\& \hspace{2cm} \lesssim n^\frac{1}{p} \hspace{0.1cm} a^{1/2} \hspace{0.1cm} \left(1+ \frac{t_\text{fin}-\tau}{n}\right)^\frac{3p+1}{2p}\left(N^{-\alpha'} (\log N)^{\frac{3p+1}{2p}+\frac{1}{p'}}\right).\end{split} \end{equation} This is optimised at $n\sim (t_\text{fin}-\tau)$, where we obtain the estimate \begin{equation} \begin{split} \label{eq: tfin > tau} \left\|\hspace{0.1cm} \sup_{\tau \leq t \leq t_\text{fin}} W\left(\mu^N_t, \phi_t(\mu^N_0)\right) \right\|_{L^p(\mathbb{P})} &\lesssim a^{1/2} (t_\text{fin}-\tau)^\frac{1}{p}\hspace{0.1cm}\left(N^{-\alpha'} (\log N)^{\frac{3p+1}{2p}+\frac{1}{p'}}\right)  \\ &\leq a^{1/2} \hspace{0.1cm} t_\text{fin}^\frac{1}{p}\hspace{0.1cm}\left(N^{-\alpha'} (\log N)^{\frac{3p+1}{2p}+\frac{1}{p'}}\right).  \end{split} \end{equation} From Lemma \ref{lemma: initial LU bound} applied up to time $\tau=\tau_N$, we have \begin{equation} \begin{split} \label{eq:shorttimecontrolforiteration} \left\|\hspace{0.1cm} \sup_{0 \leq t \leq \tau_N} W\left(\mu^N_t, \phi_t(\mu^N_0)\right) \right\|_{L^2(\mathbb{P})} &\lesssim a^{1/2} \hspace{0.1cm} N^{-\alpha'} \left(1+\frac{2\alpha}{\lambda} \log(N)\right)^{\frac{3p+1}{2p}}(\log N)^{\frac{1}{p'}}  \\ &\lesssim a^{1/2} \hspace{0.1cm}\left(N^{-\alpha} (\log N)^{\frac{3p+1}{2p}+\frac{1}{p'}}\right) .\end{split}\end{equation} Combining (\ref{eq: tfin > tau}, \ref{eq:shorttimecontrolforiteration}), and absorbing the powers of $(\log N)$ into $N^{\epsilon-\epsilon'}$, we have \begin{equation} \begin{split} \left\|\hspace{0.1cm} \sup_{0 \leq t \leq t_\text{fin}} W\left(\mu^N_t, \phi_t(\mu^N_0)\right) \right\|_{L^p(\mathbb{P})} \lesssim a^{1/2} \hspace{0.1cm}(1+t_\text{fin})^\frac{1}{p} \hspace{0.1cm}N^{-\alpha}.\end{split} \end{equation} The case where $t_\text{fin}\leq \tau+ 1$ is essentially identical to (\ref{eq:shorttimecontrolforiteration}).  \end{proof} \begin{remark} We note that this `bootstrap' argument would produce the same result with any \emph{polynomial} time dependence in  Lemma \ref{lemma: initial LU bound}. As a result, the precise time dependence of Lemmas \ref{thrm: local uniform martingale control}, \ref{lemma: initial LU bound} is uninteresting, and we do not attempt to optimise it. We also remark that this method produces the same long-time behaviour even starting from an exponential estimate, at the cost of a fractional power of $N$.  \end{remark} 

It remains to prove Lemma \ref{thrm: local uniform martingale control}. We draw attention to the fact that $M^{f, N}$ are \emph{not} themselves martingales, despite the general construction (\ref{eq: M is for martingale}), since the integrand $\phi_{t-s}(\mu^N)-\phi_{t-s}(\mu^N_{s-})$ depends on the terminal time $t$. We address this by computing an associated family of martingales:  \begin{lemma}\label{lemma:newmartingaleconstruction} Let $(M^{N,f}_t)_{t\geq 0}$ be the processes defined in Formula \ref{form:newdecomposition}. Recalling the notation $Q_t=Q\circ \phi_t$, define \begin{equation} \label{eq: defn of chi} \chi(s,t,\mu^N)=Q_{t-s}(\mu^N)-Q_{t-s}(\mu^N_{s-}).\end{equation} Suppose $f$ satisfies a growth condition $|f(v)|\leq (1+|v|^q)$, for some $q\geq 0$. Consider the martingales $Z^{N,f}_t$ given by \begin{equation}Z^{N,f}_t=\int_{(0,t]\times \mathcal{S}_N}\langle f,\mu^N-\mu^N_{s-}\rangle (m^N-\overline{m}^N)(ds, d\mu^N) \rangle. \end{equation} Then we have the equality  \begin{equation}\begin{split} &Z^{N,f}_t=M^{N,f}_t- C^{N,f}_t\\ &= M^{N,f}_t-\int_0^t ds \int_{(0,s]\times \mathcal{S}_N}  \langle f, \chi(u,s,\mu^N) \rangle (m^N-\overline{m}^N)(du, d\mu^N).\end{split}\end{equation} \end{lemma}  \begin{proof} Firstly, we note that $Z^{N,f}_t$ are martingales by standard results from Markov chains, (\ref{eq: M is for martingale}). Observe that the  integrand in the definition of $C^{N,f}_t$ is bounded, since whenever $0\leq u \leq s$, and $\mu^N$ is obtain from $\mu^N_{u-}$ by collision, we use the estimate (\ref{eq: Lipschitz continuity of Q}) with $\eta=\frac{1}{2}$, to obtain for some $k$ \begin{equation} \begin{split} |\langle f, \chi(u,s,\mu^N)\rangle | & \leq \| Q_{s-u}(\mu^N)-Q_{s-u}(\mu^N_{u-})\|_{\mathrm{TV}+q} \\[1ex] & \lesssim\Lambda_{k}(\mu^N, \mu^N_{u-})^\frac{1}{2}N^{-\frac{1}{2}} \lesssim N^{\frac{k-2}{4}} <\infty. \end{split} \end{equation} Moreover, for initial data $\mu^N \in \mathcal{S}_N$, the Boltzmann flow $(\phi_s(\mu^N))_{s=0}^t$ has uniformly bounded $(q+1)^\text{th}$ moments and so, by approximation, the Boltzmann dynamics (\ref{BE}) extend to $f$. Now, we apply Fubini to  the integral: \begin{equation} \begin{split}  &C^{N,f}_t \\  &= \int_{(0,t]\times \mathcal{S}_N} \int_0^t ds\hspace{0.1cm} \langle f, Q_{s-u}(\mu^N)-Q_{s-u}(\mu^N_{u-})\rangle \hspace{0.1cm} 1[u \le s \le t] \hspace{0.1cm} (m^N-\overline{m}^N)(du, d\mu^N)  \\&  = \int_{(0,t]\times \mathcal{S}_N} \left\{ \int_u^t  \left(\langle f, Q_{s-u}(\mu^N)\rangle -\langle f, Q_{s-u}(\mu^N_{u-})\rangle \right) ds\right\} (m^N-\overline{m}^N)(du, d\mu^N) \\ & =\int_{(0,t]\times\mathcal{S}_N} \left\{\langle f, \phi_{t-u}(\mu^N)-\phi_{t-u}(\mu^N_{u-})\rangle -\langle f, \mu^N-\mu^N_{u-}\rangle\right\}(m^N-\overline{m}^N)(du, d\mu^N) \\& \hspace{2cm} =: M^{N,f}_t - Z^{N,f}_t \end{split}\end{equation} where the third equality is precisely the (extended) Boltzmann dynamics (\ref{BE}) in the variable $s \in[u,t]$.  \end{proof} 
To prove Lemma \ref{thrm: local uniform martingale control}, we return to the decomposition (\ref{eq: decomposition of MNFT}) used in the proof of Lemma \ref{thrm: pointwise martingale control}. Our first point is to establish a control on \begin{equation} \mathbb{E} \left[ \sum_{j=0}^J \sum_{B\in \mathcal{P}_{j,l}} \left\{ M^{N;B}_{\star,t_\text{fin}}\right\}^p\right]\end{equation} where $\star$ denotes the local supremum (\ref{eq: use of star}). We will do so by breaking the supremum into two parts, each of which can be controlled by elementary martingale estimates. Let $(J^{N;B;t}_s)_{0\le s\le t}$ be the process \begin{equation} J^{N;B;t}_s = \int_{(0,s]\times\mathcal{S}_N} \langle h_B, Q_{t-u}(\mu^N)-Q_{t-u}(\mu^N_{u-})\rangle (m^N-\overline{m}^N)(du, d\mu^N)\end{equation} where, as in the proof of Theorem \ref{thrm: PW convergence}, \begin{equation}\label{eq: defn of hB} h_B=2^{2j}(1+|v|^2)1_B; \hspace{1cm} B \in \mathcal{P}_{j,l}. \end{equation} Each process $(J^{N;B:t}_s)_{0\le s\le t}$ is a martingale, by standard results for Markov chains (\ref{eq: M is for martingale}). Writing $Z^{N;B}=Z^{N,h_B}$, Lemma \ref{lemma:newmartingaleconstruction} gives \begin{equation} Z^{N;B}_t = M^{N;B}_t +\int_0^t J^{N;B;s}_s \hspace{0.1cm} ds.  \end{equation}  \begin{lemma} \label{lemma:break up LU} Let $p\geq 2$, and let $p'$ be the H\"{o}lder conjugate to $p$. In the notation above, we have the comparison \begin{equation}  \mathbb{E} \left[ \sum_{j=0}^J \sum_{B\in \mathcal{P}_{j,l}} \left\{ \left|M^{N;B}_{\star,t_\text{fin}}\right|\right\}^p\right] \lesssim \mathbb{E} \left[ \sum_{j=0}^J \sum_{B\in \mathcal{P}_{j,l}} \left\{\hspace{0.1cm} \left|M^{N;B}_{t_\text{fin}}\right|^p+t_\text{fin}^{p/p'}\int_0^{t_\text{fin}} \left|J^{N;B;t}_t \right|^p dt\right\}\right].  \end{equation} \end{lemma}  \begin{proof} For each $B$, we observe that \begin{equation} \begin{split} \sup_{t\le t_\text{fin}} \left|M^{N;B}_t - Z^{N;B}_t\right| \le \int_0^{t_\text{fin}}\left|J^{N;B;s}_s\right| ds\end{split} \end{equation} which implies the two bounds \begin{equation} \label{eq: comparing ZNB and MNB} M^{N;B}_{\star,t_\text{fin}} \le Z^{N;B}_{\star, t_\text{fin}} +\int_0^{t_\text{fin}} \left|J^{N;B;s}_s\right| ds; \hspace{1cm} Z^{N;B}_{t_\text{fin}} \le M^{N;B}_{t_\text{fin}}+\int_0^{t_\text{fin}} \left|J^{N;B;s}_s\right| ds.\end{equation} By Doob's $L^p$ inequality, we have \begin{equation} \label{eq: Doob LP} \begin{split} \left\|\hspace{0.1cm} Z^{N;B}_{\star,t_\text{fin}}\hspace{0.1cm}\right\|_{L^p(\mathbb{P})} &\leq p'\hspace{0.1cm}\left\|\hspace{0.1cm} Z^{N;B}_{t_\text{fin}}\hspace{0.1cm}\right\|_{L^p(\mathbb{P})}\end{split}. \end{equation} Combining (\ref{eq: comparing ZNB and MNB}, \ref{eq: Doob LP}), we obtain \begin{equation} \left\|\hspace{0.1cm} M^{N;B}_{\star,t_\text{fin}}\hspace{0.1cm}\right\|_{L^p(\mathbb{P})} \lesssim \left\|\hspace{0.1cm} M^{N;B}_{t_\text{fin}}\hspace{0.1cm}\right\|_{L^p(\mathbb{P})} + \left\|\hspace{0.1cm}\int_0^{t_\text{fin}}\left|J^{N;B;s}_s\right|ds \hspace{0.1cm}\right\|_{L^p(\mathbb{P})}.\end{equation} Using H\"{o}lder's inequality on the integral,\begin{equation} \begin{split} \mathbb{E} \left[\left\{\hspace{0.1cm} M^{N;B}_{\star,t_\text{fin}}\right\}^p \right] &\lesssim \mathbb{E}\left[ \hspace{0.1cm}\left|M^{N;B}_{t_\text{fin}}\right|^p \hspace{0.1cm}\right] + \mathbb{E}\left[\hspace{0.1cm}\left\{\int_0^{t_\text{fin}} \left|J^{N;B;s}_s\right| \hspace{0.1cm} ds \right\}^p\hspace{0.1cm}\right] \\ & \lesssim \mathbb{E}\left[ \hspace{0.1cm}\left|M^{N;B}_{t_\text{fin}}\right|^p \hspace{0.1cm}\right] +t_\text{fin}^{p/p'} \int_0^{t_\text{fin}} \mathbb{E}\left[\hspace{0.1cm}\left|J^{N;B;t}_t\right|^p \hspace{0.1cm}\right]\hspace{0.1cm} ds .\end{split} \end{equation} Summing over $B\in \mathcal{P}_{j,l}$ and $j=0,1,\dots ,J$, we obtain the desired comparison. \end{proof}  

   \begin{proof}[Proof of Lemma \ref{thrm: local uniform martingale control}] We begin by controlling the integral term in Lemma \ref{lemma:break up LU}. The quadratic variation is given by \begin{equation} \begin{split} \left[J^{N;B;t}\right]_s &= \int_{(0,s]\times \mathcal{S}_N} \langle h_B, \chi(u,t,\mu^N)\rangle^2 m^N(du, d\mu^N) \\& \le \int_{(0,s]\times \mathcal{S}_N} \langle h_B, |\chi(u,t,\mu^N)|\rangle^2 m^N(du, d\mu^N) \end{split}  \end{equation} where $h_B$ is as in (\ref{eq: defn of hB}) and $\chi$ is as in (\ref{eq: defn of chi}). Hence, using Burkholder's inequality (\ref{lemma:Burkholder}) we see that, for all $t\leq t_\text{fin}$,  \begin{equation} \begin{split} &\mathbb{E} \left[\sum_{j=0}^J \sum_{B\in \mathcal{P}_{j,l}} \left\{\hspace{0.1cm} \left|J^{N;B;t}_t\right|\right\}^p\right]  \\ & \hspace{2cm}\lesssim \mathbb{E} \left[\sum_{j=0}^J \sum_{B\in \mathcal{P}_{j,l}}  \left\{\int_{(0,t]\times \mathcal{S}_N}\langle h_B, |\chi(u,t,\mu^N)|\rangle^2 m^N(du, d\mu^N)\right\}^{p/2}\right].   \end{split} \end{equation} Using Minkowski's inequality to move the double sum inside the parentheses, and recalling that $\sum_j \sum_{B\in \mathcal{P}_{j,l}} h_B \lesssim (1+|v|^4)$, we obtain the bound \begin{equation}\label{eq: LU mg bound on J}\begin{split} &\mathbb{E} \left[\sum_{j=0}^J \sum_{B\in \mathcal{P}_{j,l}} \left\{\hspace{0.1cm} \left|J^{N;B;t}_t\right|\right\}^p\right]\\  &\hspace{0.5cm}\lesssim \mathbb{E}\left[  \left\{ \int_{(0,t]\times \mathcal{S}_N} \langle 1+|v|^4, |\chi(u,t,\mu^N)|\rangle^2 m^N(du, d\mu^N)  \right\}^{p/2} \right] \\ &\hspace{0.5cm}\lesssim \mathbb{E} \left[\left\{\int_{(0,t]\times \mathcal{S}_N} \|Q_{t-u}(\mu^N)-Q_{t-u}(\mu^N_{u-})\|_{\mathrm{TV}+4}^2 \hspace{0.1cm} m^N(du, d\mu^N)\right\}^{p/2}\right] \end{split} \end{equation} where the second equality is the definition of $\chi$ (\ref{eq: defn of chi}). \medskip \\ Using the continuity estimate for $Q$ established in (\ref{eq: Lipschitz continuity of Q}), and arguing as in the proof of Lemma \ref{thrm: pointwise martingale control}, we see that almost surely, for $m^N$-almost all $(u, \mu^N)$, we have \begin{equation} \|Q_{t-u}(\mu^N)-Q_{t-u}(\mu^N_{u-})\|_{\mathrm{TV}+4} \lesssim N^{\epsilon-1}\Lambda_k(\mu^N_{u-}).\end{equation} Therefore, using Cauchy-Schwarz, (\ref{eq: LU mg bound on J}) gives the bound \begin{equation} \label{eq: LU mg bound on J 2} \begin{split} & \mathbb{E} \left[\sum_{j=0}^J \sum_{B\in \mathcal{P}_{j,l}} \left\{\left|J^{N;B;t}_t\right|\right\}^p\right]  \\& \hspace{1cm}\lesssim N^{p(\epsilon-1)} \hspace{0.1cm} \mathbb{E}\left[\sup_{t\leq t_\text{fin}} \Lambda_{kp}(\mu^N_t)\right]^{1/2} \left\|m^N\left((0,t_\text{fin}]\times \mathcal{S}_N\right)\right\|_{L^p(\mathbb{P})}^{p/2}.\end{split}\end{equation} The moment term is controlled by the initial moment bound and Proposition \ref{thrm:momentinequalities} : \begin{equation} \label{eq: LU mg moment bound} \mathbb{E}\left[\sup_{t\leq t_\text{fin}} \Lambda_{kp}(\mu^N_t)\right] \lesssim (1+t_\text{fin})\Lambda_{kp}(\mu^N_0)\leq (1+t_\text{fin})a^p.\end{equation}  Since the rates of the Kac process are bounded by $2N$, we can stochastically dominate $m^N(dt \times \mathcal{S}_N)$ by a Poisson random measure $\mathfrak{m}^N(dt)$ of rate $2N$. By the additive property of Poisson processes, it follows that \begin{equation} \label{eq: Poisson} \|m^N((0,t_\text{fin}]\times \mathcal{S}_N) \|_{L^p(\mathbb{P})}\le \|\mathfrak{m}^N(0,t_\text{fin}]\|_{L^p(\mathbb{P})} \lesssim N(1+t_\text{fin}). \end{equation} Combining (\ref{eq: LU mg bound on J 2}, \ref{eq: LU mg moment bound}, \ref{eq: Poisson}), we have the control of the integrand: \begin{equation} \begin{split} \sup_{t\leq t_\text{fin}} \hspace{0.2cm} \mathbb{E} \left[\sum_{j=0}^J \sum_{B\in \mathcal{P}_{j,l}} \left\{ \left|J^{N;B;t}_t\right|\right\}^p\right] \lesssim N^{p(\epsilon-1/2)}\hspace{0.1cm} a^{p/2} (1+t_\text{fin})^{\frac{p+1}{2}}. \end{split} \end{equation} This gives the following control of the integral term in Lemma \ref{lemma:break up LU}: \begin{equation} \label{eq:control of integral term} \begin{split} t_\text{fin}^{p/p'} \hspace{0.2cm} \mathbb{E} \left[\sum_{j=0}^J \sum_{B\in \mathcal{P}_{j,l}} \mathlarger{\mathlarger{\int}}_0^{t_\text{fin}}\left\{\left|J^{N;B;t}_t\right|\right\}^p dt \right] \lesssim N^{p(\epsilon-1/2)}\hspace{0.1cm} a^{p/2} (1+t_\text{fin})^{\frac{p+3}{2}+\frac{p}{p'}}. \end{split} \end{equation} Using the definition of $p'$ as the H\"older conjugate to $p$, it is straightforward to see that the exponent of $(1+t_\text{fin})$ is $\frac{3p+1}{2}$. \medskip \\ We now perform a similar analysis for the terms $M^{N;B}_{t_\text{fin}}$ in Lemma \ref{lemma:break up LU}. Let $(M^{N;B;t}_s)_{s\leq t}$ be the martingale defined in (\ref{eq: MNBTS}). The quadratic variation is \begin{equation} \begin{split} \left[M^{N;B;t}\right]_s&=\int_{(0,s]\times\mathcal{S}_N}\langle h_B, \phi_{t-u}(\mu^N)-\phi_{t-u}(\mu^N_{u-})\rangle^2 \hspace{0.1cm} m^N(du, d\mu^N) \\ & \leq \int_{(0,s]\times\mathcal{S}_N}\langle h_B, |\phi_{t-u}(\mu^N)-\phi_{t-u}(\mu^N_{u-})|\rangle^2 \hspace{0.1cm} m^N(du, d\mu^N). \end{split} \end{equation}  Arguing using Burkholder and the stability estimate Corollary \ref{cor: new stability for BE}, an identical calculation to the above shows that  \begin{equation} \label{eq: control of M at terminal time}  \sum_{j=0}^J \sum_{B\in \mathcal{P}_{j,l}} \left\|M^{N;B}_{{t_\text{fin}}}\right\|_{L^p(\mathbb{P})}^p   \lesssim N^{p(\epsilon-1/2)}\hspace{0.1cm}a^{p/2}\hspace{0.1cm}(1+t_\text{fin})^{\frac{p+1}{2}}. \end{equation}    Hence, by Lemma \ref{lemma:break up LU}, we obtain \begin{equation} \mathbb{E} \left[ \sum_{j=0}^J \sum_{B\in \mathcal{P}_{j,l}} \left\{ \left|M^{N;B}_{\star,t_\text{fin}}\right|\right\}^p\right] \lesssim N^{p(\epsilon-1/2)}a^{p/2}(1+t_\text{fin})^\frac{3p+1}{2}.\end{equation} We control the coefficients $2^{-2j}a_B(f)$ as in the argument of Lemma \ref{thrm: pointwise martingale control}. Using H\"{o}lder's inequality in place of Cauchy-Schwarz, we obtain \begin{equation} \begin{split} \label{eq: control of mg term LU}  &\left\|\hspace{0.1cm}\sup_{f\in \mathcal{A}} \hspace{0.1cm} \sup_{t\leq t_\text{fin}} \hspace{0.1cm} \left|\sum_{j=0}^J\sum_{l=2}^L \sum_{B \in \mathcal{P}_{j,l}} 2^{-2j}a_B(f)M^{N;B}_t \hspace{0.1cm} \right|  \hspace{0.1cm}\right\|_{L^p(\mathbb{P})}  \\& \hspace{2cm} \lesssim \sum_{l=2}^L   \left[ \mathbb{E}\sum_{j=0}^J \sum_{B \in \mathcal{P}_{j,l}} \left\{M^{N;B}_{\star,t_\text{fin}}\right\}^p\right]^{1/p}2^{(d/p'-1)l}J^{1/p'}  \\& \hspace{2cm}  \lesssim \sum_{l=2}^L  N^{\epsilon-\frac{1}{2}} \hspace{0.1cm} a^{1/2} \hspace{0.1cm} (1+t_\text{fin})^{\frac{3p+1}{2p}}\hspace{0.1cm}  2^{(d/p'-1)l}J^{1/p'} \\ & \hspace{2cm} \lesssim N^{\epsilon-\frac{1}{2}} \hspace{0.1cm} a^{1/2} \hspace{0.1cm} (1+t_\text{fin})^{\frac{3p+1}{2p}}\hspace{0.1cm}2^{(d/p'-1)L} \hspace{0.1cm}   J^{1/p'}. \end{split} \end{equation} Following the argument of Lemma \ref{thrm: pointwise martingale control}, we wish to control the error terms $R^{N,f}_t$ given by (\ref{eq: definition of RNFT}), locally uniformly in time. As in (\ref{eq: majorise integrand of error term}), we majorise, for $m^N+\overline{m}^N$-almost all $(s, \mu^N)$, \begin{equation} \begin{split}  \sup_{f\in \mathcal{A}} \left|\langle \beta(f), \phi_{t-s}(\mu^N)-\phi_{t-s}(\mu^N_{s-})\rangle \right| &\lesssim (2^{-2J}+2^{-L}) \hspace{0.1cm}N^{\epsilon-1} \hspace{0.1cm} \Lambda_{k}(\mu^N_{s-})^\frac{1}{2} \\& =:H'_s. \end{split} \end{equation} As in (\ref{eq: dominate integrand and integrator seperately}), we may bound  \begin{equation}  \label{eq: break up error term}  \left\|\hspace{0.1cm} \sup_{f\in\mathcal{A}}\hspace{0.1cm} \sup_{t\le t_\text{fin}}\left|R^{N,f}_t\right|\hspace{0.1cm}\right\|_{L^p(\mathbb{P})} \le \left\|\hspace{0.1cm} \int_0^{t_\text{fin}}H'_s(m^N+\overline{m}^N)(ds, \mathcal{S}_N)\hspace{0.1cm}\right\|_{L^p(\mathbb{P})}  \le \mathcal{T}_1+\mathcal{T}_2 \end{equation} where the two error terms are \begin{equation} \mathcal{T}_1= \left\|\int_0^{t_\text{fin}} H'_s \hspace{0.1cm}m^N(ds, \mathcal{S}_N)\right\|_{L^p(\mathbb{P})}\end{equation} and \begin{equation} \mathcal{T}_2=\left\|\int_0^{t_\text{fin}} H'_s \hspace{0.1cm}\overline{m}^N(ds, \mathcal{S}_N)\right\|_{L^p(\mathbb{P})}. \end{equation}  We now deal with the two terms separately. For the $\mathcal{T}_1$, we dominate $\overline{m}^N(ds, \mathcal{S}_N)\leq 2N ds$ to see that \begin{equation} \begin{split}    \int_0^{t_\text{fin}} H'_s\hspace{0.1cm} \overline{m}^N(ds, \mathcal{S}_N)  \lesssim (2^{-2J}+2^{-L})\hspace{0.1cm}N^{\epsilon}\hspace{0.1cm}t_\text{fin}\hspace{0.1cm}\left( \hspace{0.1cm} \sup_{s\leq t_\text{fin}} \Lambda_{k}(\mu^N_s)^\frac{1}{2}\right). \end{split} \end{equation}  Using the monotonicity of $L^p$ norms, and using the moment control in the usual way, \begin{equation} \begin{split} \label{eq: control of mbar integral} \mathcal{T}_1 &\lesssim (2^{-2J}+2^{-L})\hspace{0.1cm}N^{\epsilon}\hspace{0.1cm}t_\text{fin}\hspace{0.1cm}\mathbb{E}\left[ \hspace{0.1cm} \sup_{s\leq t_\text{fin}} \Lambda_{pk}(\mu^N_s)\right]^\frac{1}{2p} \\[1ex] &  \lesssim (2^{-2J}+2^{-L})\hspace{0.1cm}N^{\epsilon}\hspace{0.1cm}a^{1/2}\hspace{0.1cm}(1+t_\text{fin})^\frac{2p+1}{2p}. \end{split} \end{equation} For $\mathcal{T}_2$, we dominate $m^N(ds, \mathcal{S}_N)$ by a Poisson random measure $\mathfrak{m}^N(ds)$ of rate $2N$, as above. Controlling $\mathfrak{m}^N$ as in (\ref{eq: Poisson}), we obtain \begin{equation} \begin{split} \label{eq: control of m integral}   \mathcal{T}_2 & \lesssim (2^{-2J}+2^{-L})N^{\epsilon-1} \left\|\int_0^{t_\text{fin}} \Lambda_{k}(\mu^N_{s-})^\frac{1}{2}\mathfrak{m}^N(ds) \right\|_{L^p(\mathbb{P})} \\[1ex]&   \lesssim (2^{-2J}+2^{-L})N^{\epsilon-1} \left\|\left(\hspace{0.1cm}\sup_{s\leq t_\text{fin}} \Lambda_{k}(\mu^N_s)^\frac{1}{2}\right)\right\|_{L^{2p}(\mathbb{P})}\left\|\mathfrak{m}^N\left(\left(0, t_\text{fin}\right]\right)\right\|_{L^{2p}(\mathbb{P})} \\[1ex]& = \lesssim (2^{-2J}+2^{-L})N^{\epsilon}(1+t_\text{fin})^\frac{2p+1}{2p}. \end{split} \end{equation} Combining the local uniform estimates (\ref{eq: control of mg term LU}, \ref{eq: break up error term}, \ref{eq: control of mbar integral}, \ref{eq: control of m integral}) of the terms in the decomposition (\ref{eq: decomposition of MNFT}), we find that \begin{multline*} \left\|\hspace{0.1cm} \sup_{f\in \mathcal{A}} \hspace{0.1cm}M^{N,f}_{\star,t_\text{fin}} \hspace{0.1cm} \right\|_{L^p(\mathbb{P})}  \lesssim N^\epsilon a^{1/2} \hspace{0.1cm} (1+t_\text{fin})^\frac{3p+1}{2p}\hspace{0.1cm} \left(N^{-1/2} \hspace{0.1cm} 2^{(d/q-1)L}J^{1/p'}+2^{-2J}+2^{-L}\right). \end{multline*} Taking $J=\lfloor\frac{p'}{4d}\log_2(N)\rfloor$ and $L=\lfloor \frac{p'}{2d}\log_2(N)\rfloor$ proves the result claimed. \end{proof}

   \section{Proof of Theorem \ref{thm: low moment regime}} \label{sec: LMR}

We now turn to the proof of Theorem \ref{thm: low moment regime}, which establishes a convergence estimate in the presence of a $k^\text{th}$ moment bound, for any $k>2$. Our strategy will be to use the ideas of \cite{ACE}, which work well with few moments, to prove convergence on a small initial time interval $[0,u_N]$, for some $u_N$ to be chosen later. Then, thanks to the moment production property recalled in Proposition \ref{thrm:momentinequalities}, we may use Theorems \ref{thrm: PW convergence}, \ref{thrm: Main Local Uniform Estimate} to control the behaviour at times $t\ge u_N$. The argument is similar to the final argument in the proof of Theorem \ref{thrm: W-W continuity of phit} given in Section \ref{sec: continuity of BE}, which may be read as a warm-up to this proof. \medskip \\ Throughout, let $k, a, (\mu^N_t), \mu_0$ be as in the statement of the Theorem.   \medskip \\  We begin by recalling the representation formula established in \cite[Proposition 4.2]{ACE}, which is a noisy version of Proposition \ref{prop: bad representation formula}. \begin{proposition} \label{prop: very bad rep formula} Let $\mu \in \mathcal{S}^k$ for some $k>2$, and let $\mu^N_t$ be a Kac process on $N$ particles. Let $\rho_t=(\phi_t(\mu_0)+\mu^N_t)/2$, and for $f\in \mathcal{A}, 0\le s\le t$, let $f_{st}$ be the propagation described in Definition \ref{def: LKP} in this environment. Then, for all $t\ge 0$, we have the equality \begin{equation} \begin{split} & \langle f, \mu^N_t-\phi_t(\mu_0)\rangle =\langle f_{0t}, \mu^N_0-\mu_0\rangle \\&\hspace{2cm}+\int_{(0,t]\times \mathcal{S}_N} \langle f_{st},\mu^N-\mu^N_{s-}\rangle (m^N-\overline{m}^N)(ds,d\mu^N) \end{split} \end{equation} where $m^N, \overline{m}^N$ are as defined in Section \ref{sec: interpolation decomposition}.\end{proposition} The major difficulty in using this representation formula is the appearance of an exponentiated random moment in the quantity $z_t$ parametrising the continuity of $f_{st}$. We will use the following proposition, which controls the stochastic integrals on the right-hand side, modulo this difficulty. \begin{proposition} \label{prop: short time mg estimate} Let $\rho_t$ be a potentially random environment such that, for some $\beta>0$, \begin{equation} \label{eq: moment condition for environment} w=\left\|\hspace{0.1cm}\sup_{t\le 1} \left(\frac{ \Lambda_3(\rho_t)}{\beta t^{\beta-1}+1}\right)\hspace{0.1cm}\right\|_{L^\infty(\mathbb{P})}<\infty. \end{equation} For $f\in \mathcal{A}$ and $0\le s\le t\le 1$, let $f_{st}[\rho]$ denote the propagation in this environment, as described in Definition \ref{def: LKP}. \medskip \\ Let $k>2$ and $a\ge 1$, and let $\mu^N_t$ be a Kac process with initial moment $\Lambda_k(\mu^N_0)\le a$, and let $m^N, \overline{m}^N$ be as in Section \ref{sec: interpolation decomposition}. We write \begin{equation} \widetilde{M}^{N,f}_t[\rho]=\int_{(0,t]\times \mathcal{S}_N} \langle f_{st}[\rho], \mu^N-\mu^N_{s-}\rangle (m^N-\overline{m}^N)(ds, d\mu^N).\end{equation} In this notation, we have the bound \begin{equation} \left\|\hspace{0.1cm} \sup_{t\le 1}\hspace{0.1cm} \sup_{f\in \mathcal{A}} \hspace{0.1cm} \widetilde{M}^{N,f}_t[\rho]\right\|_1 \le CaN^{-\eta}\end{equation} for some $C=C(d,k,\beta)$ and $\eta=\eta(d,\beta)>0$. Here, we emphasise that $\|\cdot\|_{L^1(\mathbb{P})}$ refers to the $L^1$ norm with simultaneous expectation over $\mu^N_t$ and the environment $\rho$. \end{proposition} This largely follows from the proof of \cite[Theorem 1.1]{ACE}, and the argument follows a similar pattern to Lemmas \ref{thrm: pointwise martingale control}, \ref{thrm: local uniform martingale control}, using the continuity estimate recalled in Proposition \ref{prop: continuity for branching process} and a similar estimate for the dependence on the initial time $s$. The key difference is that the hypotheses on the environment $\rho$ guarantee an $L^\infty(\mathbb{P})$ control on the quantities \begin{equation} z_1=\exp\left(8\int_0^1 \Lambda_3(\rho_u)du\right); \hspace{0.8cm} y_\beta=z_1\hspace{0.1cm}\sup_{0\le s\le s'\le 1}\left[(s'-s)^{-\beta}\int_s^{s'}\Lambda_3(\rho_u)du\right]\end{equation} which describe the continuity of $f_{st}(v)$ in $v$ and $s$ respectively. By contrast, these are only controlled in probability in \cite[Theorem 1.1]{ACE}; correspondingly, we obtain an $L^1(\mathbb{P})$ estimate rather than an estimate in probability. With this estimate, we turn to the proof of Theorem \ref{thm: low moment regime}.
\begin{proof}[Proof of Theorem \ref{thm: low moment regime}] We first introduce a localisation argument, following the argument in Section \ref{sec: continuity of BE}, which allows us to guarantee that (\ref{eq: moment condition for environment}) holds for the environment $\rho= (\mu^N_t+\phi_t(\mu_0))/2$. Let $\beta=\frac{k-2}{2}$, and let $u_N \le 1$ be chosen later. Now, define $T_N$ to be the stopping time \begin{equation} T_N=\inf\left\{t\le u_N: \Lambda_3(\rho_t) > \frac{(\beta t^{\beta-1}+1)}{8\sqrt{2}}\right\}. \end{equation} We use the convention that $\inf \emptyset =\infty$, so that if $T_N>u_N$, then $T_N=\infty$. Let $\rho^T$ be the stopped environment $\rho^T_t=\rho_{t\land T_N}$, and write $f_{st}^T$ for the propagation in the stopped environment.\medskip \\  We observe first that on the event $T_N=\infty$, we have the equality $f_{st}^T=f_{st}$ for all $f\in \mathcal{A}, s\le t\le u_N$. Moreover, since $\Lambda_3(\rho_t)$ increases by a factor of at most $4\sqrt{2}$ at jumps by Lemma \ref{lemma:momentincreaseatcollision}, we have the bound, almost surely for all $t\ge 0$,  \begin{equation} \Lambda_3(\rho^T_t) \le \frac{(\beta t^{\beta-1}+1)}{2}.\end{equation} Therefore, the stopped environment $\rho^T$ satisfies the bound \ref{eq: moment condition for environment} with $w=\frac{1}{2}$. Now, we write $\widetilde{M}^{N,f}_t=\widetilde{M}^{N,f}_t[\rho^T]$ as in the proposition above, and by the representation formula in Proposition \ref{prop: very bad rep formula}, we have the bound for all $t\le u_N$,\begin{equation}\begin{split} &W\left(\mu^N_t,\phi_t(\mu_0)\right) 1[T_N=\infty] \le CW(\mu^N_0, \mu_0)+\sup_{f\in \mathcal{A}}\hspace{0.2cm}\widetilde{M}^{N,f}_t \end{split}\end{equation} for some absolute constant $C$. By Proposition \ref{prop: short time mg estimate}, we obtain the estimate \begin{equation} \label{eq: shorttime bound 1} \left\|\sup_{t\le u_N} W\left(\mu^N_t, \phi_t(\mu_0)\right)1[T_N=\infty]\right\|_1 \lesssim W(\mu^N_0,\mu_0)+aN^{-\eta}.\end{equation} Let $k_0=k_0(d)$ be large enough that Theorem \ref{thrm: PW convergence} holds with $\epsilon=\frac{1}{2d}$. By applying Theorem \ref{thrm: PW convergence}, restarted at time $u_N$, and the moment production property, we obtain \begin{equation} \label{eq: restarted estimate} \begin{split}\sup_{t\ge u_N} \left\|W(\mu^N_t,\phi_{t-u_N}(\mu^N_{u_N}))\right\|_2&\lesssim N^{\epsilon-1/d}\hspace{0.1cm}\mathbb{E}\left[\Lambda_{k_0}(\mu^N_{u_N})\right]^{1/2}  \\ & \lesssim N^{\epsilon-1/d}u_N^{1-k_0/2}.\end{split} \end{equation} Using our continuity estimate Theorem \ref{thrm: W-W continuity of phit}, we have the bound for some $\zeta=\zeta(d)$\begin{equation} \begin{split}  &\sup_{t\ge u_N} W(\phi_{t-u_N}(\mu^N_{u_N}),\phi_t(\mu_0)) \\&\hspace{1cm}\lesssim W(\mu^N_{u_N},\phi_{u_N}(\mu_0))^\zeta\Lambda_{k_0}(\mu^N_{u_N},\phi_{u_N}(\mu_0))\end{split}\end{equation} and, considering the cases $\{T_N\le u_N\}, \{T_N=\infty\}$ separately, we see that \begin{equation} \begin{split}  &\sup_{t\ge u_N} W(\phi_{t-u_N}(\mu^N_{u_N}),\phi_t(\mu_0)) \\&\hspace{2.5cm}\lesssim W(\mu^N_{u_N},\phi_{u_N}(\mu_0))^\zeta\Lambda_{k_0}(\mu^N_{u_N},\phi_{u_N}(\mu_0))1[T_N=\infty]\\[1ex]&\hspace{3.5cm}+ 1[T_N\le u_N].\end{split}\end{equation}To ease notation, we will write $\mathcal{T}_1, \mathcal{T}_2$ for the two terms respectively. We estimate the expectation of $\mathcal{T}_1$ using H\"older's inequality: for some $k_1>k_0$, \begin{equation}\label{eq: holder estimate on BF}\begin{split} &\left\|\mathcal{T}_1\right\|_{L^1(\mathbb{P})}  \lesssim\mathbb{E}\left(W(\mu^N_{u_N},\phi_{u_N}(\mu_0))1[T_N=\infty]\right)^\zeta\mathbb{E}\left(\Lambda_{k_1}(\mu^N_{u_N},\phi_{u_N}(\mu_0))\right) \\[1ex] & \hspace{3cm}\lesssim (N^{-\eta}+W(\mu^N_0,\mu_0))^\zeta\hspace{0.1cm} u_N^{1-k_1/2}.\end{split}\end{equation} where $\eta$ is as in (\ref{eq: shorttime bound 1}) with our choice of $\beta$. In order to deal with $\mathcal{T}_2$, we now estimate $\mathbb{P}(T_N\le u_N)$. Let $Z_N$ be given by \begin{equation} Z_N=\sum_{l:2^{-l}\le u_N} 2^{(\beta-1)l+1}\beta^{-1}\sup_{t\in [2^{-l},2^{1-l}]}\langle 1+|v|^3, \rho_t\rangle \end{equation} and observe that, for all $t\le u_N$, we have the bound \begin{equation} \langle 1+|v|^3, \rho_t\rangle \le \frac{(\beta t^{\beta-1}+1)Z_N}{2}. \end{equation} Therefore, \begin{equation} \mathbb{P}(T_N\le u_N)\le \mathbb{P}(Z_N>1/8) \le 8 \mathbb{E}[Z_N].\end{equation} Using the moment production property of the Kac process and Boltzmann equation in Proposition \ref{thrm:momentinequalities}, we compute \begin{equation} \mathbb{E}(Z_N)\le \sum_{l: 2^{-l}\le u_N}2^{(\beta-1)l+1}2^{-l(k-3)}\hspace{0.1cm} \beta^{-1}a \lesssim a u_N^\beta \end{equation} and so \begin{equation} \label{eq: estimate on restarted BF} \begin{split}  &\left\|\sup_{t\ge u_N} W(\phi_{t-u_N}(\mu^N_{u_N}),\phi_t(\mu_0))\right\|_{L^1(\mathbb{P})} \\[1ex]& \hspace{1.5cm}\lesssim (N^{-\eta}+W(\mu^N_0,\mu_0))^\zeta\hspace{0.1cm} u_N^{1-k_1/2} +au_N^\beta. \end{split} \end{equation} We now return to (\ref{eq: shorttime bound 1}) and observe that \begin{equation} \label{eq: shorttime bound 2} \begin{split}& \left\|\sup_{t\le u_N} W(\mu^N_t, \phi_t(\mu_0))\right\|_1  \\& \hspace{2cm}\lesssim \left\|\sup_{t\le u_N} W\left(\mu^N_t, \phi_t(\mu_0)\right)1[T_N=\infty]\right\|_1 +\mathbb{P}(T_N\le u_N) \\[1ex] & \hspace{2cm} \lesssim W(\mu^N_0,\mu_0)+aN^{-\eta}+au_N^\beta. \end{split} \end{equation} Combining (\ref{eq: restarted estimate}, \ref{eq: estimate on restarted BF}, \ref{eq: shorttime bound 2}) and keeping the worst terms, we have shown that \begin{equation} \sup_{t\ge 0} \left\|W(\mu^N_t, \phi_t(\mu_0)\right\|_{L^1(\mathbb{P})}\lesssim (N^{-\eta}+W(\mu^N_0,\mu_0))^{\delta} u_N^{-\alpha}+ au_N^\beta\end{equation} for some $\eta, \delta, \alpha, \beta>0$, depending on $d, k$. If we choose \begin{equation} u_N=(N^{-\eta}+W(\mu^N_0,\mu_0))^{\delta/(\alpha+\beta)}\end{equation} then we finally obtain \begin{equation}\begin{split} \sup_{t\ge 0} \left\|W(\mu^N_t, \phi_t(\mu_0))\right\|_{L^1(\mathbb{P})}& \lesssim a(N^{-\eta}+W(\mu^N_0,\mu_0))^{-\beta \delta/(\alpha+\beta)} \\ & \lesssim a\left(N^{-\eta \beta \delta/(\alpha+\beta)}+W(\mu^N_0,\mu_0)^{\beta \delta/(\alpha+\beta)}\right) \\ & \lesssim  a\left(N^{-\epsilon}+W(\mu^N_0,\mu_0)^\epsilon\right)\end{split} \end{equation} as desired, for sufficiently small $\epsilon=\epsilon(d,k)>0$. The case for the local uniform estimate is similar, using Theorem \ref{thrm: Main Local Uniform Estimate} in place of Theorem \ref{thrm: PW convergence}. \end{proof}

\section{Proof of Theorem \ref{thrm: No Uniform Estimate}}
The proof of Theorem \ref{thrm: No Uniform Estimate} is based on the following heuristic argument: \begin{heuristic} Fix $N$, and consider a Kac process $(\mu^N_t)$ on $N$ particles. As $t\rightarrow \infty$, its law relaxes to the equilibrium distribution $\pi_N$, which is known to be the uniform distribution $\sigma^N$ on $\mathcal{S}_N$. Since this measure assigns non-zero probability to regions $R_N$ at macroscopic distance from the fixed point $\gamma$, given by \begin{equation}
\gamma(dv)=\frac{e^{-\frac{d}{2}|v|^2}}{(2\pi d^{-1})^{d/2}}dv,
 \end{equation} the process will almost surely hit $R_N$ on an unbounded set of times. Meanwhile, the Boltzmann flow $\phi_t(\mu_0)$ will converge to $\gamma$. Therefore, at some large time, the particle system $\mu^N_t$ will have macroscopic distance from the Boltzmann flow $\phi_t(\mu^N_0)$.\end{heuristic}
The regions $R_N$ which we construct in the proof are those where the energy is concentrated in only a few particles, which might na\"{i}vely be considered `highly ordered, and so low-entropy'. This appears to contradict the principle that entropy should increase; this \emph{apparent} paradox is explained in the discussion section at the beginning of the paper. \medskip\\ 
We recall that a \emph{labelled} Kac process is the Markov process of velocities $(v_1(t),....,v_N(t))$ corresponding to the particle dynamics. The state space is the set $\mathbb{S}^{N}=\left\{(v_1, ..., v_N) \in (\mathbb{R}^d)^N: \sum_{i=1}^N v_i=0, \sum_{i=1}^N |v_i|^2 = N\right\}$, which we call the labelled Boltzmann Sphere. We denote $\theta_N$ the map taking $(v_1,...,v_N)$ to its empirical measure in $\mathcal{S}_N$: \begin{equation} \theta_N: \mathbb{S}^N \rightarrow \mathcal{S}_N; \hspace{1cm} (v_1, ..., v_N) \mapsto \frac{1}{N} \sum_{i=1}^N \delta_{v_i}.\end{equation} Moreover, if $\mathcal{V}^N_t$ is a labelled Kac process, then the empirical measures $\mu^N_t:= \theta_N(\mathcal{V}^N_t)$ are a Kac process in the sense defined in the introduction.\medskip \\  
Considered as a $((N-1)d-1)$-dimensional sphere, $\mathbb{S}^N$ has a uniform (Hausdorff) distribution $\gamma^N$. We define the `uniform distribution' $\sigma^N$ on $\mathcal{S}_N$ to be the pushforward of $\gamma^N$ by $\theta_N$: \begin{equation} \label{eq: defn of sigmaN} \sigma^N(A):=\gamma^N\left\{(v_1, ...v_N)\in \mathbb{S}^d: \theta_N(v_1,...,v_N)\in A\right\}. \end{equation} We will use this definition to transfer the positivity of the measure $\gamma^N$ forward to $\sigma^N$. \medskip\\
As discussed in the literature review, the problem of relaxation to equilibrium for the Kac process is a subtle problem, and has been extensively studied. For our purposes, the following $L^2$ convergence is sufficient: \begin{proposition}\label{prop: relaxation} Suppose that $(\mu^N_t)_{t\geq 0}$ is a hard-spheres Kac process, where the law of the initial data $\mathcal{L}\mu^N_t$ has a density  $h^N_0\in L^2(\sigma^N)$ with respect to $\sigma^N$. Then at all positive times $t\geq 0$, the law $\mathcal{L}\mu^N_t$ has a density $h^N_t\in L^2(\sigma^N)$ with respect to $\sigma^N$, and for some universal constant $\lambda_0>0$, we have \begin{equation} \left\|h^N_t-1\right\|_{L^2(\sigma^N)} \leq e^{-\lambda_0 t}\left\| h^N_0-1\right\|_{L^2(\sigma^N)}. \end{equation} \end{proposition}  A version of this, for the labelled Kac process, appears as \cite[Theorem 6.8 and corollary]{M+M}; the result stated above follows by a pushforward argument. This is sufficient to prove the following weak ergodic theorem: \begin{lemma}\label{lemma: ergodic theorem} Let $(\mu^N_t)_{t\geq 0}$ be a hard-spheres Kac process on $N$ particles, started from $\mu^N_0\sim \sigma^N$. Let $R_N\subset \mathcal{S}_N$ be such that $p=\sigma^N(R_N)>0$. Then \begin{equation} \frac{1}{t}\int_0^t 1(\mu^N_s\in R_N)ds \rightarrow p\end{equation} in $L^2$. In particular, almost surely, $\mu^N_t$ visits $R_N$ on an unbounded set of times. \end{lemma} \begin{proof} Observe that \begin{equation} \mathbb{E}\left[\frac{1}{t}\int_0^t 1(\mu^N_s \in R_N) ds\right] = \frac{1}{t}\int_0^t \mathbb{P}(\mu^N_s\in R_N) ds = p\end{equation} so our claim reduces to bounding the variance. \\ \\ For times $t\geq 0$, write $A(t)$ as the event $A(t)=\{\mu^N_t\in R_N\}$; we will compute the covariance of $1_{A(s_1)}$ and $1_{A(s_2)}$, for $0 \leq s_1 \leq s_2$. Observe that \begin{equation} \mathbb{E}\left[1_{A(s_1)}(1_{A(s_2)}-p)\right]=p\left(\mathbb{P}\left(A(s_2)|A(s_1)\right)-p\right).\end{equation} Conditional on $A(s_1)$, the law of $\mu^N_{s_1}$ has a conditional density $h^N_{s_1}\propto 1_{R_N}$ with respect to $\sigma^N$. By  Proposition \ref{prop: relaxation}, conditional on $A(s_1)$, $\mu^N_{s_2}$ has a density $h^N_{s_2}$, and we can bound \begin{equation}|\mathbb{P}(A(s_2)|A(s_1))-p|\leq \left\|h^N_{s_2}-1\right\|_{L^1(\sigma^N)}\leq \left\|h^N_{s_2}-1\right\|_{L^2(\sigma^N)} \leq C(R_N) e^{-\lambda_0(s_2-s_1)} \end{equation} for some constant $C(R_N)$ independent of time. Hence \begin{equation} \mathbb{E}\left[(1_{A(s_1)}-p)(1_{A(s_2)}-p)\right]=p(\mathbb{P}(A(s_2)|A(s_1))-p) \leq p C(R_N) e^{-\lambda_0(s_2-s_1)}.\end{equation}We can now integrate to bound the variance: \begin{equation} \begin{split} \text{Var}\left(\frac{1}{t}\int_0^t 1(\mu^N_s \in R_N) ds\right) & =\frac{2}{t^2}\int_0^t ds_1 \int_{s_1}^t ds_2 \hspace{0.2cm}  \mathbb{E}\left[(1_{A(s_1)}-p)(1_{A(s_2)}-p)\right]  \\[1ex] &  \leq \frac{2pC}{t^2} \int_0^t ds_1\int_{s_1}^\infty  ds_2\hspace{0.2cm} e^{-\lambda_0(s_2-s_1)} \\[1ex] & \leq \frac{2pC}{\lambda_0 t} \rightarrow 0.  \end{split} \end{equation} \end{proof} An immediate corollary is that the long-run deviation must be bounded \emph{below} by the essential supremum of the deviation under the invariant measure: \begin{corollary} Let $(\mu^N_t)_{t\geq 0}$ be a $N$- particle Kac process in equilibrium. Then, almost surely, \begin{equation} \begin{split} \limsup_{t\rightarrow \infty}W(\mu^N_t, \gamma) \geq &\left\|W(\cdot, \gamma)\right\|_{L^\infty(\sigma^N)} \\ =& \esssup_{\sigma^N(d\mu)}\hspace{0.05cm} W(\mu, \gamma). \end{split} \end{equation} \end{corollary} \begin{proof} For ease of notation, write $W^*$ as the essential supremum appearing on the right hand side. For any $\epsilon>0$, let $R_{N, \epsilon}=\{\mu\in \mathcal{S}_N: W(\mu, \gamma)>W^*-\epsilon\}$; it is immediate that $\sigma^N(R_{N, \epsilon})>0$. By the remark in Lemma \ref{lemma: ergodic theorem}, almost surely, $\mu^N_t$ visits $R_{N, \epsilon}$ on an unbounded set of times, and so \begin{equation} \limsup_{t\rightarrow \infty} W(\mu^N_t, \gamma) \geq W^*- \epsilon. \end{equation} The conclusion now follows on taking an intersection over some sequence $\epsilon_n \downarrow 0$. \end{proof} To prove Theorem \ref{thrm: No Uniform Estimate}, it now only remains to show a lower bound on the essential supremum.    \begin{lemma} \label{lemma: construct bad regions} Let $f$ be given by \begin{equation} f(v)=(1+|v|^2)\min\left(\frac{|v|}{\sqrt{N/2}},1\right). \end{equation} Then $f \in \mathcal{A}$, and \begin{equation} \left\|\langle f, \mu-\gamma\rangle\right\|_{L^\infty(\sigma^N)} \geq 1-\frac{C}{\sqrt{N}} \end{equation} for some constant $C=C(d)$. In particular, this is a lower bound for the essential supremum $W^*$, and so for the long-run deviation.\end{lemma} 
\begin{proof} It is easy to see that $f \in \mathcal{A}$. Moreover, the region \begin{equation} \widetilde{R}_{N}=\left\{(v_1,...v_N)\in \mathbb{S}^N: \langle f, \theta_N(v_1,...,v_N) \rangle > 1 \right\}\end{equation} is an open subset of $\mathbb{S}^N$, containing $\left(\sqrt{\frac{N}{2}}e_1,-\sqrt{\frac{N}{2}}e_1,0,..,0\right)$. By positivity of the uniform measure $\gamma^N$ on $\mathbb{S}^N$, it follows that $\gamma^N(\widetilde{R}_{N})>0$. The corresponding region in $\mathcal{S}_N$: \begin{equation} R_{N}=\{\mu^N \in \mathcal{S}_N: \langle f, \mu^N\rangle >1 \} \supset \theta_N(\widetilde{R}_{N}).\end{equation} By definition (\ref{eq: defn of sigmaN}) of $\sigma^N$, we have \begin{equation} \sigma^N(R_{N}) \geq \gamma^N(\widetilde{R}_{N})>0. \end{equation} For all $\mu^N \in R_{N}$, we have \begin{equation} W(\mu^N, \gamma) \geq \langle f, \mu^N-\gamma \rangle \geq 1-N^{-1/2}\langle (1+|v|^2)|v|, \gamma\rangle. \end{equation} Since $R_{N}$ has positive measure, taking $C=\langle (1+|v|^2)|v|, \gamma\rangle$, we can conclude that \begin{equation} W^*\geq 1-\frac{C}{\sqrt{N}}. \end{equation} \end{proof}

\begin{proof}[Proof of Theorem \ref{thrm: No Uniform Estimate}] From the previous two lemmas, we know that for all $N\geq 2$, and for $\sigma^N$- almost all $\mu^N$, \begin{equation} \label{eq: ergodicistation} \mathbb{P}_{\mu^N}\left(\limsup_{t\rightarrow \infty} W(\mu^N_t, \gamma) \geq 1-\frac{C}{\sqrt{N}}\right)=1 \end{equation} where $\mathbb{P}_{\mu^N}$ denotes the law of a Kac process started at $\mu^N$. \\ \\ Let $N\geq 2, k> 2$ and  $a> 1$. The region $R_{\star, N}$ of the labelled sphere such that $\Lambda_{k}(\theta_N(\mathcal{V}))<a$ is an open set; to conclude that it has positive $\sigma^N$- measure, it suffices to show that it is nonempty. \medskip\\ Let $r$ be a rotation by $\frac{2\pi}{N}$ in the plane corresponding to the first two axes $(e_1,e_2)$. Then the data \begin{equation} \mathcal{V}_\star=(e_1, re_1, ..., r^{N-1}e_1) \end{equation} belongs to $\mathbb{S}^N$, and has $\Lambda_{k}(\theta_N(\mathcal{V}_\star))=\frac{1}{N}\sum_{i=1}^N 1^s = 1$. Hence $\mathcal{V}_\star \in R_{\star, N}$ is open and nonempty, so $\gamma^N(R_{\star, N})>0$. The positivity transfers to the corresponding region of $\mathcal{S}_N$: \begin{equation} \sigma^N\left\{\mu^N\in \mathcal{S}_N: \Lambda_{k}(\mu^N)<a\right\}=\gamma^N(R_{N, \star})>0.\end{equation} Hence, for any $N\geq 2$, we can choose an initial datum $\mu^N_0=\mu^N$, with $\Lambda_{k}(\mu^N_0)<a$, such that (\ref{eq: ergodicistation}) holds. Observing that \begin{equation} W(\phi_t(\mu^N_0),\gamma) \leq \|\phi_t(\mu^N_0)-\phi_t(\gamma)\|_{\mathrm{TV}+2} \rightarrow 0\end{equation} it follows that, $\mathbb{P}_{\mu^N}$- almost surely \begin{equation} \limsup_{t\rightarrow \infty} W(\mu^N_t, \gamma) = \limsup_{t\rightarrow \infty} W(\mu^N_t, \phi_t(\mu^N_0)) \geq 1-\frac{C}{\sqrt{N}}. \end{equation} \end{proof}
\begin{remark} \begin{enumerate}[label=\roman{*}).] \item The proof of Lemma \ref{lemma: ergodic theorem} leaves open the possibility that there is a non-empty `exceptional set' of initial data $\mu^N$ where (\ref{eq: ergodicistation}) does not hold. A stronger assertion would be positive Harris recurrence, as defined in \cite{Harris recurrence}, which allows a similar ergodic theorem for \emph{any} initial data $\mu^N$. This is not necessary for our purposes. \item In principle, one could use this compute the typical time scales necessary for these deviations to occur, and sharper estimates may be obtained by using more detailed forms of relaxation, such as the entropic relaxation considered by \cite{Carlen 08}. This is not necessary for our arguments. \end{enumerate}\end{remark}

\section{Proof of Theorem \ref{corr: PW convergence as POC}} \label{sec: proof of POC}

Finally, we show that Theorems \ref{thrm: PW convergence}, \ref{thm: low moment regime} implies the claimed chaoticity estimates in Theorem \ref{corr: PW convergence as POC}. The following proof largely follows that of \cite[Theorem 3.1]{M+M}, using the estimates derived in this paper. As remarked in the introduction, the novelty is the use of the H\"older estimate (\ref{eq: good continuity estimate}) to control the term $\mathcal{T}_3$. 

\medskip In the following proof, we will use estimates from Theorem \ref{thm: low moment regime}, which allow us to minimise the moment conditions required on the initial data. Better results can be obtained using Theorem \ref{thrm: PW convergence} at the cost of requiring a stronger moment estimate, although these still do not obtain optimal rates.
\begin{proof}[Proof of Theorem \ref{corr: PW convergence as POC}] Let $k>2$, and $\epsilon=\epsilon(d,k)>0$ be the resulting exponent from Theorem \ref{thm: low moment regime}. Let $\mu^N_0 \in \mathcal{S}_N$ satisfy $\Lambda_k(\mu^N_0)\le a$. \medskip \\  Recall that we wish to estimate \begin{equation} \frac{\mathcal{W}_{1,l}\left(\Pi_l[\mathcal{P}^N_t(\mu^N_0, \cdot), \phi_t(\mu^N_0)^{\otimes l}\right)}{l}\end{equation}  uniformly in $t\ge 0$ and $l=1,...,N$, and where $\mathcal{W}_{1,l}$ is the Wasserstein$_1$ distance on laws, given by (\ref{eq: definition of script W}). Let $\mathcal{V}^N_t$ be a labelled Kac process, and let $\mu^N_t$ be the associated process of empirical measures. Fixing a test function $f\in B_X^{\otimes l}$, we break up the difference as \begin{equation}\begin{split} \label{eq: decomposition for chaos} &\int_{(\mathbb{R}^d)^N} f(V)\left(\Pi_l[\mathcal{P}^N_t(\mu^N_0, \cdot)]-(\phi_t(\mu^N_0))^{\otimes l}\right)(dV) \\ & \hspace{1cm} =\mathbb{E}_{\mu^N_0}\left[\prod_{j=1}^l f_j(v_j(t))\right]-\prod_{j=1}^l\langle f_j, \phi_t(\mu^N_0)\rangle \\[1ex] & \hspace{1cm} = \mathcal{T}_1+\mathcal{T}_2\end{split} \end{equation} where $\mathbb{E}_{\mu^N_0}$ denotes expectation under the law $\mathcal{P}^N_t(\mu^N_0, \cdot)$, and where the two error terms are\begin{equation} \mathcal{T}_1:= \mathbb{E}_{\mu^N_0}\left[\prod_{j=1}^l f_j(v_j(t))-\prod_{j=1}^l \langle f_j, \mu^N_t\rangle\right]; \end{equation}  \begin{equation} \mathcal{T}_2:=\mathbb{E}_{\mu^N_0}\left[\prod_{j=1}^l \langle f_j, \mu^N_t\rangle-\prod_{j=1}^l\langle f_j, \phi_t(\mu^N_0)\rangle \right].\end{equation} Now, $\mathcal{T}_1$ is a purely combinatorial term, based on the use of empirical measures, and $\mathcal{T}_2$ may be controlled using the pointwise estimates Theorems \ref{thrm: PW convergence}, \ref{thm: low moment regime}. We will indicate how these terms may be controlled for the simple case $l=2$, and use this to show the full, `infinite dimensional' chaos estimate claimed.
\paragraph{Step 1: Estimate on $\mathcal{T}_1$} Since the law $\mathcal{P}^N_t(\mu^N_0, \cdot)$ is symmetric, we may rewrite  \begin{equation} \mathbb{E}_{\mu^N_0}\left[f_1(v_1(t))f_2(v_2(t))\right]=\mathbb{E}_{\mu^N_0}\left[\frac{1}{N(N-1)}\sum_{i\neq j} f_1(v_i(t))f_2(v_j(t))\right] \end{equation} where $N(N-1)$ counts the number of \emph{ordered} pairs of indexes $(i,j)$. Similarly, the second term may be written \begin{equation} \mathbb{E}_{\mu^N_0} \left[\langle f_1, \mu^N_t\rangle\langle f_2, \mu^N_t\rangle\right]=\mathbb{E}_{\mu^N_0}\left[\left(\frac{1}{N}\sum_{i=1}^N f_1(v_i(t))\right)\left(\frac{1}{N}\sum_{j=1}^N f_1(v_j(t))\right)\right].\end{equation}  Comparing the two terms, and using the bound $\|f_j\|_{L^\infty}\le \|f_j\|_X\le 1$ for $j=1,2$, we obtain the estimate \begin{equation} \begin{split} \left|\mathcal{T}_1\right| &\le \sum_{i\neq j} \left|\frac{1}{N(N-1)}-\frac{1}{N^2}\right|+\sum_{i=1}^N \frac{1}{N^2}. \end{split}  \end{equation} Therefore, we have the bound $|\mathcal{T}_1| \le \frac{2}{N}$, uniformly in $f$ and $t$. 
\paragraph{Step 2: Estimate on $\mathcal{T}_2$} For the case $l=2$, we break up the product as \begin{equation} \begin{split} &\prod_{j=1}^2 \langle f_j, \mu^N_t\rangle-\prod_{j=1}^2\langle f_j, \phi_t(\mu^N_0)\rangle \\&\hspace{1cm} = \langle f_1, \mu^N_t-\phi_t(\mu^N_0)\rangle \langle f_2, \mu^N_t\rangle + \langle f_1, \phi_t(\mu^N_0)\rangle \langle f_2, \mu^N_t-\phi_t(\mu^N_0)\rangle. \end{split}\end{equation} In each case, the difference term is dominated by a multiple of the Wasserstein distance $W(\mu^N_t, \phi_t(\mu))$, where $W$ is as in (\ref{eq: definition of W}), and the remaining term is absolutely bounded, by the boundedness of $f_j, j=1,2$. Therefore, we estimate \begin{equation}\label{eq: extensivity}\left|\prod_{j=1}^l \langle f_j, \mu^N_t\rangle-\prod_{j=1}^l\langle f_j, \phi_t(\mu^N_0)\rangle\right| \lesssim W(\mu^N_t, \phi_t(\mu^N_0)). \end{equation} Now, the right-hand side is precisely the term controlled by Theorems \ref{thrm: PW convergence}, \ref{thm: low moment regime}, in the special case $\mu_0=\mu^N_0$. By the choice of $\epsilon$ and $k$ above, we obtain the control \begin{equation} \mathcal{T}_2\lesssim \hspace{0.05cm}\Lambda_k(\mu^N_0)^\frac{1}{2}\hspace{0.05cm}N^{-\epsilon} \lesssim \hspace{0.05cm}a\hspace{0.05cm}N^{-\epsilon} \end{equation} for some explicit $\epsilon=\epsilon(d,k)>0$.\medskip \\ We also remark here that this implication, given Theorems \ref{thrm: PW convergence}, \ref{thm: low moment regime} is immediate. However, attempting to reverse this implication, and deduce a theorem similar to \ref{thrm: PW convergence} from a control of $\mathcal{T}_2$, requires moving the supremum over test functions $f$ \emph{inside} the expectation. This corresponds to the most technical step in our proof (Lemmas \ref{thrm: pointwise martingale control}, \ref{thrm: local uniform martingale control}). Therefore, while it may be possible to deduce a version Theorem \ref{thrm: PW convergence} from the control of $\mathcal{T}_2$ given by \cite{M+M}, this would scarcely be less technical than the proof given, and would not lead to a proof of Theorem \ref{thrm: Main Local Uniform Estimate}.
\paragraph{Step 3: Deduction of Infinite-Dimensional Chaos} Combining the two estimates for the case $l=2$ above, we deduce that there exists $\epsilon=\epsilon(d,k)>0$ such that\begin{equation} \sup_{t\ge 0} \hspace{0.05cm} \mathcal{W}_{1,2}\left(\Pi_2\left[\mathcal{P}^N_t(\mu^N_0, \cdot)\right], \phi_t(\mu^N_0)^{\otimes l}\right) \lesssim a\hspace{0.05cm}N^{-\epsilon}.\end{equation} To deduce the full statement, we appeal to the following result from \cite{Hauray Mischler}, which may also be found in \cite[Theorem 2.1]{Mischler}. For any probability measure $\mu$ on $\mathbb{R}^d$, and any symmetric distribution $\mathcal{L}^N$ on $(\mathbb{R}^d)^N$, we may estimate \begin{equation} \max_{l\le N}\hspace{0.05cm}\frac{\mathcal{W}_{1,l}\left(\Pi_l[\mathcal{L}^N], \mu^{\otimes l}\right)}{l} \le C\left(\mathcal{W}_{1,2}\left(\Pi_2[\mathcal{L}^N], \mu^{\otimes 2}\right)^{\alpha_1}+N^{-\alpha_2}\right)\end{equation} for some explicit constants $C, \alpha_1, \alpha_2>0$ depending on the dimension $d$. The claimed result (\ref{eq: CPOC}) now follows. \medskip \\ We now turn to the two consequences claimed as a result. \paragraph{i). Chaotic Case} Let $\mu_0 \in \mathcal{S}$ have an $k^\mathrm{th}$ moment $\Lambda_k(\mu_0)\le a$, and construct  $\mathcal{V}^N_0=(v_1(0),...,v_N(0))$ be as described in the statement of the theorem with associated empirical measure $\mu^N_0$. It is straightforward to show that this construction preserves moments up to a constant: that is, $\mathbb{E}(\Lambda_k(\mu^N_0))\lesssim a.$ \medskip \\ For a fixed test function $f\in B_X^{\otimes l}$, we return to the decomposition (\ref{eq: decomposition for chaos}). For this case, where $\mu^N_0\neq \mu_0$, we have a third error term: \begin{equation} \int_{(\mathbb{R}^d)^N} f(V)(\Pi_l[\mathcal{LV}^N_t]-(\phi_t(\mu_0))^{\otimes l})(dV)=\mathcal{T}_1+\mathcal{T}+\mathcal{T}_3.\end{equation} Here, $\mathcal{T}_1$ and $\mathcal{T}_2$ are as above, replacing $\mathbb{E}_{\mu^N_0}$ by the full expectation $\mathbb{E}$, and $\mathcal{T}_3$ is an additional error term, from approximating $\mu_0$ by $\mu^N_0$: \begin{equation} \mathcal{T}_3:=\mathbb{E}\left[\prod_{j=1}^l \langle f_j, \phi_t(\mu^N_0)\rangle-\prod_{j=1}^l\langle f_j, \phi_t(\mu_0)\rangle\right].\end{equation} As in the case above, we consider first the case $l=2$. The first two terms $\mathcal{T}_1, \mathcal{T}_2$ may be estimated as above, by conditioning on $(v_1(0),...,v_N(0))$ to conclude that \begin{equation} \mathcal{T}_1+\mathcal{T}_2 \lesssim aN^{-\epsilon}\end{equation} for some $\epsilon>0$, uniformly in $f\in B_X^{\otimes l}$ and $t\ge 0$. \medskip \\ Arguing as in (\ref{eq: extensivity}), we bound \begin{equation} \mathcal{T}_3 \lesssim\mathbb{E}W(\phi_t(\mu^N_0), \phi_t(\mu_0)).\end{equation}We estimate this term using the contunity estimate Theorem \ref{thrm: W-W continuity of phit}. Let $k'\in (2,k)$, and let $\zeta>0$ be the resulting exponent using Theorem \ref{thrm: W-W continuity of phit}; by making $\zeta$ smaller if necessary, we assume that \begin{equation} \frac{\zeta k}{k-k'} \le 1. \end{equation}  From Theorem \ref{thrm: W-W continuity of phit}, we have the estimate \begin{equation} \sup_{t\ge 0} W(\phi_t(\mu^N_0),\phi_t(\mu_0))\lesssim\Lambda_{k'}(\mu^N_0,\mu_0)W(\mu^N_0,\mu_0)^\zeta\end{equation} and we use H\"older's inequality to obtain, uniformly in $t\ge 0$, \begin{equation} \begin{split}\mathbb{E}\left[W(\phi_t(\mu^N_0)\phi_t(\mu_0))\right] &\lesssim \mathbb{E}\left[\Lambda_k(\mu^N_0)\right]^{k'/k}\mathbb{E}\left[W(\mu^N_0,\mu_0)^\frac{\zeta k}{k-k'}\right]^\frac{k-k'}{k} \\[1ex] & \lesssim a^{k'/k} \hspace{0.1cm}\mathbb{E}\left[W(\mu^N_0,\mu_0)\right]^\zeta. \end{split} \end{equation} From \cite[Proposition 9.2]{ACE}, there is a constant $\beta=\beta(d,k)>0$ such that $ \mathbb{E} W(\mu^N_0,\mu_0)\lesssim N^{-\beta}$, so we obtain \begin{equation} \mathbb{E}\left[W(\phi_t(\mu^N_0),\phi_t(\mu_0))\right] \lesssim aN^{-\beta\zeta}.\end{equation} Combining, and since all of our estimates are uniform in $f$ and $t$, we have shown that \begin{equation} \mathcal{W}_{1,2}\left(\Pi_2[\mathcal{LV}^N_t], \phi_t(\mu_0)^{\otimes 2}\right) \lesssim aN^{-\alpha}\end{equation} for some $\alpha=\alpha(d,k)>0$. The improvement to infinite-dimensional chaos is exactly as above.  
\paragraph{ii). General Case} The general case follows from the first case, by taking expectations over the initial data $\mu^N_0$. Indeed, for all $l\le N$, all $f\in B_X^{\otimes l}$ and $t\ge 0$, and for any initial data $(v_1(0),...v_N(0))$ with associated measure $\mu^N_0$, we have the bound \begin{equation} \frac{1}{l}\hspace{0.1cm}\mathbb{E}_{\mu^N_0}\left[f_1(v_1(t))...f_l(v_l(t))-\prod_{j=1}^l \langle f_j, \phi_t(\mu^N_0)\rangle\right] \lesssim \Lambda_k(\mu^N_0)N^{-\epsilon}.\end{equation} Taking expectation over the random initial data $(v_1(0),...,v_N(0))$ produces a full expectation on the left-hand side, and by definition of $\mathcal{L}^l_t$ in (\ref{eq: defn of ll}), \begin{equation}\mathbb{E}\left[\prod_{j=1}^l \langle f_j, \phi_t(\mu^N_0)\rangle\right]=\int_{(\mathbb{R}^d)^l} f(V)\hspace{0.1cm}\mathcal{L}^l_t(dV). \end{equation} Optimising over $f\in B_X^{\otimes l}$, $l\le N$ and $t\ge 0$ proves the claimed result.      \end{proof}

\appendix

\section{Calculus of Martingales} \label{sec: Calculus of mgs}

We also review some basic facts and inequalities for martingales associated to the Kac process. All of these facts are true for general Markov chains, see \cite{D&N}.

Let $\mu^N_t$ be a Kac process, and write $m^N$, $\overline{m}^N$ for the jump measure and compensator defined in Section \ref{sec: interpolation decomposition}. Then, for any bounded and measurable $F^N: [0,T]\times \mathcal{S}_N  \rightarrow \mathbb{R}$, the process \begin{equation} \label{eq: M is for martingale} \mathcal{M}^N_t= \int_{(0,t]\times \mathcal{S}_N} \left\{F^N_s(\mu^N)-F^N_s(\mu^N_{s-})\right\} (m^N-\overline{m}^N)(ds, d\mu^N) , \hspace{0.5cm} 0\leq t\leq T\end{equation} is a martingale for the natural filtration $(\mathcal{F}^N_t)_{t\geq 0}$ of the process. We have the $L^2$ control \begin{equation} \label{eq: elltwo martingale control} \left\| \mathcal{M}^N_t\right\|^2_2 = \mathbb{E}\left\{\int_{(0,t]\times \mathcal{S}_N} \left\{F^N_s(\mu^N)-F^N_s(\mu^N_{s-})\right\}^2 \overline{m}^N(ds, d\mu^N)\right\}. \end{equation} We will also use another special case of It\^{o}'s isometry for the measure $m^N-\overline{m}^N$ for a similar form of martingale. If $F^N$ is bounded and measurable on $[0,T]\times \mathcal{S}_N$, then for $t\leq T$, \begin{equation} \label{eq: QV of M}\left\|\int_0^t F^N_s(\mu^N_{s-})(m^N-\overline{m}^N)(ds, \mathcal{S}_N)\right\|^2_2 = \mathbb{E}\left\{\int_0^t F^N_s(\mu^N_s)^2\hspace{0.1cm} \overline{m}^N(ds, \mathcal{S}_N)\right\}.\end{equation} For the local uniform case, Theorem \ref{thrm: Main Local Uniform Estimate}, it will be necessary to control martingales of the form (\ref{eq: M is for martingale}) in general $L^p$ spaces, rather than simply $L^2$. Since $\mathcal{M}^N$ of this form are finite variation martingales, the quadratic variation is given by \begin{equation} \left[\mathcal{M}^N\right]_t=\int_{(0,t]\times \mathcal{S}_N} \left\{F^N_s(\mu^N)-F^N_s(\mu^N_{s-})\right\}^2 m^N(ds, d\mu^N), \hspace{0.5cm} 0\leq t\leq T.\end{equation} Our analysis in $L^p$ is based on Burkholder's inequality for c\`{a}dl\`{a}g martingales, which we state here for the class of martingales constructed above:  \begin{lemma} \label{lemma:Burkholder} Suppose that $(\mathcal{M}^N_t)_{t=0}^T$ is the process given by (\ref{eq: M is for martingale}), and let $p\geq 2$. Then there exists a constant $C=C(p)<\infty$ such that for all $t\leq T$, we have the $L^p$ control \begin{equation} \left\|\hspace{0.1cm}\sup_{s\leq t} \left|\mathcal{M}^N_s\right|\hspace{0.1cm}\right\|^p_p \leq C(p)  \mathbb{E}\left[\left(\int_0^t \left\{F^N_s(\mu^N)-F^N_s(\mu^N_{s-})\right\}^2 m^N(ds, d\mu^N) \right)^{p/2} \right].\end{equation} \end{lemma}

\end{document}